\documentclass[a4paper,11pt]{article}
\usepackage{pdfpages}

\usepackage[margin=1in]{geometry}  
\usepackage{amsmath}               
\usepackage{amsfonts}              
\usepackage{amsthm}      
\usepackage{amssymb}
\usepackage{bbm}

\usepackage[utf8]{inputenc}
\usepackage{mathtools}
\usepackage{cite}

\usepackage{color}   
\usepackage{hyperref}
\hypersetup{
    colorlinks=true, 
    linktoc=all,     
    linkcolor=black,  
}

\def\acts{\curvearrowright}
\DeclarePairedDelimiter\abs{\lvert}{\rvert}%
\DeclarePairedDelimiter\norm{\lVert}{\rVert}%

\makeatletter
\let\oldabs\abs
\def\abs{\@ifstar{\oldabs}{\oldabs*}}
\let\oldnorm\norm
\def\norm{\@ifstar{\oldnorm}{\oldnorm*}}

\makeatother

\theoremstyle{definition}
\newtheorem{thm}{Theorem}[section]
\newtheorem*{thm*}{Theorem}
\newtheorem{lem}[thm]{Lemma}
\newtheorem{prop}[thm]{Proposition}
\newtheorem*{claim}{Claim}
\newtheorem{cor}[thm]{Corollary}

\newtheorem{defn}[thm]{Definition}
\newtheorem{example}{Example}

\newtheorem{remark}[thm]{Remark}

\newtheorem{question}{Question}

\DeclareMathOperator{\cost}{cost}

\DeclareMathOperator{\PSL}{PSL}
\DeclareMathOperator{\SL}{SL}

\DeclareMathOperator{\Isom}{Isom}
\DeclareMathOperator{\stab}{stab}
\DeclareMathOperator{\support}{support}

\DeclareMathOperator{\Cay}{Cay}

\let\phi\varphi

\let\empt\varnothing

\newcommand{\EE}{\mathbb{E}}      

\newcommand{\RR}{\mathbb{R}}      

\newcommand{\e}{\varepsilon} 
\newcommand{\ZZ}{\mathbb{Z}}      
\newcommand{\CC}{\mathbb{C}}      
\newcommand{\PP}{\mathbb{P}}      
\newcommand{\MM}{\mathbb{M}}      
\newcommand{\Maper}{\mathbb{M}^\text{aper}}

\newcommand{\NN}{\mathbb{N}}      
\newcommand{\HH}{\mathbb{H}}      
\newcommand{\dprok}{d_{\text{prok}}}      

\newcommand{\1}{\mathbbm{1}}

\newcommand\restr[2]{{
  \left.\kern-\nulldelimiterspace 
  #1 
  \vphantom{\big|} 
  \right|_{#2} 
  }}

\usepackage{color}
\usepackage{mathrsfs}
\newcommand{\Rel}{\mathcal{R}} 
\newcommand{\MMo}{{\mathbb{M}_0}}
\newcommand{\Marrow}{\overrightarrow{\mathbb{M}_0}} 
\newcommand{\muarrow}{\overrightarrow{\mu_0}}

\newcommand{\PPP}{\mathscr{P}}

\DeclareMathOperator{\graph}{Graph}
\DeclareMathOperator{\covol}{covol}
\DeclareMathOperator{\intensity}{int}
\DeclareMathOperator{\diam}{diam}
\DeclareMathOperator{\Var}{Var}
\DeclareMathOperator{\RG}{RG}

\begin{document}

\title{Point processes, cost, and the growth of rank in locally compact groups\footnote{This work was partially supported by ERC Consolidator Grant 648017 and the KKP NKFI-139502 Grant.}}

\author{
  Mikl\'{o}s Ab\'{e}rt\thanks{Renyi Institute of Mathematics, Realtanoda Utca 13-15, 1053 Budapest, Hungary.\newline
  \textit{E-mail:} \texttt{abert@renyi.hu}}
  \and
  Sam Mellick\thanks{ENS de Lyon site Monod, UMPA, 5669 CNRS, 46 allée d’Italie, 69364 Lyon Cedex 07, FRANCE\newline
  \textit{E-mail:} \texttt{samuel.mellick@ens-lyon.fr}}
}

\maketitle

\abstract{Let $G$ be a locally compact, second countable, unimodular group that is nondiscrete and noncompact. We explore the ergodic theory of invariant point processes
on $G$. Our first result shows that every free probability measure preserving (pmp) action of $G$ can be realized by an invariant point process.

We then analyze the cost of pmp actions of $G$ using this language. We show that
among free pmp actions, the cost is maximal on the Poisson processes. This follows from showing that every free point process
weakly factors onto any Poisson process and that the cost is monotone for weak
factors, up to some restrictions.
We apply this to show that $G\times \mathbb{Z}$ has fixed price $1$, solving a
problem of Carderi. 

We also show that when $G$ is a semisimple real Lie group,
the rank gradient of any Farber sequence of lattices in $G$ is dominated by the cost of
the Poisson process of $G$. The same holds for the symmetric space $X$ of $G$. This, in particular, implies that if the cost of
the Poisson process of the hyperbolic $3$-space $\HH^3$ vanishes, then the ratio of the
Heegaard genus and the rank of a hyperbolic $3$-manifold tends to infinity
over arbitrary expander Farber sequences, in particular, the ratio can get arbitrarily large. On the other hand, if the cost of the Poisson process on $\HH^3$ does not vanish, it solves the cost versus $L^2$ Betti problem of Gaboriau for countable equivalence relations.}

\section{Introduction}

Let $G$ be a locally compact, second countable, unimodular group that is
nondiscrete and noncompact, endowed with a Haar measure $\lambda$. We think of
$\lambda$ as an inherent parameter of $G$, as all the notions trivally scale
with $\lambda$. Throughout the paper, we will make these assumptions on $G$ except when stated otherwise. 

A \emph{point process} $\Pi$ on $G$ is a random closed and discrete subset $\Pi$ of $G$. More precisely, it is a random variable taking values in the \emph{configuration space} of $G$:
\[
    \MM = \MM(G) = \{ \omega \subseteq G \mid \omega \text{ is closed and discrete} \}.
\]
When the law of $\Pi$ is invariant under the left $G$-action, we call $\Pi$ an \emph{invariant point process}. 
We do not assume the reader has any knowledge of point process theory and have made the paper as self-contained as possible. 

Invariant point processes are examples of probability measure preserving (pmp) actions. Recall that a pmp action is \emph{essentially free} (or simply \emph{free} for short) if the stabiliser of almost every point in the action space is trivial. In the particular case of point processes, this means that the set of points will almost surely have no symmetries. Our first theorem shows that actually \emph{every} free pmp action can be realised this way.

\begin{thm}
\label{minden}Every free pmp action of $G$ is isomorphic to a point process on $G$.
\end{thm}

Note that freeness is a necessary condition here as can be seen from the action of $\RR^2$ on $\RR^2 / (\ZZ \times \RR)$. This action is however a point process on a homogeneous space of $\RR^2$. 

The proof of Theorem \ref{minden} exhibits an
analogy between point processes of locally compact groups and the symbolic
dynamics of countable groups. For a pmp action of a countable group $\Gamma$,
every Borel partition of the underlying space gives rise to an invariant
random coloring of $\Gamma$ by considering the orbit of a random point of the
underlying space. Similarly, every cross section of a free pmp action of $G$,
when considering its intersection with the $G$-orbit of a random point, will
become a point process on $G$. So point processes serve as \emph{stochastic
visualizations} of pmp actions of locally compact groups, just like invariant
random colorings do for countable groups. This paper aims to show that this
visualization leads to new and meaningful results. 

The correspondence above and the classical theorem of Forrest \cite{MR417388} on the existence of cross sections (see also \cite{MR3335405} and Section 3.B of \cite{kechris2019theory}) immediately yields that every free pmp action factors onto an invariant point process. The factor map can be upgraded to an isomorphism by using a \emph{marked} point process. These are random discrete subsets where each point carries a mark from some mark space (for example, a finite set of colours). Then Theorem \ref{minden} is proved by showing that every marked point process is isomorphic to an unmarked one, by ``spatially encoding'' the marks.

We now introduce the cost of a point process $\Pi$ on $G$. A \emph{factor graph} $\mathscr{G}$ of $\Pi$ is an equivariantly and measurably defined graph $\mathscr{G}(\Pi)$ whose vertex set is $\Pi$. For example, one can define the distance graph for $r > 0$ to be the set of pairs $g, h \in \Pi$ with $d(g, h) < r$, where $d$ is a left-invariant metric on $G$. Informally speaking, the \emph{cost} of $\Pi$ is defined as the infimum of the average degrees over connected factor graphs of $\Pi$, suitably normalised to be an isomorphism invariant. For precise definitions see Section \ref{cost}. We then define cost for free pmp actions of $G$ via Theorem \ref{minden}, which is well-defined since cost is an isomorphism invariant.

The cost of pmp actions of countable groups has been an active subject in the last twenty years, see Gaboriau's paper \cite{gaboriau2016around} and the survey paper \cite{furman2009survey} for the literature. It has been known in the community that cross sections naturally allow one to extend the notion of cost to free pmp actions of locally compact groups, but due to the lack of results, the notion stayed dormant. The first explicit appearance of the definition can be found in a recent paper of Carderi \cite{carderi2018asymptotic}. The definition we work with is essentially equivalent to his, but we develop it intrinsically as a point process theoretic notion.

One of the most important families of processes on a discrete group is Bernoulli percolations $\text{Ber}(p)$. The natural analogue of this family for non discrete groups is the \emph{Poisson point process} of intensity $t > 0$. Here the \emph{intensity} of an invariant point process is the expected number of points which fall in a set of unit volume. This quantity can be shown to be independent of the choice of set. An explicit description of Poisson point processes will be given later, but one should know that these processes are ``completely independent''. 

\begin{thm}\label{Poissmax}
Poisson point processes have maximal cost among all free pmp actions on $G$. In particular, the cost of a Poisson point process does not depend on its intensity. 
\end{thm}

We denote the cost of a Poisson point process on $G$ by $c_P(G)$. The above result can be looked at as a locally compact analogue of a result of Ab\'{e}rt and Weiss \cite{abert2013bernoulli} where they show that for a countable group, Bernoulli actions have maximal cost among all pmp actions.

A central open problem in cost theory is the Fixed Price problem of Gaboriau, that asks whether all free pmp actions of a countable group have the same cost. This is also open in the locally compact setting. 

\begin{question}
Is it true that all free point processes on $G$ have the same cost? 
\end{question}

Gaboriau \cite{gaboriau2002invariants} asks if for a countable pmp equivalence relation, the cost of the relation equals its first $L^{2}$ Betti number $\beta_{1}(G)$ plus 1. Note that an affirmative answer for this would imply an affirmative answer to Question 1, as well, using the cross section correspondence.  

Since the cost of any free process is at least one, a viable way to prove that a group has fixed price one is by showing that the Poisson point process admits connected factor graphs with average degree $1 + \e$ for all $\e > 0$. We succeed in this for the first nontrivial case, answering a question of Carderi \cite{carderi2018asymptotic}:

\begin{thm}
\label{zszorzat}Every free pmp action of $G\times\mathbb{Z}$ has cost one if $G$ is compactly generated.
\end{thm}

Our proof is truly a stochastic proof in nature as it essentially uses some special properties of Poisson point processes. 

In countable cost theory, it remains an open question if the direct product $\Gamma \times \Delta$ of two infinite countable groups $\Gamma$ and $\Delta$ has fixed price one. It is known to hold if one of the groups contains a fixed price one subgroup. When trying to generalize Theorem \ref{zszorzat} to arbitrary products, we seem to hit a somewhat similar barrier.

\begin{question}
Let $G$ and $H$ be compactly generated but noncompact groups. Does $G \times H$ have fixed price one?
\end{question}

Another application of Theorem \ref{Poissmax} concerns the growth of the
minimal number of generators (the rank gradient) for a sequence of lattices in
semisimple real Lie groups. Recall that a discrete subgroup $\Gamma\leq G$ is
a \emph{lattice} if it has finite covolume in $G$. Let $d(\Gamma)$ denote the
minimal number of generators (also known as the \emph{rank}) of $\Gamma$. When $G$ is a semisimple Lie
group, $d(\Gamma)$ is finite and by a theorem of Gelander \cite{MR2863908}, we have
\[
\frac{d(\Gamma)-1}{\mathrm{vol}(G/\Gamma)}\leq C
\]
for some constant $C$ only depending on $G$. 

A sequence of lattices $\Gamma_{n}$ in $G$ is \emph{Farber}, if $G/\Gamma_{n}$
approximates $G$ in the Benjamini-Schramm topology, or, equivalently, if the
corresponding invariant random subgroups weakly converge to the trivial
subgroup. 

\begin{thm}\label{carderiextension}
Let $G$ be a semisimple real Lie group and let $\Gamma_{n}$ be a Farber sequence of lattices in $G$. Then
\[
\limsup_{n\to\infty}\frac{d(\Gamma_n)-1}{\mathrm{vol}(G/\Gamma_n)}
\leq\mathrm{c}_{\mathrm{P}}(G)-1.
\]
\end{thm}

In particular, if the Poisson point process has cost $1$ then the rank grows
sublinearly in the covolume, for Farber sequences of lattices. This correspondence connects computing the cost of the Poisson process to some
exciting open problems that have been investigated extensively. Theorem \ref{carderiextension}
extends Carderi's result \cite{carderi2018asymptotic} who proved it for uniformly discrete (in
particular, cocompact) Farber sequences. Here a sequence of lattices is \emph{uniformly discrete} if
there exists $C>0$ such that the infimal injectivity radius is bounded below
by $C$ for $G/\Gamma_{n}$. 

\begin{question}
Let $G$ be a semisimple real Lie group that is not a compact extension of $\SL_{2}(\RR)$. Is
$\mathrm{c}_{\mathrm{P}}(G)=1$? 
\end{question}

Note that by work of Conley, Gaboriau, Marks, and Tucker-Drob the group $\SL_{2}(R)$ is treeable and has fixed price greater than $1$. 

We now showcase three concrete cases for a semisimple Lie group where computing the cost of the Poisson process would solve known problems of a different nature. Note that it is natural to ask about the cost of the Poisson process on the symmetric space $X$ of $G$ rather than on the group $G$ itself. As we show in Theorem \ref{symmetricspacecostmax} these two invariants are equal. \medskip 

\textbf{Case }$G=\SL_{2}(\CC)$ and $X=\HH^3$\textbf{.} If $\mathrm{c}_{\mathrm{P}}(G)>1$, then
we get free point processes in $G$ with different cost. Moreover, we also get
a countable equivalence relation whose first $L^2$ Betti number is not equal to
its cost-1, answering a question of Gaboriau \cite{gaboriau2002invariants}. If, on the other hand,
$\mathrm{c}_{\mathrm{P}}(G)=1$, then we get that the Heegaard genus divided by
the rank of the fundamental group of a (compact) hyperbolic $3$-manifold can get arbitrarily large. In
fact, we yield this for \emph{any} expander Farber sequence of hyperbolic
$3$-manifolds, which is understood as the typical behavior. Indeed, by the work of Lackenby \cite{lackenby2002heegaard} for expander sequences, the Heegaard genus grows linearly, while using our work, the rank would grow sublinearly in the volume. Note that it is a
longstanding open problem whether this ratio is absolutely bounded over all $3$-manifolds, and in fact it was only proved recently in the deep paper of Li \cite{li2013rank} that for compact hyperbolic $3$-manifolds, the Heegaard genus can \emph{differ} from the rank. \medskip

\textbf{Case when }$G$\textbf{ has higher rank.} For these Lie groups, Fraczyk
recently proved in a beautiful paper \cite{fraczyk2018growth} that the growth of the first mod $2$
homology vanishes for Farber sequences in $G$. Surprisingly, his method does
not seem to carry over to odd primes, so for primes other than $2$, this is
still an open problem. As the rank is an upper bound for the first mod $p$
homology of a discrete group, proving $\mathrm{c}_{\mathrm{P}}(G)=1$ would
settle this problem. 

By a standard induction argument,  proving $\mathrm{c}_{\mathrm{P}}(G)=1$ would
show that any lattice in $G$ has fixed price $1$, a problem of Gaboriau that
is still open for cocompact lattices. \medskip 

\textbf{Case when }$G$\textbf{ has higher rank and property (T).} Using \cite{MR3664810}, for semisimple Lie groups with (T) one can actually omit the Farber condition. 

\begin{cor}
Let $G$ be a higher rank semisimple real Lie group with property (T) and let $\Gamma_{n}$
be any sequence of lattices in $G$ with $\mathrm{vol}(G/\Gamma_n)\rightarrow
\infty$. Then
\[
\lim\sup_{n\rightarrow\infty}\frac{d(\Gamma_n)-1}{\mathrm{vol}(G/\Gamma_n)}%
\leq\mathrm{c}_{\mathrm{P}}(G)-1\text{.}%
\]

\end{cor}

In particular, if $\mathrm{c}_{\mathrm{P}}(G)=1$, then we get a \emph{totally
uniform} vanishing theorem for the growth of rank for lattices in these
groups, including $SL(d,R)$ ($d\geq3$). It is shown in \cite{abert2017rank} that any Farber
sequence in SL(3,Z) has vanishing rank gradient, but the uniform version is
wide open. 

Note that in their very recent paper Lubotzky and Slutsky \cite{lubotzky2021asymptotic} showed that in the above situation, every sequence of non-uniform lattices will have rank gradient $0$. Their proof uses deep classical results on non-uniform lattices, like arithmeticity and the Congruence Subgroup Property but in turn gives a much stronger upper bound for the number of generators than what we ask, logarithmic in the covolume. In most cases, they can even improve this with a loglog factor. Their methods do not seem to readily generalize to co-compact lattices. Our purely geometric approach may have the potential to be applied more widely but the payoff is that, being a limiting argument, it is not expected to yield such explicit estimates. 

The proof of Theorem \ref{Poissmax} uses the stochastic visualisation method to show that every free action is ``sufficiently rich'' in randomness to ``simulate'' the Poisson point process. In particular, connected factor graphs of the Poisson point process can be transferred to an arbitrary free process in a way that can at worst \emph{decrease} the average degree. Simulation here refers to \emph{weak factoring}, a notion we introduce that is inspired by weak containment of actions, see the survey of Kechris and Burton \cite{kechris2019theory}.

For invariant point processes $\Pi$ and $\Upsilon$, we say that $\Upsilon$ is a \emph{factor} of $\Pi$ if
there exists a $G$-equivariant Borel map $\Phi: \MM \to \MM$ such that $\Phi(\Pi) = \Upsilon$. We say that $\Upsilon$ is a \emph{weak factor} of $\Pi$ or $\Pi$ \emph{weakly factors} onto $\Upsilon$ if there exist factor maps $\Phi_1, \Phi_2, \ldots$ of $\Pi$ such that $\Phi_n(\Pi)$ converges weakly to $\Upsilon$.

\begin{thm}
\label{Poissweak}Let $\Pi$ be a free point process on $G$. Then $\Pi$ weakly
factors onto the Poisson point process of intensity $t$, for all $t$. 
\end{thm}

In particular, Poisson processes on $G$ of different intensities weakly factor onto each other. More is known in the amenable case: Ornstein and Weiss showed \cite{ornstein1987entropy} that for a large class of amenable groups, the Poisson point processes of different intensity are in fact \emph{isomorphic} as actions (see \cite{soo2019finitary} for an alternative construction on $\RR^n$ with additional properties). 

The proof of Theorem \ref{Poissweak} revolves around IID-marked processes. Let $[0,1]^\Pi$ denote the random $[0,1]$-marked subset of $G$ whose underlying set is $\Pi$ and has independent and identically distributed $\texttt{Unif}[0,1]$ random variables. We call this the \emph{IID of $\Pi$}. Once this definition and that of the Poisson point process is understood, one can readily see that the IID of \emph{any} process factors onto the Poisson point process. We then prove:

\begin{thm}
\label{putiid}Let $\Pi$ be a free point process on $G$. Then $\Pi$ weakly factors onto $[0,1]^\Pi$, its own IID.
\end{thm}

Somewhat surprisingly, it is not entirely trivial to show that weak factoring is a \emph{transitive} notion, but we are able to prove it. Thus in particular, Theorem \ref{putiid} implies that free point processes weakly factor onto the Poisson point process.

We next investigate how cost behaves with respect to factor maps. It is easy to see that it can only \emph{increase} under a factor map: if $\Pi$ factors onto $\Upsilon$, then $\cost(\Pi) \leq \Upsilon$. In particular, this shows that cost is an \emph{isomorphism invariant} of actions. This monotonicity of cost under factor maps can be pushed further:

\begin{thm}
Suppose $\Pi$ weakly factors onto $\Upsilon$, as witnessed by a sequence of factor maps $\Phi_n(\Pi)$ weakly converging to $\Upsilon$. Under appropriate tightness conditions on $\Pi, \Upsilon$, and the sequence $\Phi_n$, we have $\cost(\Pi) \leq \cost(\Upsilon)$.
\end{thm}

See Section \ref{certainfactors} for a precise statement. This cost monotonicity theorem, limited as it is, is powerful enough to prove that the Poisson point process has maximal cost amongst all free processes. \medskip 

\noindent {\bf Acknowledgements.} The authors wish to thank the anonymous referee for a very thorough and helpful report. The second author thanks Mikolaj Fraczyk and Alessandro Carderi for helpful discussions. 

\medskip 
The paper is structured as follows. 
In Section 1, we give the basic definitions and notations of point processes for those who have never encountered them before, and describe the most important examples of point processes for our work.
In Section 2, we introduce the \emph{Palm measure} of a point process and the rerooting groupoid.
In Section 3, we define the \emph{cost} of an invariant point process and prove basic properties of it.
In Section 4, we define \emph{weak factoring} of point processes and prove that (in certain circumstances) cost is \emph{monotone} with respect to weak factoring. We use this to show that the Poisson has maximal cost amongst all free processes.
In Section 5, we use the fact that the Poisson has maximal cost to give the first nontrivial examples of nondiscrete groups with fixed price.
In Section 6, we connect the rank gradient of sequences of lattices in a group with the cost of the Poisson point process on said group. 
In Section 7, we discuss the modifications required to extend the above theory to invariant point processes on symmetric spaces.
In the appendix, we include a summary of necessary technical facts from point process theory with references for proofs. No originality is claimed for this material.


\section{Point processes and factors of interest}

Let $(Z,d)$ denote a complete and separable metric space (a csms). A \emph{point process on $Z$} is a random discrete subset of $Z$. We will also study random discrete subsets of $Z$ that are \emph{marked} by elements of an additional csms $\Xi$. Typically $\Xi$ will be a finite set that we think of as colours.

\begin{defn}
	The \emph{configuration space} of $Z$ is
    	\[
        	\MM(Z) = \{ \omega \subset Z \mid \omega \text{ is locally finite} \},
        \]
    and the \emph{$\Xi$-marked configuration space} of $Z$ is
    	\[
        	\Xi^\MM(Z) = \{ \omega \subset Z \times \Xi \mid \omega \text{ is discrete, and if } (g, \xi) \in \omega \text{ and } (g, \xi') \in \omega \text{ then } \xi = \xi' \}.
        \]
\end{defn}

Note that $\Xi^\MM(Z) \subset \MM(Z \times \Xi)$. We think of a $\Xi$-marked configuration $\omega \in \Xi^\MM(Z)$ as a locally finite subset of $Z$ with labels on each of the points (whereas a typical element of $\MM(Z \times \Xi)$ is a locally finite subset where each point has possibly multiple marks). 

If $\omega \in \Xi^\MM(Z)$ is a marked configuration, then we will write $\omega_z$ for the unique element of $\Xi$ such that $(z, \omega_z) \in \omega$. 

The Borel structure on configuration spaces is exactly such that the following \emph{point counting functions} are measurable. Let $U \subseteq Z$ be a Borel set. It induces a function $N_U : \MM(Z) \to \NN_0 \cup \{ \infty \}$ given by
\[
	N_U(\omega) = \abs{\omega \cap U}.
\]

We will primarily be interested in point processes defined on locally compact and second countable (lcsc) groups $G$. Such groups admit a unique (up to scaling) left-invariant Haar measure $\lambda$, we fix such a choice. We will further assume that $G$ is \emph{unimodular}, although it is not strictly necessary for every argument in the paper. Recall:

\begin{thm}[Struble's theorem, see Theorem 2.B.4 of \cite{MR3561300}]
Let $G$ be a locally compact topological group. Then $G$ is second countable \emph{if and only if} it admits a proper\footnote{Recall that a metric is \emph{proper} if closed balls are compact.} left-invariant metric.
\end{thm}

Such a metric is unique up to coarse equivalence (bilipschitz if the group is compactly generated). We fix $d$ to be any such metric. 

We mostly consider the configuration space of a fixed group $G$. So out of notational convenience let us write $\MM = \MM(G)$ and $\Xi^\MM = \Xi^\MM(G)$. The latter here is an abuse of notation: formally $\Xi^\MM$ ought to denote the set of functions from $\MM$ to $\Xi$, but instead we are using it to denote the set of functions from \emph{elements} of $\MM$ to $\Xi$.

Note that the marked and unmarked configuration spaces of $G$ are Borel $G$-spaces. To spell this out, $G \acts \MM$ by $g \cdot \omega = g\omega$ and $G \acts \Xi^\MM$ by
\[
    g \cdot \omega = \{(gx, \xi) \in G \times \Xi \mid (g, \xi) \in \omega \}.
\]

\begin{defn}
	A \emph{point process} on $G$ is a $\MM(G)$-valued random variable $\Pi : (\Omega, \PP) \to \MM(G)$. Its \emph{law} or \emph{distribution} $\mu_\Pi$ is the pushforward measure $\Pi_*(\PP)$ on $\MM(G)$. It is \emph{invariant} if its law is an invariant probability measure for the action $G \acts \MM(G)$.
	
	The associated \emph{point process action} of an invariant point process $\Pi$ is $G \acts (\MM(G), \mu_\Pi)$.
\end{defn}

Some remarks and caveats are in order:
\begin{itemize}
	\item Point processes which are not invariant are very much of interest, but will only come up when we discuss ``Palm processes''. Thus we will sometimes say ``point process'' when we strictly mean \emph{invariant} point process.
	\item Speaking properly, we are discussing \emph{simple} point processes, that is, those where each point has multiplicity one. We will discuss this more later.
    \item $\Xi$-marked point processes are defined similarly, with $\Xi^\MM$ taking the place of $\MM$. There isn't much difference between marked point processes and unmarked ones for our purposes (it's just a case of which is more convenient for the particular problem at hand). Thus ``point process'' might also mean ``marked point process''.
    \item One could certainly define point processes on a discrete group, but this is essentially percolation theory. We are specifically trying to understand the nondiscrete case, and so will assume $G$ is nondiscrete.
    \item The other case of interest we will discuss is $\Isom(S)$-invariant point processes on $S$, where $S$ is a Riemannian symmetric space. For instance, one would consider isometry invariant point processes on Euclidean space $\RR^n$ or hyperbolic space $\HH^n$. We will discuss this case more in Section \ref{homogeneousspaces}.
    \item Our interest in point processes is almost exclusively \emph{as actions}. We will therefore rarely distinguish between a point process proper and its distribution. Thus we may use expressions like ``suppose $\mu$ is a point process'' to mean ``suppose $\mu$ is the distribution of some point process''.
    \item The configuration space of any Polish space will be Polish, so the probability theory of point processes on such spaces is well behaved. The metric properties of configuration spaces that we require are listed in the appendix, with references for proofs.
\end{itemize}

\begin{defn}
    
    The \emph{intensity} of a point process $\mu$ is
    \[
        \intensity(\mu) = \frac{1}{\lambda(U)} \EE_\mu \left[ N_U \right],
    \]
    where $U \subset G$ is any Borel set of positive (but finite) Haar measure, and $N_U(\omega) = \abs{\omega \cap U}$ is its point counting function.
\end{defn}

To see that the intensity is well-defined (that is, does not depend on our choice of $U$), observe that the function $U \mapsto \EE_\mu[N_U]$ defines a Borel measure on $G$ which inherits invariance from the shift invariance of $\mu$. So by uniqueness of Haar measure, it is some scaling of our fixed Haar meausure $\lambda$ -- the intensity is exactly this multiplier. We also see that whilst the intensity depends on our choice of Haar measure, it scales linearly with it.

Note that a point process has intensity zero if and only if it is empty almost surely.

\subsection{Examples of point processes}

\begin{example}[Lattice shifts]

Let $\Gamma < G$ be a \emph{lattice}, that is, a discrete subgroup that admits an invariant probability measure $\nu$ for the action $G \acts G / \Gamma$. The natural map $\MM(G / \Gamma) \to \MM(G)$ given by
\[
    \omega \mapsto \bigcup_{a\Gamma \in \omega} a\Gamma
\]
is left-equivariant, and hence maps invariant point processes on $G / \Gamma$ to invariant point processes on $G$. In particular, we have the \emph{lattice shift}, given by choosing a $\nu$-random point $a\Gamma$.

\end{example}

\begin{example}[\textbf{Induction from a lattice}]
Now suppose one also has a pmp action $\Gamma \acts (X, \mu)$. It is possible to \emph{induce} this to a pmp action of $G$ on $G / \Gamma \times X$. This can be described as an $X$-marked point process on $G$. To do this, fix a fundamental domain $\mathscr{F} \subset G$ for $\Gamma$. Choose $f \in \mathscr{F}$ uniformly at random, and independently choose a $\mu$-random point $x \in X$. Let
	\[
	  \Pi = \{ (f\gamma, \gamma \cdot x) \in G \times X \mid \gamma \in \Gamma \}.  
	\]
	Then $\Pi$ is a $G$-invariant $X$-marked point process.
\end{example}

In this way one can view point processes as generalised lattice shift actions. Note that there are groups without lattices (for instance Neretin's group, see \cite{MR2881324}), but every group admits interesting point processes, as we discuss now. The most fundamental of these is known as \emph{the Poisson point process}. We will define this after reviewing the Poisson distribution:

Recall that a random integer $N$ is \emph{Poisson distributed with parameter $t > 0$} if
\[
\PP[N = k] = \exp(-t)\frac{t^k}{k!}.\]
We write $N \sim \texttt{Pois}(t)$ to denote this. It is convenient to extend this definition to $t = 0$ and $t = \infty$ by declaring $N \sim \texttt{Pois}(0)$ when $N = 0$ almost surely and $N \sim \texttt{Pois}(\infty)$ when $N = \infty$ almost surely.

\begin{defn}
	Let $X$ be a complete and separable metric space equipped with a non-atomic Borel measure $\lambda$.
	
	A point process $\Pi$ on $X$ is \emph{Poisson with intensity $t > 0$} if it satisfies the following two properties:
    	\begin{description}
        	\item[(Poisson point counts)] for all $U \subseteq G$ Borel, $N_U(\Pi)$ is a Poisson distributed random variable with parameter $t \lambda(U)$, and
            \item[(Total independence)] for all $U, V \subseteq G$ \emph{disjoint} Borel sets, the random variables $N_U(\Pi)$ and $N_V(\Pi)$ are \emph{independent}.
        \end{description}
\end{defn}

For reasons that should not be immediately apparent, the above defining properties are in fact equivalent. We will write $\PPP_t$ for the distribution of such a random variable, or simply $\PPP$ if the intensity is understood. 

We think of the Poisson point process as a completely random scattering of points in the group. It is an analogue of Bernoulli site percolation for a continuous space.

We now construct the process (somewhat) explicitly. Partition $G$ into disjoint Borel sets $U_1, U_2, \ldots$ of positive but finite volume. For each of these, independently sample from a Poisson distribution with parameter $t \lambda(U_i)$. Place that number of points in the corresponding $U_i$ (independently and uniformly at random).

This description can be turned into an explicit sampling rule\footnote{That is, one can define a measurable function $f : \prod_n X_n \to \MM$ defined on an appropriate product of probability spaces such that the pushforward measure is the distribution of the Poisson point process.}, if one desires.

For proofs of basic properties of the Poisson point process (such as the fact that it does not depend on the partition chosen above), see the first five chapters of Kingman's book \cite{MR1207584}. 
\begin{defn}

A pmp action $G \acts (X, \mu)$ is \emph{ergodic} if for every $G$-invariant measurable subset $A \subseteq X$, we have $\mu(A) = 0$ or $\mu(A) = 1$.

The action is \emph{mixing} if for all measurable $A, B \subseteq (X, \mu)$ we have
\[
    \lim_{g \to \infty} \mu(gA \cap B) = \mu(A)\mu(B).
\]
The action is \emph{essentially free} if $\stab_G(x) = \{1\}$ for $\mu$ almost every $x \in X$. In the case of point process actions we will sometimes use the term \emph{aperiodic} to refer to this.

\end{defn}

\begin{prop}
	The Poisson point process actions $G \acts (\MM, \PPP_t)$ on a noncompact group $G$ are essentially free and ergodic (in fact, mixing).
\end{prop}

A proof of freeness that is readily adaptable to our setting can be found as Proposition 2.7 of \cite{MR3664810}. For ergodicity and mixing, see the proof of the discrete case in Proposition 7.3 of the Lyons-Peres book \cite{MR3616205}. It directly adapts, once one knows the required cylinder sets exist.

Although the subscript $t$ suggests that the Poisson point processes form a continuum family of actions, this is not always the case:

\begin{thm}[Ornstein-Weiss]
	Let $G$ be an amenable group which is not a countable union of compact subgroups. Then the Poisson point process actions $G \acts (\MM, \PPP_t)$ are all isomorphic.
\end{thm}

The following definition uses notation that does not appear in the literature (the object of course does, but there does not appear to be a symbolic representation for it):

\begin{defn}
	If $\Pi$ is a point process, then its \emph{IID version} is the $[0,1]$-marked point process $[0,1]^\Pi$ with the property that conditional on its set of points, its labels are independent and IID $\text{Unif}[0,1]$ distributed. If $\mu$ is the law of $\Pi$, then we will write $[0,1]^\mu$ for the law of $[0,1]^\Pi$.
	
	One can define the IID of a point process over spaces other than $[0,1]$ (for instance, $[n] = \{1, 2, \ldots, n\}$ with the counting measure), but we will only use the full IID.
	
\end{defn}

\begin{remark}

As we've mentioned, $[0,1]$-marked point processes on $G$ are particular examples of point processes on $G \times [0,1]$. One can show (see Theorem 5.6 of \cite{MR3791470}) that the Poisson point process on $G \times [0,1]$ with respect to the product measure $\lambda \otimes \text{Leb}$ is just the IID version of the Poisson point process on $G$, a fact which we will make use of later.

\end{remark}

\begin{prop}

The IID Poisson point process on a noncompact group is ergodic (and in fact mixing).

\end{prop}

This can be seen by viewing the IID Poisson on $G$ as the Poisson point process on $G \times S^1$, restricted to $G$. Note that the restriction of a mixing action to a noncompact subgroup is mixing.

\begin{remark}

One can define ``the IID'' of any probability measure preserving countable Borel equivalence relation, see \cite{MR3813200}. This construction is known as \emph{the Bernoulli extension}, and is ergodic if the base space is ergodic. 

\end{remark}

\begin{prop}

Let $\Pi$ be a point process on a group $G$ which is non-empty almost surely. Then $\abs{\Pi} = \infty$ almost surely if and only if $G$ is noncompact.

\end{prop}

\begin{proof}
    
    It is immediate that any point process on a compact group must be finite almost surely (as it is a discrete subset of the space).
    
    Now suppose $\Pi$ is a non-empty point process on $G$ which is finite almost surely. Then the IID of this process $[0,1]^\Pi$ still has this property. We define the following $G$-valued random variable:
    \[
        f([0,1]^\Pi) = \text{ the unique } x \in \Pi \text{ with maximal label in } [0,1]^\Pi.
    \]
    The invariance of the point process translates into equivariance of the map $f : [0,1]^\MM \to G$. Thus this random variable's law is an invariant probability measure on $G$. Such a measure exists exactly when $G$ is compact.
\end{proof}

\subsection{Factors of point processes}
\begin{defn}
    
    A \emph{point process factor map} is a $G$-equivariant and measurable map $\Phi : \MM \to \MM$. If $\mu$ is a point process and $\Phi$ is only defined $\mu$ almost everywhere, then we will call it a \emph{$\mu$ factor map}.
    
    We will be interested in two monotonicity conditions:
    \begin{itemize}
        \item if $\Phi(\omega) \subseteq \omega$ for all $\omega \in \MM$, we will call $\Phi$ a \emph{thinning} (and usually denote it by $\theta$), and 
        \item if $\Phi(\omega) \supseteq \omega$ for all $\omega \in \MM$, we will call $\Phi$ a \emph{thickening} (and usually denote it by $\Theta$).
    \end{itemize}
    
    We use the same terms for marked point processes as well. 
    
\end{defn}

\begin{remark}\label{thinningconfusio}

There are \emph{two} possible ways to interpret the above monotonicity conditions for a $\Xi$-marked point process, depending on what you want to do with the mark space. One can consider
\[
    \Phi : \Xi^\MM \to \Xi^\MM, \text{ or } \Phi : \Xi^\MM \to \MM.
\]  
In the former case, the definition above works verbatim. In the latter case, one should interpret a statement like ``$\omega \subseteq \Phi(\omega)$'' as ``$\omega$ is contained in the underlying set $\pi(\Phi(\omega))$ of $\Phi(\omega)$, where $\pi : \Xi^\MM \to \MM$ is the map that forgets labels.
\end{remark}

\begin{example}[Metric thinning]\label{deltathinningdefn}
    
    Let $\delta > 0$ be a tolerance parameter. The \emph{$\delta$-thinning} is the equivariant map $\theta_\delta : \MM \to \MM$ given by
    \[
        \theta^\delta(\omega) = \{ g \in \omega \mid d(g, \omega \setminus \{g\}) > \delta \}.
    \]
  
    When $\theta^\delta$ is applied to a point process, the result is always a $\delta$-separated point process\footnote{Probabilists refer to such processes as \emph{hard-core}.} (but possibly empty).
  
    We define $\theta^\delta$ in the same way for marked point processes (that is, it simply ignores the marks).
  
\end{example}

\begin{example}[Independent thinning]\label{independentthinning}
Let $\Pi$ be a point process. The \emph{independent $p$-thinning} defined on its IID $[0,1]^\Pi$ is given by
\[
    \mathcal{I}_p([0,1]^\Pi) = \{g \in \Pi \mid \Pi_g \leq p \}.
\]
\end{example}

One can show that independent $p$-thinning of the Poisson point process of intensity $t > 0$ yields the Poisson point process of intensity $pt$, as one would expect. See Chapter 5 of \cite{MR3791470} for further details.

\begin{example}[Constant thickening]\label{constantthickening}
    
    Let $F \subset G$ be a finite set containing the identity $0 \in G$, and $\Pi$ be a point process which is \emph{$F$-separated} in the sense that $\Pi \cap \Pi f = \empt$ for all $f \in F\setminus\{0\}$. Then there is the associated thickening $\Theta^F(\Pi) = \Pi F$. 
    It is intuitively obvious that $\intensity (\Theta^F(\Pi)) = \abs{F} \intensity (\Pi)$. This can be formally established as follows: let $U \subseteq G$ be of unit volume. Then
    \begin{align*}
        \intensity (\Theta^F(\Pi) ) &= \EE\abs{U \cap \Pi F}  && \text{By definition}  \\
        &= \sum_{f \in F} \EE\abs{U \cap \Pi f} && \text{By }F\text{-separation} \\
        &= \sum_{f \in F} \EE \abs{Uf^{-1} \cap \Pi}  && \\
        &= \sum_{f \in F} \EE\abs{U \cap \Pi} && \text{By \emph{unimodularity}} \\
        &= \abs{F} \intensity (\Pi).
    \end{align*}
    This is the first real appearance of our unimodularity assumption. 
    
    In particular, we can demonstrate that $\intensity \Theta^F(\Pi) = \abs{F} \intensity \Pi$ is \emph{not} automatically true without unimodularity. For this, let $\Pi$ denote the unit intensity Poisson point process on $G$, and $F = \{0, f\}$ where $f \in G$ is chosen such that $\lambda(Uf^{-1}) < 1$. Then $\abs{Uf^{-1} \cap \Pi}$ is Poisson distributed with parameter $\lambda(Uf^{-1})$, and so by the above calculation $\intensity\Theta^F(\Pi) < 2 \cdot \intensity \Pi$.
\end{example}

Monotone maps have been investigated in the specific case of the Poisson point process on $\RR^n$. We note the following interesting theorems:

\begin{thm}[Holroyd, Peres, Soo \cite{MR2884878}]
Let $s > t > 0$. Then the Poisson point process on $\RR^n$ of intensity $s$ can be thinned to the Poisson point process of intensity $t$. That is, there exists an equivariant and deterministic thinning $\theta : (\MM(\RR), \PPP_s) \to (\MM(\RR), \PPP_t)$.

\end{thm}

\begin{thm}[Gurel-Gurevich and Peled \cite{MR3096589}]
Let $t > s > 0$ be intensities. Then the Poisson point process on $\RR^n$ of intensity $s$ cannot be thickened to the Poisson point process of intensity $t$. That is, there is no equivariant and deterministic thickening $\Theta : (\MM(\RR), \PPP_s) \to (\MM(\RR), \PPP_t)$.

\end{thm}

We stress in the above theorems the \emph{deterministic} nature of the maps. If one is allowed additional randomness (that is, one asks for a factor of IID map), then both theorems are easily established.

We note the following fact, which we will use (and prove) later after developing some notation.

\begin{example}
    If $\Pi$ is any non-empty point process, then its IID factors onto the Poisson (in fact, onto the IID Poisson). 
\end{example}
    
    
\begin{defn}
    
    A \emph{factor $\Xi$-marking} of a point process is a $G$-equivariant map $\mathscr{C} : \MM \to \Xi^\MM$ such that the underlying subset in $G$ of $\mathscr{C}(\omega)$ is $\omega$. That is, $\mathscr{C}$ is a rule that assigns a mark from $\Xi$ to each point of $\omega$ in some deterministic way. Again, if $\mathscr{C}$ is only defined $\mu$ almost everywhere then we will call it a \emph{$\mu$ factor $\Xi$-marking}.
    
\end{defn}

\begin{example}\label{colouring}
    
    Let $\theta : \MM \to \MM$ be a thinning. Then the associated $2$-colouring is $\mathscr{C}_\theta : \MM \to \{0, 1\}^\MM$ given by
    \[
        \mathscr{C}_\theta(\omega) = \{ (g, \1[{g \in \theta(\omega)}]) \in G \times \{0, 1\} \mid g \in \omega \}.
    \]
    We will see that all markings are built out of thinnings in a similar way.
\end{example}

\begin{remark}\label{thinninglost}

There is a difference between the \emph{thinning map $\theta$} and the resulting \emph{thinned process $\theta_*(\mu)$} that can be a source for confusion. Passing to the thinned process (in principle) can lose information about $\mu$.

For example, let $\Pi$ denote a Poisson point process on $G$ and $\Upsilon$ an independent random shift of a lattice $\Gamma < G$. Define the following thinning $\theta : \MM \to \MM$ by
\[
    \theta(\omega) = \{ g \in \omega \mid g\Gamma \subseteq \omega \}.
\]
Then $\theta(\Pi \cup \Upsilon) = \Upsilon$, and so the thinning completely loses the Poisson point process.

\end{remark}

\begin{defn}\label{inputoutputdefn}
Let $\Phi : \MM \to \MM$ be a factor map. We think of its input $\omega$ as being red, its output $\Phi(\omega)$ as being blue, and their overlap $\omega \cap \Phi(\omega)$ as being purple. 

For $g \in \omega$, let $\texttt{Colour}(g) \in \{\text{Red, Blue, Purple}\}$ be
\[ \texttt{Colour}(g) = 
    \begin{cases}
        \text{Red} & \text{ If } g \in \omega \setminus \Phi(\omega), \\
        \text{Blue} & \text{ If } g \in \Phi(\omega)\setminus \omega, \\
        \text{Purple} & \text{ If } g \in \omega \cap \Phi(\omega).
    \end{cases}
\]
Now define $\Theta^\Phi : \MM \to \{\text{Red, Blue, Purple}\}^\MM$ to be the following \emph{input/output thickening} of $\Phi$ (see also Figure \ref{inputoutputfigure}):
\[
    \Theta^\Phi(\omega) = \{ (g, \texttt{Colour}(g)) \in G \times \text{Red, Blue, Purple}\} \mid g \in \omega \}.
\]

\begin{figure}[h]\label{inputoutputfigure}
\includegraphics[scale=.4]{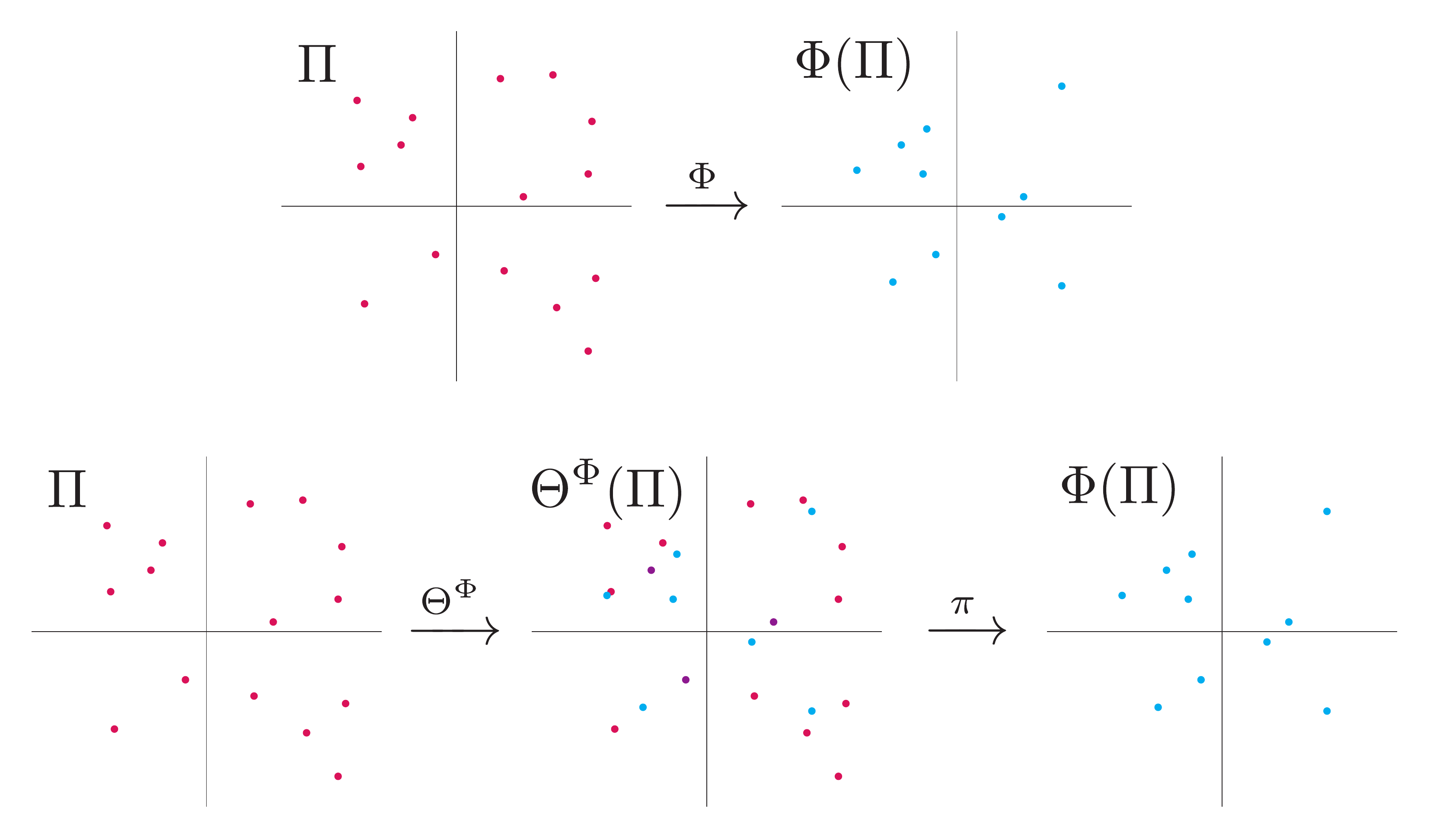}
\centering
\caption{This is how you should picture the input/output thickening of a factor map.}
\end{figure}

Let $\pi : \{\text{Red, Blue, Purple}\}^\MM \to \MM$ be the projection map that deletes red points and then forgets colours, that is,
\[
    \pi(\omega) = \{ g \in \omega \mid \omega_g \in \{\text{Blue, Purple}\} \}.
\]

\end{defn}

\begin{remark}\label{factorsdecompose}
Observe that $\Phi = \pi \circ \Theta^\Phi$ -- that is, an \emph{arbitrary} factor map decomposes as the composition of a thinning and a thickening. In this way we can often reduce the study of arbitrary factors to that of \emph{monotone} factors.
\end{remark}

\begin{defn}
	The \emph{space of graphs in $G$} is
    \[
    	\graph(G) = \{ (V, E) \in \MM(G) \times \MM(G \times G) \mid E \subseteq V \times V \}.
    \]
    This is a Borel $G$-space (with the diagonal action). 
    
    A \emph{factor graph} is a measurable and $G$-equivariant map $\Phi : \MM(G) \to \graph(G)$ with the property that the vertex set of $\Phi(\omega)$ is $\omega$.
    
    If a factor graph is connected, then we will refer to it as a \emph{graphing}.
    
\end{defn}

\begin{remark}
	The elements of $\graph(G)$ are technically directed graphs, possibly with loops, and without multiple edges between the same pair of vertices. It's possible to define (in a Borel way) whatever space of graphs one desires (coloured, undirected, etc.) by taking appropriate subsets of products of configuration spaces.
\end{remark}

\begin{remark}
	One might prefer to call factor graphs as above \emph{monotone} factor graphs. Our terminology follows that of probabilists, see for instance \cite{holroyd2005}. We have not found a use for the less restrictive factor graph concept. 
\end{remark}

\begin{example}\label{distanceR}
    
    The \emph{distance-$R$} factor graph is the map $\mathscr{D}_R : \MM \to \graph(G)$ given by
    \[
        \mathscr{D}_R(\omega) = \{ (g, h) \in \omega \times \omega \mid d(g, h) \leq R \}.
    \]
    The connectivity properties of this graph fall under the purview of continuum percolation theory, see for instance \cite{meester1996continuum}.
\end{example}

\section{The rerooting equivalence relation and groupoid}

We now introduce a pair of algebraic objects that capture factors of a point process. For exposition's sake, we will first discuss unmarked point processes on a group $G$. 

\begin{defn}
    The \emph{space of rooted configurations on $G$} is
    \[
        \MMo(G) = \{ \omega \in \MM(G) \mid 0 \in \omega \}.
    \]
    
    If $G$ is understood, then we will drop it from the notation for clarity.
    
    The \emph{rerooting equivalence relation} on $\MMo$ is the orbit equivalence relation of $G \acts \MM$ restricted to $\MMo$. Explicitly: 
    \[
        \Rel = \{ (\omega, g^{-1}\omega) \in \MMo \times \MMo \mid g \in \omega \}.
    \]
    
    This defines a countable Borel equivalence relation structure on $\MMo$. It is degenerate whenever $\omega \in \MMo$ exhibits symmetries: for instance, the equivalence class of $\ZZ$ viewed as an element of $\MMo(\RR)$ is a singleton. We are usually interested in essentially free actions, where such difficulties will not occur. Nevertheless, we do care about lattice shift point processes and so we will introduce a groupoid structure that keeps track of symmetries.
    
    The \emph{space of birooted configurations} is
    \[
        \Marrow = \{ (\omega, g) \in \MMo \times G \mid g \in \omega \}.
    \]
    
    We visualise an element $(\omega, g) \in \Marrow$ as the rooted configuration $\omega \in \MMo$ with an arrow pointing to $g \in \omega$ from the root (ie, the identity element of $G$).
    
    The above spaces form a \emph{groupoid} $(\MMo, \Marrow)$ which we will refer to as the \emph{rerooting groupoid}. Its unit space is $\MMo$ and its arrow space is $\Marrow$. We can identify $\MMo$ with $\MMo \times \{0\} \subset \Marrow$. 
    
    The multiplication structure is as follows: we declare a pair of birooted configurations $(\omega, g), (\omega', h)$ in $\Marrow$ to be \emph{composable} if $\omega' = g^{-1}\omega$, in which case
    \[
        (\omega, g) \cdot (\omega', h) := (\omega, gh).
    \]
    
    Note that if $\Gamma < G$ is a discrete subgroup (so $\Gamma \in \MMo(G)$), then the above multiplication on $\{\Gamma\} \times \Gamma \subset \Marrow(G)$ is just the usual one.
    
    The \emph{source map} $s : \Marrow \to \MMo$ and \emph{target map} $t : \Marrow \to \MMo$ are
    \[
        s(\omega, g) = \omega, \text{ and } t(\omega, g) = g^{-1}\omega.
    \]
\end{defn}

    Note that the rerooting groupoid is \emph{discrete} in the sense that $s^{-1}(\omega)$ is at most countable for all $\omega \in \MMo$. 

\begin{remark}

Let $\Maper_0$ denote the set of rooted configurations $\omega$ that are \emph{aperiodic} in the sense that $\stab_G(\omega) = \{e\}$. Then the groupoid generated by $\Maper_0$ in $\MMo$ is \emph{principal}\footnote{Recall that a groupoid is \emph{principal} if its isotropy subgroups are all trivial. That is, the groupoid structure is just that of an equivalence relation}.

\end{remark}

\begin{defn}

If $\Xi$ is a space of marks, then the \emph{space of $\Xi$-marked rooted configurations} is 
\[
    \Xi^\MMo = \{ \omega \in \Xi^\MM \mid \exists \xi \in \Xi \text{ such that } (0, \xi) \in \omega \}.
\]

The \emph{$\Xi$-marked rerooting groupoid} is defined as previously, with $\Xi^\MMo$ taking the place of $\MMo$.

\end{defn}

\subsection{Borel correspondences between the groupoid and factors}\label{borelcorrespondences}

Suppose $\theta : \MM \to \MM$ is an equivariant and measurable thinning. Then we can associate to it a subset of the rerooting groupoid, namely
\[
    A_\theta = \{ \omega \in \MM \mid 0 \in \theta(\omega) \}.
\]
This association has an inverse: given a Borel subset $A \subseteq \MMo$, we can define a thinning $\theta^A : \MM \to \MM$
\[
    \theta^A(\omega) = \{g \in \omega \mid g^{-1}\omega \in A \}.
\]

Thus we see that \emph{Borel subsets $A \subseteq \MMo$ of the rerooting groupoid correspond to Borel thinning maps $\theta : \MM \to \MM$}.

In the $\Xi$-marked case, one associates to a subset $A \subseteq \Xi^\MMo$ a thinning $\theta^A : \Xi^\MM \to \Xi^\MM$. 

In a similar way, we can see that if $P : \MMo \to [d]$ is a Borel partition of $\MMo$ into $d$ classes, then there is an associated factor $[d]$-colouring $\mathscr{C}^P : \MM \to [d]^\MM$ given by
\[
    \mathscr{C}^P(\omega) = \{ (g, P(g^{-1}\omega) \in G \times [d] \mid g \in \omega \},
\]
and given a factor $[d]$-colouring $\mathscr{C} : \MM \to [d]^\MM$ one associates the partition $P^\mathscr{C} : \MMo \to [d]$ given by
\[
    P(\omega) = c, \text{ where } c \text{ is the unique element of } [d] \text{ such that } (0, c) \in \mathscr{C}(\omega).  
\]
Again, these associations are mutual inverses. 

More generally, we see that \emph{Borel factor $\Xi$-markings $\mathscr{C} : \MM \to \Xi^\MM$ correspond to Borel maps $P : \MMo \to \Xi$}.

Now suppose that $\mathscr{G} : \MM \to \graph(G)$ is an equivariant and measurable factor graph. Then we can associate to it a subset of the rerooting groupoid's arrow space, namely
\[
    \mathscr{A}_\mathscr{G} = \{ (\omega, g) \in \Marrow \mid (0,g) \in \mathscr{G}(\omega)\}.
\]
In the other direction, we associate to a subset $\mathscr{A} \subseteq \Marrow$ the factor graph $\mathscr{G}^\mathscr{A} : \MM \to \graph(G)$
\[
    \mathscr{G}^\mathscr{A}(\omega) = \{ (g, h) \in \omega \times \omega \mid (g^{-1}\omega, g^{-1}h) \in \mathscr{A} \}.
\]

Thus we see that \emph{Borel subsets $\mathscr{A} \subseteq \Marrow$ of the rerooting groupoid's arrow space correspond to Borel factor (directed) graphs $\mathscr{G} : \MM \to \graph(G)$}.

\begin{remark}

If $\mu$ is a point process, then the correspondence still works in one direction: namely, we can associate subsets $A \subset \MMo$ (or $\mathscr{A} \subseteq \Marrow$) to $\mu$-thinnings $\theta^A: (\MM, \mu) \to \MM$ (or $\mu$-factor graphs $\mathscr{G}_\mathscr{A}: (\MM, \mu) \to \MM$ respectively). 

We run into trouble in the other direction: suppose $\theta : \MM \to \MM$ is a thinning, but only defined $\mu$ almost everywhere. We wish to restrict it to $\MMo$, but a priori this makes no sense -- that is a subset of measure zero. It turns out that there is a way to make sense of this due to equivariance, but it will require some more theory that we explain in the next section.

\end{remark}

\subsection{The Palm measure}

We will now associate to a (finite intensity) point process $\mu$ a probability measure $\mu_0$ defined on the rerooting groupoid $\MMo$. When the ambient space is unimodular, this will turn the rerooting groupoid into a \emph{probability measure preserving (pmp) discrete groupoid}.

Informally, the Palm measure of a point process $\Pi$ is the process conditioned to contain the root. A priori this makes no sense (the subset $\MMo$ has probability zero), but there is an obvious way one could interpret the statement: condition on the process to contain a point in an $\e$ ball about the root, and take the limit as $\e$ goes to zero. See Theorem 13.3.IV of \cite{daley2007introduction} and Section 9.3 of \cite{MR3791470} for further details.

We will instead take the following concept of \emph{relative rates} as our basic definition:

\begin{defn}
	Let $\Pi$ be a point process of finite intensity with law $\mu$. Its (normalised) \emph{Palm measure} is the probability measure $\mu_0$ defined on Borel subsets of $\MMo$ by
    \[
    	\mu_0(A) := \frac{\intensity(\theta^A(\Pi))}{\intensity(\Pi)},
    \]
    where $\theta^A$ is the thinning associated to $A \subseteq \MMo$. 
    
    More explicitly,
    \[
    	\mu_0(A) := \frac{1}{\intensity(\mu) \lambda(U)} \EE_\mu \left[ \#\{g \in U \mid g^{-1}\omega \in A \} \right],
    \]
    where $U \subseteq G$ is any measurable set with $0 < \lambda(U) < \infty$. To make formulas simpler, we will often choose $U$ to be of unit volume. Alternatively, note that by the definition of intensity we may write
    \[
        \mu_0(A) = \frac{\EE \left[ \#\{g \in U \mid g^{-1}\Pi \in A \} \right]}{\EE \abs{U \cap \Pi}}.
    \]

We also define the Palm measure of a $\Xi$-marked point process similarly, with $\Xi^\MMo$ taking the place of $\MMo$.

A \emph{Palm version} of $\Pi$ is any random variable $\Pi_0$ with law $\mu_0$. That is, we require that for all Borel $B \subseteq \MMo$ we have
\[
    \PP[\Pi_0 \in B] = \mu_0(B).
\]

\end{defn}

We now describe some \emph{Palm calculus}. If $\Pi$ is a point process with Palm version $\Pi_0$ and $\Phi(\Pi)$ is some factor map then we wish to express the Palm version $\Phi(\Pi)_0$ of $\Phi(\Pi)$ in terms of $\Pi_0$ and $\Phi$. The Palm calculus tells us how this is done. It will be sufficient for our purposes to compute the Palm measure of factors for factor which are forgettings, thinnings, coloured thickenings, and colourings. In each case the answer is more or less obvious, so we will give an informal description of the answer and then verify that it satisfies the required property.

\begin{example}[Forgetting labels]
If $\Pi$ is a labelled point process, then the Palm measure of $\Pi$ \emph{after} we forget the labels is the same thing as forgetting the labels from the Palm measure $\Pi_0$.

We prove this after the following clarification:

When talking about the Palm measure for a $\Xi$-marked point process, it is important in the above to choose the correct thinning. Recall from Remark \ref{thinningconfusio} that for a subset $A \subseteq \Xi^\MMo$ one can discuss \emph{two} possible kinds of thinnings, namely
\[
    \theta^A : \Xi^\MM \to \Xi^\MM \text{ or } \pi \circ \theta^A : \Xi^\MM \to \MM,
\]
where $\pi : \Xi^\MM \to \MM$ is the map that forgets labels.

It is the \emph{former} kind of thinning one should take.

Note that if $\Pi$ is a $\Xi$-marked point process, then its intensity remains the same if you forget the marks, that is, $\intensity \Pi = \intensity \pi(\Pi)$. More generally, the operation of taking the Palm version \emph{commutes} with forgetting labels. That is, $\pi(\Pi_0) = (\pi(\Pi))_0$. To see this, let $B \subseteq \MMo$, and observe
\begin{align*}
    \PP[ \pi(\Pi_0) \in B] &= \PP[\Pi_0 \in \pi^{-1}(B)]  \\
    &= \frac{\intensity \theta^{\pi^{-1}(B)}(\Pi)}{\intensity \Pi} \\
    &= \frac{\intensity \pi(\theta^{\pi^{-1}(B)}(\Pi))}{\intensity \pi(\Pi)} \\
    &= \frac{ \intensity \theta^B(\pi(\Pi))}{\intensity \pi(\Pi)} \\
    &= \PP[ \pi(\Pi)_0 \in B],
\end{align*}
where we simply followed our nose.
\end{example}

\begin{example}[Lattice actions]
	If $\Gamma < G$ is a lattice, then the Palm measure of the associated lattice shift is just $\delta_\Gamma$ -- that is, the atomic measure on $\Gamma \in \MMo(G)$. 
    More generally, if $\Gamma \acts (X, \mu)$ is a pmp action, then the Palm measure of the associated induced $X$-marked point process is its \emph{symbolic dynamics}. That is, the map $\Sigma : (X, \mu) \to X^\MM$ given by
    \[
        \Sigma(x) = \{ (\gamma, \gamma^{-1} \cdot x) \in G \times X \mid \gamma \in \Gamma \}.
    \]
    pushes forward $\mu$ to the Palm measure. In words, you sample a $\mu$-random point $x \in X$ and track its orbit under $\Gamma$ (the inverse is an artefact of our left bias).
\end{example}

\begin{remark}

Suppose $\Pi$ is a finite intensity point process such that its Palm version is an atomic measure, say $\Pi_0 = \Omega$ almost surely where $\Omega \in \MMo$. Then $\Omega$ is a lattice in $G$. Note that $\Omega$ is automatically a discrete subset of $G$, and a simple mass transport argument shows that it is a subgroup. The covolume of this subgroup is the reciprocal of the intensity of $\Pi$.

\end{remark}

\begin{example}[Mecke-Slivnyak Theorem]\label{palmofpoisson}
    
    If $\Pi$ is a Poisson point process, then its Palm measure has the same law as $\Pi \cup \{0\}$, where $0 \in G$ is the identity. 
    
    In fact, this is a \emph{characterisation} of the Poisson point process: if the Palm measure of $\mu$ is obtained by simply adding the root\footnote{More formally, consider the map $F : \MM \to \MMo$ given by $F(\omega) = \omega \cup \{0\}$, by ``adding the root'' we mean the Palm measure $\mu_0$ is the pushforward $F_*\mu$.}, then $\mu$ is the Poisson point process (of some intensity).
    
\end{example}

The proof of the above fact can be found in Section 9.2 of \cite{MR3791470}. As a consequence, the Palm measure of the IID Poisson is the IID of the Palm measure of the Poisson itself. 

\begin{example}\label{palmofcolouring}
    The Palm version $\mathscr{C}^A(\Pi)_0$ of a $2$-colouring $\mathscr{C} : \MM \to \{0,1\}^\MM$ determined by a subset $A \subseteq \MMo$ (as in Example \ref{colouring}) is simply $\mathscr{C}(\Pi_0)$.
\end{example}

\begin{example}[Thinnings]
    The Palm version $\theta(\Pi)_0$ of a thinning $\theta = \theta^A$ of $\Pi$ (determined by a subset $A \subseteq \MMo$) is described in terms of its Palm version $\Pi_0$ as a conditional probability as follows:
    \[
        \PP[\theta(\Pi)_0 \in B] = \PP[\theta(\Pi_0) \in B \mid \Pi_0 \in A]
    \]
    for any $B \subseteq \MMo$.
    
    That is, the Palm measure $\theta(\Pi)_0$ can be obtained by sampling from $\Pi_0$ conditioned that the root is retained in the thinning, and then applying the thinning.
    
    To see this, first one should work from the definitions to show that $\theta^B(\theta^A(\Pi)) = \theta^{A \cap (\theta^A)^{-1}(B)}$. Therefore
    \begin{align*}
        \PP[(\theta(\Pi))_0 \in B] &= \frac{\intensity \theta^B(\theta^A(\Pi))}{\intensity \theta^A(\Pi)} &&  \\
        &=  \left.{\frac{\intensity \theta^{A \cap (\theta^A)^{-1}(B)}(\Pi)}{\intensity \Pi}} \middle/ {\frac{\intensity \theta^A(\Pi)}{\intensity \Pi}}\right.  && \text{By the observation}    \\
        &= \frac{\PP[\Pi_0 \in A \cap (\theta^A)^{-1}(B)]}{\PP[\Pi_0 \in A]} && \\
        &= \frac{\PP[\{\theta(\Pi_0) \in B\} \cap \{\Pi_0 \in A\}]}{\PP[\Pi_0 \in A]} \\
        &= \frac{\PP[\{\theta(\Pi_0) \in B\} \cap \{\Pi_0 \in A\}]}{\PP[\Pi_0 \in A]},
    \end{align*}
    which is exactly the definition of the desired conditional probability.
\end{example}

\begin{example}\label{palmofthickening}
    
    Let $\Theta = \Theta^F$ be a constant thickening determined by $F \subset G$, as described in Example \ref{constantthickening}. If $\Pi$ is an $F$-separated process, then the Palm version $\Theta(\Pi)_0$ of the thickening $\Theta(\Pi)$ is as follows: sample from $\Pi_0$, and independently choose to root $\Theta(\Pi_0)$ at a uniformly chosen element $X$ of $F$. That is, $\Theta(\Pi)_0 \overset{d}{=} X^{-1} \Theta(\Pi_0)$.
    
    To see this, we compute\footnote{When we define the Palm measure of a set $B \subseteq \MMo$, we usually write ``$g \in U$'' rather than ``$g \in U \cap \Pi$'', as the latter condition $g^{-1} \Pi \in B$ already implies $g \in \Pi$. For this computation it is better to really spell it out though.} as follows:
    
    \begin{align*}
        &\PP[\Theta(\Pi)_0 \in B] = \frac{1}{\intensity \Theta(\Pi)} \EE[ \#\{g \in U \cap \Pi F \mid g^{-1} \Theta(\Pi) \in B \} ] && \text{By definition} \\
        &= \frac{1}{\abs{F}} \frac{1}{\intensity \mu} \sum_{f \in F} \EE[ \#\{g \in U \cap \Pi f \mid g^{-1} \Theta(\Pi) \in B \} ] && \text{By Example \ref{constantthickening}} \\
         &= \frac{1}{\abs{F}} \frac{1}{\intensity \mu} \sum_{f \in F} \EE[ \#\{g \in Uf^{-1} \cap \Pi \mid g^{-1} \Pi \in \Theta^{-1}(B) \} ] && \text{By equivariance} \\
         &= \frac{1}{\abs{F}} \frac{1}{\intensity \mu} \sum_{f \in F} \EE[ \#\{g \in U \cap \Pi \mid g^{-1} \Pi \in \Theta^{-1}(B) \} ] && \text{By unimodularity} \\
         &= \frac{1}{\abs{F}} \sum_{f \in F} \PP[ \Pi_0 \in \Theta^{-1}(B)] && \text{By definition} \\
         &= \frac{1}{\abs{F}} \sum_{f \in F} \PP[ \Theta(\Pi_0) \in B] && \\
         &= \PP[X^{-1} \Theta(\Pi_0) \in B].
    \end{align*}
    
\end{example}

The Palm measure has an associated integral equation, which we will refer to as ``the CLMM'', following the convention of \cite{baszczyszyn:cel-01654766}. It is also referred to as ``the refined Campbell theorem'' in \cite{last2011poisson} and \cite{vere2003introduction}, for example. 

\begin{thm}[Campbell-Little-Mecke-Matthes]\label{CLMM}
	Let $\mu$ be a finite intensity point process on $G$ with Palm measure $\mu_0$. Write $\EE$ and $\EE_0$ for the associated integral operators.
    
    If $f : G \times \MMo \to \RR_{\geq 0}$ is a measurable function (\emph{not} necessarily invariant in any way), then
    \[
    	\EE \left[\sum_{x \in \omega} f(x, x^{-1}\omega) \right] = \intensity(\mu) \EE_0 \left[ \int_G   f(x, \omega)  d\lambda(x)\right].
    \]
\end{thm}

The proof of the above theorem is a standard monotone convergence argument, and as such we will only give the first step of the argument and leave the details to the reader. Observe that by definition of the Palm measure, for any $U \subseteq(G)$ of finite volume and any measurable $A \subseteq \MMo$, we have
\[
    \EE\left[\#\{ g\in U \mid g^{-1}\Pi \in A\} \right] = \intensity(\mu) \mu_0(A) \lambda(U).
\]
Now rewrite the integrand on the lefthand side as a sum, and the right hand side as an integral we see:
\[
  \EE\left[\sum_{g \in \Pi} \1[(g, g^{-1}\omega) \in U \times A] \right] = \intensity(\mu) \int_{\MMo} \int_G \1[(g, \omega) \in U \times A]d\lambda(g)d\mu_0(\omega),
\]
and observe that this is exactly the claimed theorem (in slightly different notation) in the case of $f(x,\omega) = \1[(x, \omega) \in U \times A]$. The theorem follows for arbitrary $f$ by the monotone convergence theorem.
\begin{remark}\label{VIF}

If $\nu$ is a point process with $\nu_0 = \mu_0$, then $\nu = \mu$, that is, the Palm measure \emph{determines} the point process.

To see this, we use the existance of a map $\mathscr{V} : [0,1] \times \MMo \to \MM$ with the property that if $\mu$ is \emph{any} point process with Palm measure $\mu_0$, then $\mathscr{V}_*(\text{Leb} \otimes \mu_0) = \mu$. This is a consequence of the \emph{Voronoi inversion formula}, see Section 9.4 of \cite{MR3791470}. 

\end{remark}

\subsection{Unimodularity and the Mass Transport Principle}\label{unimodularity}

The Mass Transport Principle is a powerful tool in percolation theory, see \cite{MR3616205} for an introduction and historical context. For the convenience of the reader, we include a proof of it for our context and in our notation, but no originality is claimed. For further generalisations of the mass transport principle see \cite{kallenberg2011invariant}, \cite{gentner2011palm}, and Chapter 7 of \cite{kallenberg2017random} and for further exposition in the context of point processes see \cite{baszczyszyn:cel-01654766}.

The source and range maps $s, t : \Marrow \to \MMo$ induce a pair of measures on $\Marrow$ defined by
\[
    \muarrow^s(\mathscr{G}) = \int_\MMo \abs{s^{-1}(\omega) \cap \mathscr{G}} d\mu_0(\omega), \text{ and } \muarrow^t(\mathscr{G}) = \int_\MMo \abs{t^{-1}(\omega) \cap \mathscr{G}} d\mu_0(\omega).
\]
In our factor graph interpretation this corresponds to the expected indegree and outdegree of $\mathscr{G}$ respectively, where we view $\mathscr{G}$ as a \emph{directed} rooted graph. To see this, recall that for a rooted configuration $\omega \in \MMo$,
\[
    s^{-1}(\omega) = \{(\omega, g) \in \MMo \times G \mid g \in \omega\} \text{ and } t^{-1}(\omega) = \{(g^{-1}\omega, g^{-1}) \in \MMo \times G \mid g \in \omega \},
\]
and that there is an edge from $0$ to $g$ in $\mathscr{G}(\omega)$ exactly when $(\omega, g) \in \mathscr{G}$, and an edge from $g$ to $0$ exactly when $(g^{-1}\omega, g^{-1}) \in \mathscr{G}$. Thus
\[
    \overrightarrow{\deg}_0({\mathscr{G}(\omega)}) = \abs{s^{-1}(\omega) \cap \mathscr{G}(\omega)} \text{ and } \overleftarrow{\deg}_0({\mathscr{G}(\omega)}) = \abs{t^{-1}(\omega) \cap \mathscr{G}(\omega)}.
\]

\begin{remark}

We have had to adapt notation to suit our purposes. Usually a groupoid would be denoted by a letter like $\mathcal{G}$, and that is the set of arrows. Then its units would be denoted $\mathcal{G}_0$. We have tried to match this up with the necessary notation from point process theory as closely as possible. 

We choose to denote outdegree by an expression like $\overrightarrow{\deg}_0({\mathscr{G}(\omega)})$ instead of $\deg^+_{\mathscr{G}(\omega)}(0)$ as the arrows are more evocative, and the subscript notation becomes very small (as in, for instance, $\deg^+_{\mathscr{G}(\Pi_0)}(0)$. 

\end{remark}

\begin{prop}\label{pmpgroupoid}
    
    If $G$ is \emph{unimodular}, then $\muarrow^s = \muarrow^t$. That is, $(\Marrow, \muarrow)$ forms a discrete pmp groupoid.
    
    Equivalently, if $\Pi_0$ is the Palm version of any point process $\Pi$ on $G$, then
    \[
        \EE\left[ \overrightarrow{\deg}_0({\mathscr{G}(\Pi_0)}) \right] = \EE\left[ \overleftarrow{\deg}_0({\mathscr{G}(\Pi_0)}) \right].
    \]
    
    We will denote by $\muarrow$ this common measure $\muarrow^s = \muarrow^t$.
\end{prop}

\begin{proof}[Proof of Proposition \ref{pmpgroupoid}]
    Fix $U \subseteq G$ of unit volume. We compute:
    \begin{align*}
        \muarrow^s(\mathscr{G}) &= \EE_{\mu_0} \left[ \sum_{g \in \omega} \1[{(\omega, g) \in \mathscr{G}}] \right] && \text{By definition} \\
        &= \EE_{\mu_0} \left[\int_G \1[{x \in U}] \sum_{g \in \omega} \1[{(\omega, g) \in \mathscr{G}}] d\lambda(x) \right] &&  \\
        &= \frac{1}{\intensity \mu} \EE_\mu \left[ \sum_{x \in \omega} \1[{x \in U}] \sum_{g \in x^{-1}\omega} \1[{(x^{-1}\omega, g) \in \mathscr{G}}]   \right] && \text{By the CLLM}    \\
        &= \frac{1}{\intensity \mu} \EE_\mu \left[ \sum_{h \in \omega}  \sum_{g^{-1} \in h^{-1}\omega} \1[{hg^{-1} \in U}] \1[{(gh^{-1}\omega, g) \in \mathscr{G}}]   \right] &&   \\
        &= \EE_{\mu_0} \left[ \int_G \sum_{g^{-1} \in \omega} \1[{hg^{-1} \in U}] \1[{(g\omega, g) \in \mathscr{G}}] d\lambda(h) \right] && \text{By the CLLM}  \\
       &= \EE_{\mu_0} \left[ \int_G \sum_{g \in \omega} \1[{hg \in U}] \1[{(g^{-1}\omega, g^{-1}) \in \mathscr{G}}] d\lambda(h) \right] && \text{By unimodularity}  \\
       &= \EE_{\mu_0} \left[\sum_{g \in \omega} \1[{(g^{-1}\omega, g^{-1}) \in \mathscr{G}}] \int_G  \1[{h \in Ug^{-1}}]  d\lambda(h) \right] &&   \\
     &= \EE_{\mu_0} \left[\sum_{g \in \omega} \1[{(g^{-1}\omega, g^{-1}) \in \mathscr{G}}] \right] && \text{By unimodularity}   \\
       &= \muarrow^t(\mathscr{G}),
    \end{align*}
    as desired, where the first instance of CLMM is applied with
    \[
    f_1(x,\omega) = \1[x \in U] \sum_{g \in \omega} \1[(\omega, g) \in \mathscr{G}], \text{ and}
    \]
    \[
    f_2(x, \omega) = \sum_{g^{-1} \in \omega} \1[xg^{-1} \in U]\1[(g\omega, g) \in \mathscr{G}].
    \]
    in the second instance.\end{proof}

\begin{defn}

The \emph{Palm groupoid} of a point process $\Pi$ with law $\mu$ is $(\Marrow, \muarrow)$. If $\Pi$ is free, then this groupoid is principal, and thus we refer to $\Pi$'s \emph{Palm equivalence relation} $(\MMo, \Rel, \mu_0)$.

\end{defn}

\begin{defn}\label{edgemeasure}

Let $\Pi$ be a point process and $\mathscr{G}$ an \emph{undirected} factor graph of $\Pi$. Its \emph{edge density} is $\EE[ \deg_0(\mathscr{G}(\Pi_0))]$, where $\Pi_0$ is the Palm version of $\Pi$. 

\end{defn}

By the above proposition, if $\mathscr{G}'$ is any \emph{orientation} of $\mathscr{G}$, then the edge density can be expressed as
\[
    \EE[ \deg_0(\mathscr{G}(\Pi_0))] = 2 \EE\left[ \overrightarrow{\deg}_0({\mathscr{G'}(\Pi_0)}) \right].
\]
Speaking properly then, we should be talking of \emph{directed} factor graphs, but for this reason we will often think of the factor graphs as undirected.

\begin{thm}[The Mass Transport Principle]\label{MTP}
    Let $G$ be a unimodular group, and $\Pi$ a point process on $G$ with Palm version $\Pi_0$. Suppose $T : G \times G \times \MM \to \RR_{\geq 0}$ is a measurable function which is \emph{diagonally invariant} in the sense that $T(gx, gy; g\omega) = T(x, y; \omega)$ for all $g \in G$. Then
    \[
         \EE \left[ \sum_{y \in \Pi_0} T(0, y; \Pi_0) \right] = \EE \left[ \sum_{x \in \Pi_0} T(x, 0; \Pi_0) \right].
    \]
\end{thm}

We view $T(x, y; \Pi_0)$ as representing an amount of \emph{mass} sent from $x$ to $y$ when the configuration is $\Pi_0$. Thus the integrand on the lefthand side represents the total mass sent out from the root, and similarly the integrand on the righthand side represents the total mass received by the root.

\begin{proof}[Proof of Theorem \ref{MTP}]
    The mass transport principle follows from Proposition \ref{pmpgroupoid}. 
    
    First, observe that as in the proof of the CLMM, there is an integral equation implied by Proposition \ref{pmpgroupoid} and a monotone convergence argument. To wit,
    \[
        \EE\left[ \sum_{g \in \Pi_0} f(\Pi_0, g) \right] = \EE\left[ \sum_{g \in \Pi_0} f(g^{-1}\Pi_0, g^{-1})\right],
    \]
    where $f : \Marrow \to \RR_{\geq 0}$ is a measurable function (defined $\muarrow$ almost everywhere).
    
    We now apply the integral equation with $f(\omega, g) = T(0, g; \omega)$. Note that
    \[
    f(g^{-1}\omega, g^{-1} = T(0, g^{-1}; g^{-1}\omega) = T(g, 0; \omega),
    \]
    where the second equality is by diagonal invariance. Hence the integral equation for $f$ yields
        \[
        \EE\left[ \sum_{g \in \Pi_0} T(0, g; \Pi_0) \right] = \EE\left[ \sum_{g \in \Pi_0} T(g, 0; \Pi_0)\right],
    \]
    which is exactly the mass transport principle expressed with a differently named integrating variable.
\end{proof}

\begin{remark}\label{palmformula}
One can use the CLMM formula (see Theorem \ref{CLMM}) to express $\muarrow(\mathscr{G})$ without reference to the Palm measure. Let $U \subseteq G$ be of unit volume, and apply the formula to $f(x,\omega) = \1[{x \in U}] \overrightarrow{\deg}_0({\mathscr{G}(\omega)})$, resulting in

\[
    \muarrow(\mathscr{G}) = \frac{1}{\intensity \Pi} \EE \left[ \sum_{x \in \Pi} \1[{x \in U}] \overrightarrow{\deg}_x({\mathscr{G}(\Pi)}) \right]
\]
(note that by equivariance $\overrightarrow{\deg}_0({\mathscr{G}(x^{-1}\omega)}) = \overrightarrow{\deg}_x({\mathscr{G}(\omega)})$).

\end{remark}

As an application of the CLMM, we will find an expression for the Palm version of general thickenings:

\begin{example}[Palm measures of general thickenings]\label{palmofgeneralthickening}
    Suppose one has for each configuration $\omega \in \MM$ and each $g \in \omega$ a measurably defined \emph{finite} subset $F_\omega(g)$ satisfying the following properties:
    \begin{description}
        \item[Monotonicity:] That $g \in F_\omega(g)$,
        \item[Separation:] If $g, h \in \omega$ are \emph{distinct} then $F_\omega(g) \cap F_\omega(h) = \empt$, and
        \item [Equivariance:] For all $\gamma \in G$, we have $F_{\gamma \omega}(\gamma \omega) = \gamma F_\omega(g)$.
    \end{description}
 Then one can define a thickening $\Theta : \MM \to \MM$ by
    \[
        \Theta(\omega) = \bigsqcup_{g \in \omega} F_\omega(g).
    \]
    That is, each point $g \in \omega$ looks at the current configuration, and adds points $F_\omega(g)$ locally to it according to some equivariant rule. Every thickening has this form (see Definition \ref{voronoidefn} and the ensuing discussion). We refer to points of $\omega$ as \emph{progenitors} and points of $F_\omega(g)$ as $g$'s \emph{spawn} in $\omega$. 
    
    It stands to reason that if $\Pi$ is a point process satisfying the above rules almost surely, then $\intensity \Theta(\Pi) = \EE\abs{F_{\Pi_0}(0)} \cdot \intensity \Pi$. Just as in Example \ref{constantthickening} though, this will require unimodularity to prove, this time in the form of the Mass Transport Principle.
    
    Let us identify the thickening with its input/output version. Note that if we compute the Palm version of the latter, then we get it for the former by simply forgetting the labels. Our reason for doing this is simple: we need to be able to identify which points were progenitors and which points are spawn. This is only possible if we use the input/output version, but the downside is that that is more notationally cumbersome. 
    
    We first verify that $\intensity \Theta(\Pi) = \EE\abs{F_{\Pi_0}(0)} \cdot \intensity \Pi$. In order to apply mass transport, we need to know the following fact:
    \[
        \PP[\Theta(\Pi)_0 \in A | 0 \text{ is a progenitor}] = \PP[\Theta(\Pi_0) \in A].
    \]
    This follows from the definitions by similar manipulations to those we've already seen.
    
    With this fact in hand, define\footnote{An advantage of using the input/output version of the thickening is that we can exactly identify who spawned who in a well-defined way.} a transport as follows:
    \[
        T(x, y, \Theta(\Pi)) = \1[x \text{ spawned } y \text{ in } \Theta(\Pi)].
    \]
    Then the total mass received by the root is always one (as everyone is spawned by someone), and hence the expected mass received is one.
    
    The expected mass sent out is
    \[
    \EE_0\left[ \1[0 \text{ is a progenitor}] \cdot \#\{\text{spawn of } 0\}\right] = \PP[0 \text{ is a progenitor}] \cdot \EE[\abs{F_{\Pi_0}(0)}],
    \]
    where $\EE_0$ denotes expectation with respect to the Palm measure of $\Theta(\Pi)$, and the equality follows from the fact above and the definition of conditional probability. 
    
    We have by the definition of progenitor
    \[
        \PP[0 \text{ is a progenitor}] = \frac{\intensity \Pi}{\intensity \Theta(\Pi)},
    \]
    so $\intensity \Theta(\Pi) = \EE\abs{F_{\Pi_0}(0)} \cdot \intensity \Pi$ by the mass transport principle.
    
    We now express the Palm version of $\Theta(\Pi)$ in terms of $\Pi_0$ and $\Theta$. Note that for this to be defined we must assume $\Pi$ has finite intensity and that $\EE[\abs{F_{\Pi_0}(0)}] < \infty$.
    
    Let
    \begin{itemize}
        \item $N$ be a random variable with
        \[
        \PP[N = n] = \frac{n\PP[\abs{F_{\Pi_0}(0)} = n]}{\EE\abs{F_{\Pi_0}(0)}} = \frac{n\PP[0 \text{ spawns } n \text{ points of } \Theta(\Pi_0)]}{\EE\abs{F_{\Pi_0}(0)}},
        \]
        \item $\Upsilon^n$ denote $\Pi_0$ conditioned on the event $\{0 \text{ spawns } n \text{ points of } \Theta(\Pi_0)\}$,
        \item $X$ be a uniformly chosen element of $F_{\Upsilon^n}(0)$ (conditional on $\Upsilon^n$).
    \end{itemize}
    We claim that $X^{-1}\Theta(\Upsilon^N)$ is a Palm version of $\Theta(\Pi)$.
    
    In words, we are sampling from the Palm measure $\Pi_0$ biased\footnote{To see that some kind of size biasing is required, consider the point process $\ZZ + \texttt{Unif}[0,1] \subset \RR$, and define a thickening which leaves points marked $0$ as they are and adds a thousand points tightly packed around points marked $1$ A ``typical point'' of the resulting process should look more like a configuration with a thousand points near the origin, and the size biasing accommodates for this.} towards the configurations that spawn more points, and then applying the thickening and rooting at one of the spawns uniformly at random.
    
    Let $A \subseteq \{\text{Red, Blue, Purple}\}^\MMo$. We find an expression for $\PP[\Theta(\Pi)_0 \in A]$ by using mass transport: define
    \[
        T(x, y; \Theta(\Pi)) = \1[\{x \text{ spawns } y \text{ in } \Theta(\Pi) \} \cap \{y \in \theta^A(\Theta(\Pi)) \}].
    \]
    The expected mass in with respect to $T$ is exactly $\PP[\Theta(\Pi)_0 \in A]$. The expected mass out is
    \begin{align*}
        &\EE\left[\sum_{y \in \Theta(\Pi)_0} T(0, y; \Theta(\Pi)_0) \right] \\
        &=  \EE\left[\1[0 \text{ is a progenitor}] \cdot \#\{0 \text{ spawns } y \text{ with } y \in \theta^A(\Theta(\Pi)_0)\} \right] \\
        &= \PP[0 \text{ is a progenitor}] \EE\left[\#\{0 \text{ spawns } y \text{ with } y \in \theta^A(\Theta(\Pi)_0)\} | 0 \text{ is a progenitor} \right]\\
        &= \frac{1}{\EE\abs{F_{\Pi_0}(0)}} \EE \left[\#\{y \in F_{\Pi_0}(0) \mid y^{-1}\Theta(\Pi_0) \in A\} \right] \\
        &= \frac{1}{\EE\abs{F_{\Pi_0}(0)}} \sum_n \EE \left[\#\{y \in F_{\Pi_0}(0) \mid y^{-1}\Theta(\Pi_0) \in A\} | \abs{F_{\Pi_0}(0)} = n \right] \PP[\abs{F_{\Pi_0}(0)} = n].
    \end{align*}
    We can now match up this expression with our earlier description of $X^{-1}\Phi(\Upsilon^N)$.
    
    Recall that if $Y \subseteq [n]$ is a random subset, then $\EE \abs{Y} = n \PP[X \in Y]$, where $X$ is a uniformly chosen element of $[n]$. 
    
    \begin{align*}
        &\EE\left[\sum_{y \in \Theta(\Pi)_0} T(0, y; \Theta(\Pi)_0) \right] \\
        &= \frac{1}{\EE\abs{F_{\Pi_0}(0)}} \sum_n \EE \left[\#\{y \in F_{\Pi_0}(0) \mid y^{-1}\Theta(\Pi_0) \in A\} | \abs{F_{\Pi_0}(0)} = n \right] \PP[\abs{F_{\Pi_0}(0)} = n] \\
        &= \frac{1}{\EE\abs{F_{\Pi_0}(0)}} \sum_n \EE \left[\#\{y \in F_{\Upsilon^n}(0) \mid y^{-1}\Theta(\Upsilon^n) \in A\} \right] \PP[\abs{F_{\Pi_0}(0)} = n] \\
        &=  \sum_n n \PP[X^{-1}\Theta(\Upsilon^n) \in A] \frac{\PP[\abs{F_{\Pi_0}(0)} = n]}{\EE\abs{F_{\Pi_0}(0)}} \\  
        &= \sum_n  \PP[X^{-1}\Theta(\Upsilon^n) \in A] \PP[N = n] \\ 
        &= \PP[X^{-1} \Theta(\Upsilon^N) \in A],
    \end{align*}
    as desired.
\end{example}

In fact, every thickening can be expressed \`{a} la Example \ref{palmofgeneralthickening}, as we shall now see.

\begin{defn}\label{voronoidefn}

    Let $\omega \in \MM$ be a configuration, and $g \in \omega$ one of its points. The associated \emph{Voronoi cell} is
    \[
        V_\omega(g) = \{ x \in G \mid d(x, g) \leq d(x, h) \text{ for all } h \in \omega \}.
    \]
    The associated \emph{Voronoi tessellation} is the ensemble of closed sets $\{V_\omega(g)\}_{g \in \omega}$.

\end{defn}

Left-invariance of the metric $d$ implies that the Voronoi cells are equivariant in the sense that for all $\gamma \in G$, we have $V_{\gamma \omega}(\gamma g) = \gamma V_\omega(g)$.

Note that discreteness of the configuration implies that the Voronoi tessellation forms a locally finite \emph{cover} of the ambient space by closed sets. We would like to think of these sets as forming a \emph{partition} of the ambient space, but this isn't necessarily true even in the measured sense: the boundaries of the Voronoi cells can have positive volume. For example, let $\Gamma$ be a discrete group and consider $\Gamma \times \{0\} \subset \Gamma \times \RR$. 

Lie groups and Riemannian symmetric spaces essentially avoid this deficiency, as hyperplanes\footnote{Sets of the form $\{x \in X \mid d(x, g) = d(x, h) \}$ for a fixed distinct pair $g, h \in X$.} have zero volume.

So depending on the examples one is interested in one can assume that the Voronoi cells are essentially disjoint (that is, that their intersection is Haar null). If this property is necessary then one can make a small modification to ensure it: we introduce a \emph{tie breaking} function that allows points belonging to multiple Voronoi cells to decide which one they shall belong to. Take any\footnote{Recall that standard Borel spaces are isomorphic if they have the same cardinality.} Borel isomorphism $T : G \to \RR$. Let us define
    \[
    \begin{split}
        V_\omega^T(g) = \{ x \in G \mid \text{for all } h \in \omega \setminus \{g\},\;  d(x, g) < d(x, h)\\
        \text{ or } d(x, g) = d(x,h) \text{ and } T(x^{-1}g) < T(x^{-1}h) \}.
    \end{split}
    \]

Note that these tie-broken Voronoi cells form a \emph{measurable} partition of $G$. That is, we have traded the Voronoi cells being closed for them being genuinely disjoint. The equivariance property $V^T_{\gamma \omega}(\gamma g) = \gamma V^T_\omega(g)$ still holds as well.

If $\Theta : \MM \to \MM$ is a thickening, then we simply define $F_\omega(g) = V_\omega(g) \cap \Theta(\omega)$.

\subsection{Ergodicity and the factor correspondences in the measured category}\label{ergodicity}

In this section we show how to extend the correspondences of Section \ref{borelcorrespondences} to the measured category, which connects the distribution $\mu$ of a point process to its Palm measure $\mu_0$, and objects understood to be defined $\mu$ almost everywhere with those defined $\mu_0$ almost everywhere.

\begin{defn}
    A subset $A \subseteq \MM$ of unrooted configurations is \emph{shift-invariant} if for all $\omega \in A$ and $g \in G$, we have $g\omega \in A$.

	A subset $A_0 \subseteq \MMo$ of rooted configurations is \emph{rootshift invariant} if for all $\omega \in A_0$ and $g \in \omega$, we have $g^{-1}\omega \in A_0$. 
	
	The groupoid $(\Marrow, \muarrow)$ is \emph{ergodic} if every rootshift invariant subset $A \subseteq \MMo$ has $\mu_0(A) = 0$ or $1$.
	
\end{defn}

Note that if $A \subseteq \MM$ is shift-invariant, then $A_0 := A \cap \MMo$ is rootshift invariant, and if $A_0 \subseteq \MMo$ is rootshift-invariant, then $A := GA_0$ is shift invariant. Thus shift-invariant subsets and rootshift-invariant subsets are in bijective correspondence. Moreover:

\begin{prop}\label{transferprinciple}

Let $\mu$ be a point process with Palm measure $\mu_0$. 
\begin{enumerate}
    \item If $A \subseteq \MMo$ is rootshift invariant, then $\mu_0(A) = \mu(GA)$.
    \item If $A \subseteq \MM$ is shift invariant, then $\mu_0(A \cap \MMo) = \mu(A)$.
\end{enumerate}

That is, under the correspondence between rootshift invariant subsets of $\MMo$ and shift invariant subsets of $\MM$, the measures $\mu_0$ and $\mu$ coincide.

In particular, $G \acts (\MM, \mu)$ is ergodic \emph{if and only if} $(\Marrow, \muarrow)$ is ergodic.

\end{prop}

\begin{proof}
    We assume ergodicity and prove the statements about measures. We then prove the general statement by using the ergodic case.
    
    First, suppose $G \acts (\MM, \mu)$ is ergodic, and let $A \subseteq \MMo$ be rootshift invariant. Then for any $U \subseteq G$ of unit volume,
    \begin{align*}
        \mu_0(A) &= \frac{1}{\intensity \mu} \EE_\mu\left[ \#\{g \in U \mid g^{-1}\omega \in A \} \right] && \text{By definition} \\
        &= \frac{1}{\intensity \mu} \EE_\mu\left[ \abs{\omega \cap U} \1[{\omega \in GA}] \right] && \text{By rootshift invariance of } A   \\
        &= \mu(GA) && \text{By ergodicity}.
    \end{align*}
    
    In particular, we see that $\mu_0(A)$ is zero or one, so the equivalence relation is ergodic.
    
    Now suppose $(\MMo, \Rel, \mu)$ is ergodic, and let $A \subseteq \MM$ be shift invariant.
    \begin{align*}
        \mu_0(A \cap \MMo) &= \frac{1}{\intensity \mu} \EE_\mu \left[ \#\{g \in U \mid g^{-1}\omega \in A \cap \MMo \} \right] && \text{By definition} \\
        &= \frac{1}{\intensity \mu} \EE_\mu\left[ \abs{\omega \cap U} \1[{\omega \in A}] \right] && \text{By shift invariance of } A \\
        &= \mu(A) && \text{By ergodicity}.
    \end{align*}

For the general case, we appeal to the ergodic decomposition theorem (see \cite{MR1784210} for a proof):
\begin{thm}

Let $G$ be an lcsc group, and $G \acts (X, \mu)$ a pmp action on a standard Borel space. Then there exists a standard Borel space $Y$ equipped with a probability measure $\nu$ and a family $\{ p_y \mid y \in Y\}$ of probability measures $p_y$ on $X$ with the following properties:
\begin{enumerate}
    \item For every Borel $A \subset X$, the map $y \mapsto p_y(A)$ is Borel, and
    \[
        \mu(A) = \int_Y p_y(A) d\nu(y).
    \]
    \item For every $y \in Y$, $p_y$ is an invariant and ergodic measure for the action $G \acts (X, p_y)$,
    \item If $y, y' \in Y$ are distinct, then $p_y$ and $p_y'$ are mutually singular.
\end{enumerate}

\end{thm}

There is an almost identically stated version of the above theorem for pmp cbers as well. These two decompositions are essentially equivalent, in a way that we shall now discuss.

If $(Y, \nu)$ and $\{p_y \mid y \in Y\}$ is the ergodic decomposition for $G \acts (\MM, \mu)$, then the Palm measures $(p_y)_0$ of the $p_y$ form an ergodic decomposition for $(\MMo, \Rel, \mu_0)$. That is, for all $A \subseteq \MMo$ we have

\[
    \mu_0(A) = \int_Y (p_y)_0(A) d\nu(y).
\]
Applying the previous ergodic case to this yields the general formula.
\end{proof}






\begin{thm}\label{correspondencetheorem}
Let $G$ be a locally compact and second countable group, and $\Pi$ an invariant point process on $G$ with law $\mu$. 

Then associated to this data is an $r$-discrete probability measure preserving groupoid $(\Marrow, \muarrow)$ called \emph{the Palm groupoid} of $\Pi$. It has the following properties:
\begin{itemize}
    \item Thinning maps $\theta : (\MM, \mu) \to \MM$ of $\Pi$ are in correspondence with Borel subsets $A$ of the unit space $\MMo$ of the Palm groupoid defined $\mu_0$ almost everywhere,
    \item Factor $\Xi$-markings $\mathscr{C} : (\MM, \mu) \to \Xi^\MM$ are in correspondence with Borel $\Xi$-valued maps $P$ defined on the unit space $\MMo$ of the Palm groupoid defined $\mu_0$ almost everywhere, and
    \item Factor graphs $\mathscr{G} : (\MM, \mu) \to \graph(G)$ of $\Pi$ are in correspondence with Borel subsets $\mathscr{A}$ of the arrow space $\Marrow$ of the Palm groupoid defined $\muarrow$ almost everywhere.
\end{itemize}

\end{thm}

The Palm measure is well studied, but the equivalence relation structure seems to have been mostly overlooked. One can find two direct references to it: Example 2.2 in a paper of Avni \cite{avni2005spectral} and a question of Bowen in \cite{bowen2018all} (specifically, (Questions and comments, item 1).

We now prove Theorem \ref{correspondencetheorem}, building on Section \ref{borelcorrespondences}. The task here is to verify that under the correspondence, objects which are equal almost everywhere with respect to the point process are equal almost everywhere with respect to the Palm measure, and vice versa.

\begin{lem}\label{extensionlemma}
Let $\mu$ be a point process on $G$ with Palm measure $\mu_0$, and $X$ a Borel $G$-space.

Let $\Phi, \Phi' : \MM \to X$ be an equivariant Borel map. Then
\[
    \Phi = \Phi' \;\; \mu \text{ almost everywhere \emph{if and only if} } \restr{\Phi}{\MMo} = \restr{\Phi'}{\MMo}\;\; \mu_0 \text{ almost everywhere}.
\]
\end{lem}

\begin{proof}
Observe that by equivariance the sets
\[
    \{ \omega \in \MM \mid \Phi(\omega) = \Phi'(\omega)\} \text{ and } \{ \omega \in \MMo \mid \Phi(\omega) = \Phi'(\omega) \}
\]
are shift invariant and rootshift invariant respectively. So by Proposition \ref{transferprinciple} one is $\mu$-sure if and only if the other is $\mu_0$-sure, as desired.
\end{proof}

\begin{proof}[Proof of Theorem \ref{correspondencetheorem}]

The method is essentially the same for thinnings and for markings, so we will just prove the thinning statement. To that end, let $\theta : (\MM, \mu) \to \MM$ be a thinning. Note that by our assumption that $\theta$ is equivariant, we have
    \[
        \{ \omega \in \MM \mid \theta(\omega) \subseteq \omega \} \text{ has } \mu \text{ measure one}.
    \]
    This is a shift invariant set, so by Proposition \ref{transferprinciple} we have
    \[
        \{ \omega \in \MMo \mid \theta(\omega) \subseteq \omega \} \text{ has } \mu_0 \text{ measure one}.
    \]
We are now able to define $A = \{\omega \in \MMo \mid 0 \in \theta(\omega) \}$, and this will be our desired subset of $(\MMo, \mu_0)$. 

It follows from equivariance that the thinning $\theta^A$ associated to $A$ satisfies
\[
    \restr{\theta^A}{\MMo} = \restr{\theta}{\MMo} \;\; \mu_0 \text{ almost everywhere,}
\]
so by Lemma \ref{extensionlemma} we have $\theta^A = \theta$ ($\mu$ almost everywhere).

It remains to verify that if $A = B$ $\mu_0$ almost everywhere (that is, that $\mu_0(A \triangle B) = 0$, then $\theta^A = \theta^B$ ($\mu$ almost everywhere).

Recall\footnote{This is a general fact about nonsingular cbers, and it follows from the fact that they can all be generated by actions of \emph{countable} groups.} that the \emph{saturation} of $A \triangle B$
\[
    [A \triangle B] = \{ g^{-1}\omega \in \MMo \mid \omega \in A \triangle B \text{ and } g \in \omega \}
\]
is $\mu_0$ null if $A \triangle B$ is $\mu_0$ null.

Observe that for $\omega \not\in [A \triangle B]$ we have $\theta^A(\omega) = \theta^B(\omega)$, and hence $\restr{\theta^A}{\MMo} = \restr{\theta^B}{\MMo}$ $\mu_0$ almost everywhere, and we are finish by again applying Lemma \ref{extensionlemma}.

If $\mathscr{G}$ is a factor graph of $\mu$, then in the same fashion we see that it has a well-defined restriction to $(\MMo, \mu_0)$. We then define
\[
    \mathscr{A} = \{ (\omega, g) \in \MMo \times G \mid (0, g) \in \mathscr{G}(\omega) \}.
\]

We must verify that if $\mathscr{A}, \mathscr{B} \subseteq \Marrow$ are subsets with $\muarrow(A \triangle B) = 0$, then their associated factor graphs $\mathscr{G}^{\mathscr{A}}$ and $\mathscr{G}^{\mathscr{B}}$ are equal $\mu$ almost everywhere. This assumption states
\[
    \int_{\MMo} \# \{g \in \omega \mid (\omega, g) \in \mathscr{A} \triangle \mathscr{B} \} d\mu_0(\omega) = 0
\]
and hence the integrand is zero $\mu_0$ almost everywhere. By again considering the saturation of sets, we see that
\[
    \mu_0( \{\omega \in \MMo \mid \text{ for all } g \in \omega, g^{-1}\omega \in \mathscr{A} \triangle \mathscr{B} \}) = 0,
\]
from which the argument finishes as in the case of thinnings.
\end{proof}

\subsection{Every free action is a point process, and the cross-section perspective}\label{crosssectionappendix}

We have taken the perspective that point processes are an \emph{intrinsically interesting} class of pmp actions of lcsc groups to study. They are also a fairly general class: in this section we will prove Theorem \ref{minden}, that \emph{every} free and pmp action of a \emph{nondiscrete} lcsc group $G$ on a standard Borel measure space $(X, \mu)$ is abstractly isomorphic to a finite intensity point process. 

This is similar to the following fact: let $\Gamma \acts (X, \mu)$ be a pmp action of a discrete group $\Gamma$. The \emph{symbolic dynamics} of this action is the map
\begin{align*}
    &\Sigma : (X, \mu) \to X^\Gamma \\
    &\Sigma_x(\gamma) = \gamma^{-1}x.
\end{align*}
This is an injective and equivariant map, so we may identify the action $\Gamma \acts (X, \mu)$ with the invariant colouring action $\Gamma \acts (X^\Gamma, \Sigma_* \mu)$.

In this way, we see that all pmp actions of discrete groups are isomorphic to invariant colourings\footnote{If desired, one can fix a Borel isomorphism $X \cong [0,1]$ so that the colouring space is the same for all actions}.

A standard technique in the study of free pmp actions of lcsc groups is to analyse their associated \emph{cross-sections}. This will gives an analogue of symbolic dynamics for nondiscrete groups.

\begin{defn}
Let $G \acts (X, \mu)$ be a pmp action on a standard Borel measure space $(X, \mu)$.

A \emph{discrete cross-section} for the action is a Borel subset $Y \subset X$ such that for $\mu$-every $x \in X$ the set $\{g \in G \mid g^{-1}x \in Y \}$ is a closed and discrete non-empty subset of $G$. 
\end{defn}

\begin{example}
The set $\MMo \subset \MM$ is a discrete cross-section \emph{for all} non-empty point process actions $G \acts (\MM, \mu)$.
\end{example}

There is a sense in which this $\MMo$ is the \emph{only} cross-section, which we now discuss.

Fix a discrete cross-section $Y$ for $G \acts X$. We associate to this data two maps
\begin{align*}
    &\mathcal{V} : (X, \mu) \to \MM &&  \mathscr{V} : (X, \mu) \to Y^\MM  \\
    &\mathcal{V}_x = \{g \in G \mid g^{-1}x \in Y \} &&  \mathscr{V}_x = \{(g, g^{-1}x) \in G \times Y \mid g^{-1}x \in Y \}.
\end{align*}

These are equivariant maps, and the second one is always injective. In particular\footnote{Recall that an \emph{injective} map between standard Borel spaces is always a Borel isomorphism onto its image}, we see that every action which admits a cross-section also admits a point process factor, and is isomorphic to a \emph{marked} point process.

Note that $\mathcal{V}^{-1}(\MMo) = Y$. In this way we see that \emph{a discrete cross-section is the same thing as an unmarked point process factor}.

\begin{remark}[Terminological discussion]

If $\mathcal{P}(\omega)$ is some property of discrete subsets $\omega$ in $G$, then we can investigate discrete cross-sections of actions $G \acts (X, \mu)$ such that the associated subset $\mathcal{V}_x$ satisfies $\mathcal{P}$ for $\mu$ almost every $x \in X$. 

For instance, $\mathcal{P}(\omega)$ might be the property ``$\omega$ is uniformly discrete'' or ``$\omega$ is a net'' (see Definition \ref{metricdefs} for the meaning of these terms). We will refer to a discrete cross-section such that $\mathcal{P}(\mathcal{V}_x)$ is satisfied for $\mu$ almost every $x \in X$ as a \emph{$\mathcal{P}$ cross-section}.

Note that if $G \acts (\MM, \mu)$ is the Poisson point process action, then $\MMo$ is \emph{not} a lacunary cross-section. It is for this reason that we feel the terminology should be modified slightly.

\end{remark}

\begin{thm}[Forrest\cite{MR417388}, see also \cite{MR3335405}]\label{crosssectionsexist}

Every free and \emph{nonsingular}\footnote{Recall that an action is \emph{nonsingular} if it preserves null sets, that is, if $\mu(A) = 0$ then $\mu(gA) = 0$ for all $g \in G$.} action of an lcsc group on a standard probability space admits a discrete cross-section. Moreover, the cross-section can be chosen to be uniformly separated and even a net.

\end{thm}

One sees that Theorem \ref{minden} is true by applying the above theorem with the unmarking technique of Proposition \ref{abstractlyisom}.

\begin{remark}

In fact, cross-sections of actions are known to exist in great generality, see \cite{kechris2019theory} for further examples.

Our keen interest in \emph{free} actions is because it allows us to identify the orbit $Gx$ of any point $x \in X$ with $G$ itself. One can run into issues in the absence of this.

For instance, let $\RR \times \RR$ act on $\{\bullet\} \times \RR/\ZZ$ diagonally, where $\{\bullet\}$ denotes a singleton with trivial action.

Then $\{ (\bullet, 0) \}$ is a lacunary cross-section for the action. If we try to construct a map $\mathcal{V}$ as before, then we would map $(\bullet, x) \in \{\bullet\} \times \RR/\ZZ$ to the subset of $\RR^2$
\[
    \mathcal{V}_{(\bullet, x)} = \RR \times \{ x + \ZZ \}.
\]
In this way one has constructed a \emph{random closed set} as a factor of the action, but it is not a point process. In fact, it is possible to view an \emph{arbitrary} pmp action as a kind of ``bundle'' of point processes over the various homogeneous spaces $G/H$, where $H$ ranges over the closed subgroups of $G$, but we will not explore this further. 
\end{remark}

The following theorem is described as folklore in \cite{MR3335405}:

\begin{thm}[Folklore theorem, see Proposition 4.3 of \cite{MR3335405}]

Let $G$ be a unimodular lcsc group, and $G \acts (X, \mu)$ a pmp action on a standard Borel space. Fix a lacunary cross-section $Y \subset X$ for the action. Then:
\begin{enumerate}
    \item The orbit equivalence relation of $G \acts X$ restricts to a cber $\Rel$ on $Y$.
    \item There exists an $\Rel$-invariant probability measure $\nu$ on $Y$.
    \item The action $G \acts (X, \mu)$ is ergodic if and only if the cber $(Y, \Rel, \nu)$ is ergodic.
    \item The group $G$ is noncompact if and only if the cber is aperiodic $\nu$ almost everywhere.
    \item The group $G$ is amenable if and only if the cber $(Y, \Rel, \nu)$ is amenable.
\end{enumerate}
\end{thm}

The mathematical content of Theorem \ref{correspondencetheorem} can be viewed as a rediscovery of the above theorem with different proofs, together with interpretation of factor constructions as objects living on the Palm groupoid.

\begin{question}
Is there a more point process theoretic method to construct discrete cross-sections of free pmp actions?
\end{question}

We have seen that if $G \acts (X, \mu)$ is a free pmp action, then cross-sections are the same thing as point process factor maps, and that every choice of cross-section gives an isomorphic representation of the action as a marked point process. These ideas can be combined.

Suppose $\Phi : (X, \mu) \to \MM$ is an equivariant factor map. Then $Y = \Phi^{-1}(\MMo)$ is a cross-section for the action $G \acts (X, \mu)$. We also have the isomorphism $\mathscr{V} : (X, \mu) \to Y^\MM$. These can be combined, and we see that the map $\Phi \circ \mathscr{V}^{-1} : Y^\MM \to \MM$ is simply the map that forgets labels.

In other words, every extension of a point process is just the point process with an enriched mark space.

\section{The cost of a point process}\label{cost}

\subsection{Definition and monotonicity for factors}

Our goal is to extend the notion of cost for pmp cbers to point processes. For further background on cost, see \cite{gaboriau2000cout}, \cite{gaboriau2010cost}, \cite{gaboriau2016around}, and \cite{kechris2004topics}.

Informally speaking, the \emph{cost} of a point process is the ``cheapest'' way to wire it up. We look at all \emph{connected} factor graphs of the process and compute the expected degree at the origin in the Palm version. This is then suitably normalised to give an isomorphism invariant.

\begin{defn}\label{groupoidcostdefn}

Let $\Pi$ be a point process on $G$ (possibly marked) with finite but non-zero intensity. Its \emph{groupoid cost} is defined by
\[
    \cost(\Pi) - 1 = \intensity \mu \cdot \inf_{\mathscr{G}} \left\{ \frac{1}{2}\EE\left[\deg_0{\mathscr{G}(\Pi_0)}\right] - 1 \right\},
\]
where the infimum is taken over all connected factor graphs $\mathscr{G}$ of $\Pi$ and $\Pi_0$ denotes the Palm version of $\Pi$.
Equivalently by Remark \ref{palmformula},
\[
    \cost(\Pi) - 1 = \inf_{\mathscr{G}}\left\{  \frac{1}{2}\EE\left[ \sum_{x \in U \cap \Pi} \deg_x{\mathscr{G}(\Pi)} \right] \right\} - \intensity(\Pi),
\]
where $U$ is a set of unit volume in $G$.
\end{defn}

\begin{remark}

The cost respects the ergodic decomposition of a process, and so for this reason it suffices to consider ergodic processes.

\end{remark}

\begin{defn}
    The \emph{cost} of a group is the infimum of the cost of all its free point processes.

    A group is said to have \emph{fixed price} if all of its \emph{essentially free} point processes have the same cost.
\end{defn}
 At the time of writing there are no groups known that do not have this property.

\begin{remark}

We sometimes refer to Definition \ref{groupoidcostdefn} as \emph{groupoid} as it can be thought of as the infimal ``size'' of a generator of the groupoid $(\Marrow, \muarrow)$, in a way that we now discuss.

Recall that (directed) factor graphs are in correspondence with subsets of $\Marrow$. We identify objects under this correspondence. 

One defines the \emph{product} of two factor graphs $\mathscr{G}, \mathscr{H} \subset \Marrow$ by taking all well-defined products. More explicitly,
\[
    \mathscr{G} \cdot \mathscr{H} = \{ (\omega, gh) \in \Marrow \mid (\omega, g) \in \mathscr{G} \text{ and } (g^{-1}\omega, h) \in \mathscr{H} \}.
\]
From the factor graph viewpoint, the edges of $\mathscr{G} \cdot \mathscr{H}$ are those pairs of vertices that can be reached by following an edge of $\mathscr{G}$ and then an edge of $\mathscr{H}$.

A \emph{Borel generator} of $\Marrow$ is a Borel factor graph $\mathscr{G}$ such that
\[
    \langle \mathscr{G} \rangle := \bigcup_{n} \mathscr{G}^n = \Marrow.
\]
In other words, it is a \emph{connected} factor graph.

If $\Pi$ is a point process with law $\mu$, then a generator of the measured groupoid $(\Marrow, \muarrow)$ is a factor graph $\mathscr{G}$ such that
\[
    \muarrow(\Marrow \setminus \langle \mathscr{G} \rangle) = 0.
\] 
In other words, it is a factor graph which is \emph{connected almost surely}.

With these definitions, one can equivalently rephase the probabilistic definition of the cost of $\Pi$ as

\[
    \cost(\Pi) - 1 = \intensity(\Pi) \cdot \inf_{\mathscr{G}}\left\{ \muarrow(\mathscr{G}) - 1\right\},
\]
where $\mathscr{G}$ runs over all generators of $(\Marrow, \muarrow)$. 
\end{remark}

\begin{example}
    
    If $\Pi$ is the lattice shift corresponding to $\Gamma < G$, then \[
        \cost(\Pi) = 1 +  \frac{d(\Gamma) - 1}{\covol(G / \Gamma) },
    \]
    where $d(\Gamma)$ denotes the \emph{rank} of $\Gamma$, that is, its minimum number of generators. To see this, observe that by equivariance a factor graph of the lattice shift is determined by a \emph{single} subset $S \subset \Gamma$, and connects $x \in \Pi$ to all $xs \in \Pi$ for $s \in S$. The graph is connected exactly when $S$ generates $\Gamma$. The formula then follows from the definition of cost.
\end{example}

\begin{remark}

 In a concurrently appearing work\cite{mellick2021palm} by the second author, it is shown that the Palm equivalence relation of any free point process on an amenable group is hyperfinite almost everywhere. It follows that amenable groups (in particular $\RR^n$) have fixed price one.

We will show that all groups of the form $G \times \RR$ have fixed price one. This gives an alternative proof that $\RR^n$ has fixed price one. 

It would be interesting to see a ``direct'' proof of this fact. That is, to exhibit \emph{reasonably explicit} connected factor graphs that have cost less than $1 + \e$ for every $\e > 0$.

In \cite{coupier20132d} an explicit factor graph of the Poisson point process on $\RR^2$ is described and shown to be a connected and one-ended tree. It follows that it has cost one.

\end{remark}

\begin{lem}\label{costmonotone}

Let $\Pi$ be a point process of finite intensity, and $\Phi$ a factor map of $\Pi$ such that $\Phi(\Pi)$ has finite intensity. Then
\[
    \cost(\Pi) \leq \cost(\Phi(\Pi)).
\]

Thus cost is \emph{monotone} for factors.

\end{lem}

\begin{cor}

If $\mu$ and $\nu$ are finite intensity point processes that factor onto each other, then $\cost(\mu) = \cost(\nu)$. In particular, the cost of $\mu$ only depends on its isomorphism class as an action.

\end{cor}

\begin{proof}[Proof of Lemma \ref{costmonotone}]
    Recall from Remark \ref{factorsdecompose} that $\Phi$ decomposes as the composition of a thinning $\pi$ and a thickening $\Theta^\Phi$. We prove
    \[
        \cost(\Pi) \leq \cost(\Theta^\Phi(\Pi)) \leq \cost(\pi(\Theta^\Phi(\Pi))) = \cost(\Phi(\Pi)),
    \]
    where the last equality holds as $\Phi = \pi \circ \Theta^\Phi$.
    
    We prove the second inequality first, as it is simpler. For this we use the non-Palm definition of cost.
    
    To that end, let $\mathscr{G}$ be a graphing of $\Phi(\Pi)$ that $\e$-computes the cost, that is, with
    \[
    \EE\left[\sum_{x \in U \cap \Phi(\Pi)} \overrightarrow{\deg}_x{\mathscr{G}(\Phi(\Pi))} \right]- \intensity (\Phi(\Pi)) \leq \cost(\Phi(\Pi)) - 1 + \e.
    \]
    We will use it to define a graphing $\mathscr{H}$ of the thickened process $\Theta^\Phi(\Pi)$. Recall that this process has three types of points: red, purple, and blue.
    
    Let $\mathscr{N}$ be the factor graph of $\Theta^\Phi(\Pi)$ that connects each red point $x$ to its nearest blue neighbour. If this is not well-defined, then we use the tie-breaking function $T : G \to \RR$ of Section \ref{voronoidefn} to make it so in an equivariant way. 
    
    That is, if $y_1, y_2, \ldots, y_n$ are the (finitely many!) blue points of $\Theta^\Phi(\Pi)$ that are closest to $x$, then let $y$ be the element that minimises $T(x^{-1}y_i)$ and add in a directed edge $x \to y$ to $\mathscr{N}$.
    
    We can view $\mathscr{G}$ as defining a factor graph on $\Theta^\Phi(\Pi)$, which lives on the blue and purple points. 
    
    Now let $\mathscr{H}(\Theta^{\Phi}(\Pi)) = \mathscr{G}(\Phi(\Pi)) \sqcup \mathscr{N}(\Theta^\Phi(\Pi))$. This is connected as an undirected graph, so by the definition of cost:
    \begin{align*}
        &\cost(\Theta^\Phi(\Pi)) - 1 \leq \EE\left[\sum_{x \in \Theta^\Phi(\Pi) \cap U}\overrightarrow{\deg}_x{\mathscr{H}(\Theta^\Phi(\Pi))}\right] - \intensity(\Theta^\Phi(\Pi)) \\
        &= \EE\left[\sum_{x \in U \cap \Pi \setminus \Phi(\Pi)} 1 + \sum_{x \in U \cap \Phi(\Pi)} \overrightarrow{\deg}_x{\mathscr{G}(\Phi(\Pi))} \right] - \intensity (\Pi \setminus \Phi(\Pi)) - \intensity (\Phi(\Pi)) \\
        &= \EE\left[\sum_{x \in U \cap \Phi(\Pi)} \overrightarrow{\deg}_x{\mathscr{G}(\Phi(\Pi))} \right]- \intensity (\Phi(\Pi)) \\
        &\leq \cost(\Phi(\Pi)) - 1 + \e.
    \end{align*}
    
    As $\e$ was arbitrary, this proves the second inequality.
    
    For the other inequality, we use the explicit description of the Palm measure as in Example \ref{palmofgeneralthickening} and the Palm definition of cost.

    The idea of the proof is: we have a graphing defined on a larger subset, and we must push it onto a smaller subset somehow. We will simply transfer all edges of $\Theta^\Phi(\Pi)$ to $\Pi$ along the Voronoi cells.

    For $g \in \Pi$, let $F_\Pi(g) = V_\Pi(g) \cap \Theta^\Phi(\Pi)$.

    Let us call a graphing $\mathscr{G}$ of $\Theta^\Phi(\Pi)$ \emph{starlike} if for all $g \in \Pi$ and $x \in F_\Pi(g)$, we have $(g,x) \in \mathscr{G}$. If $\mathscr{G}$ is any graphing, then we can perturb it to find a starlike graphing of the same edge measure. Let us take this for granted for now and see how the proof concludes.
    
    Let $\mathscr{G}$ be a starlike graphing of $\Theta^\Phi(\Pi)$ that $\e$-computes the cost. Let us define a graphing $\mathscr{H}$ of $\Pi$ as follows: join $x, y \in \Pi$ by an edge in $\mathscr{H}(\Pi)$ if there exists $x' \in F_\Pi(x)$ and $y' \in F_\Pi(y)$ such that $x'$ and $y'$ are connected by an edge in $\Theta^\Phi(\Pi)$.

    When we push $\mathscr{G}$ onto $\Pi$, some edges get killed. For instance, if two Voronoi cells have many edges between them, then some get killed. By assuming that the graphing is \emph{starlike} we are guaranteed to kill enough edges. In particular, we kill $\abs{F_\Pi(g)} - 1$ edges at each $g \in \Pi$.
    
    To make the proof more legible, we write $I_\Pi = \intensity(\Pi)$ and $I_{\Theta} = \intensity(\Theta^\Phi(\Pi))$, so that $I_\Theta = I_\Pi \cdot \EE[F_{\Pi_0}(0)]$.
    
    We compute its expected outdegree as follows:
    \begin{align*}
        &I_\Pi \cdot \EE\left[\overrightarrow{\deg}_0{\mathscr{H}(\Pi_0)} - 1\right] \leq I_\Pi \cdot \EE\left[\sum_{x \in F_{\Pi_0}(0)}\overrightarrow{\deg}_x{\mathscr{G}(\Theta^\Phi(\Pi_0))} - \abs{F_{\Pi_0}(0)}\right] \\
        &= I_\Pi \cdot \EE\left[\sum_{x \in F_{\Pi_0}(0)}\overrightarrow{\deg}_x{\mathscr{G}(\Theta^\Phi(\Pi_0))} \right] - I_\Pi\cdot \EE\abs{F_{\Pi_0}(0)} \\
        &= I_\Pi \cdot \EE\left[\sum_{x \in F_{\Pi_0}(0)}\overrightarrow{\deg}_x{\mathscr{G}(\Theta^\Phi(\Pi_0))} \right] - I_\Theta.
    \end{align*}
    We now work on this first term.
    \begin{align*}
    &I_\Pi \cdot \EE\left[\sum_{x \in F_{\Pi_0}(0)}\overrightarrow{\deg}_x{\mathscr{G}(\Theta^\Phi(\Pi_0))} \right] = \frac{I_\Theta}{\EE \abs{F_{\Pi_0}(0)}} \EE\left[\sum_{x \in F_{\Pi_0}(0)}\overrightarrow{\deg}_x{\mathscr{G}(\Theta^\Phi(\Pi_0))} \right]\\
    &= \sum_{k \geq 1}\frac{I_\Theta}{\EE \abs{F_{\Pi_0}(0)}} \EE\left[\sum_{x \in F_{\Pi_0}(0)}\overrightarrow{\deg}_x{\mathscr{G}(\Theta^\Phi(\Pi_0))} \Big| \abs{F_{\Pi_0}(0)} = k \right] \PP[\abs{F_{\Pi_0}(0)} = k] \\
    &= I_\Theta \sum_{k \geq 1} \EE\left[\frac{1}{\abs{F_{\Pi_0}(0)}}\sum_{x \in F_{\Pi_0}(0)}\overrightarrow{\deg}_x{\mathscr{G}(\Theta^\Phi(\Pi_0))} \Big| \abs{F_{\Pi_0}(0)} = k \right] \PP[\abs{F_{\Pi_0}(0)} = k] \\
    &= I_\Theta \EE\left[\overrightarrow{\deg}_0(\mathscr{G}(\Theta^\Phi(\Pi)_0))\right],
    \end{align*}
    where we use the explicit description of the Palm measure of a general thickening proven in Example \ref{palmofgeneralthickening}. Thus
    \[
        I_\Pi \cdot \EE\left[\overrightarrow{\deg}_0{\mathscr{H}(\Pi_0)} - 1\right] \leq I_\Theta \EE\left[\overrightarrow{\deg}_0(\mathscr{G}(\Theta^\Phi(\Pi)_0)) - 1\right]
    \]
    proving $\cost(\Pi) \leq \cost(\Theta^\Phi(\Pi))$, as desired.
    
At last, we must show how to perturb graphings to be starlike. The idea is simple: if some $g \in \Pi$ is not starlike, then there is some $x \in F_\Pi(g)$ such that $(g, x) \not\in \mathscr{G}$. However, there must be \emph{some} path from $g$ to $x$ in $\mathscr{G}$ so we pinch an edge from that path and thus rob Peter to pay Paul. In this way we can improve a given factor graph to be more starlike. By iterating in an appropriate way we can construct the desired factor graph.

Let
\[
    \Pi' = \bigcup_{g \in \Pi} \{ h \in F_\Pi(g) \mid h \neq g \text{ and } (g, h) \not\in \mathscr{G} \}
\]
denote the subprocess of points that violate starlikeness. 

The edges of $\mathscr{G}$ are of three kinds according to how they interface with the Voronoi cells of $\Pi$:
\begin{description}
    \item[Starlike edges,] those of the form $(g, h)$ where $g \in \Pi$ and $h \in F_\Pi(g)$,
    \item[Intracell edges,] those of the form $(h, h')$ where $h, h' \in F_\Pi(g)$ for some $g \in \Pi$ with neither of $h$ or $h'$ being $g$, and
    \item[Crossing edges,] those of the form $(h, h')$ with $h \in F_\Pi(g)$ and $h' \in F_\Pi(g')$ with $g, g' \in \Pi$ and $g \neq g'$.
\end{description}

\begin{figure}[h]\label{edgesexample}
\includegraphics[scale=0.4]{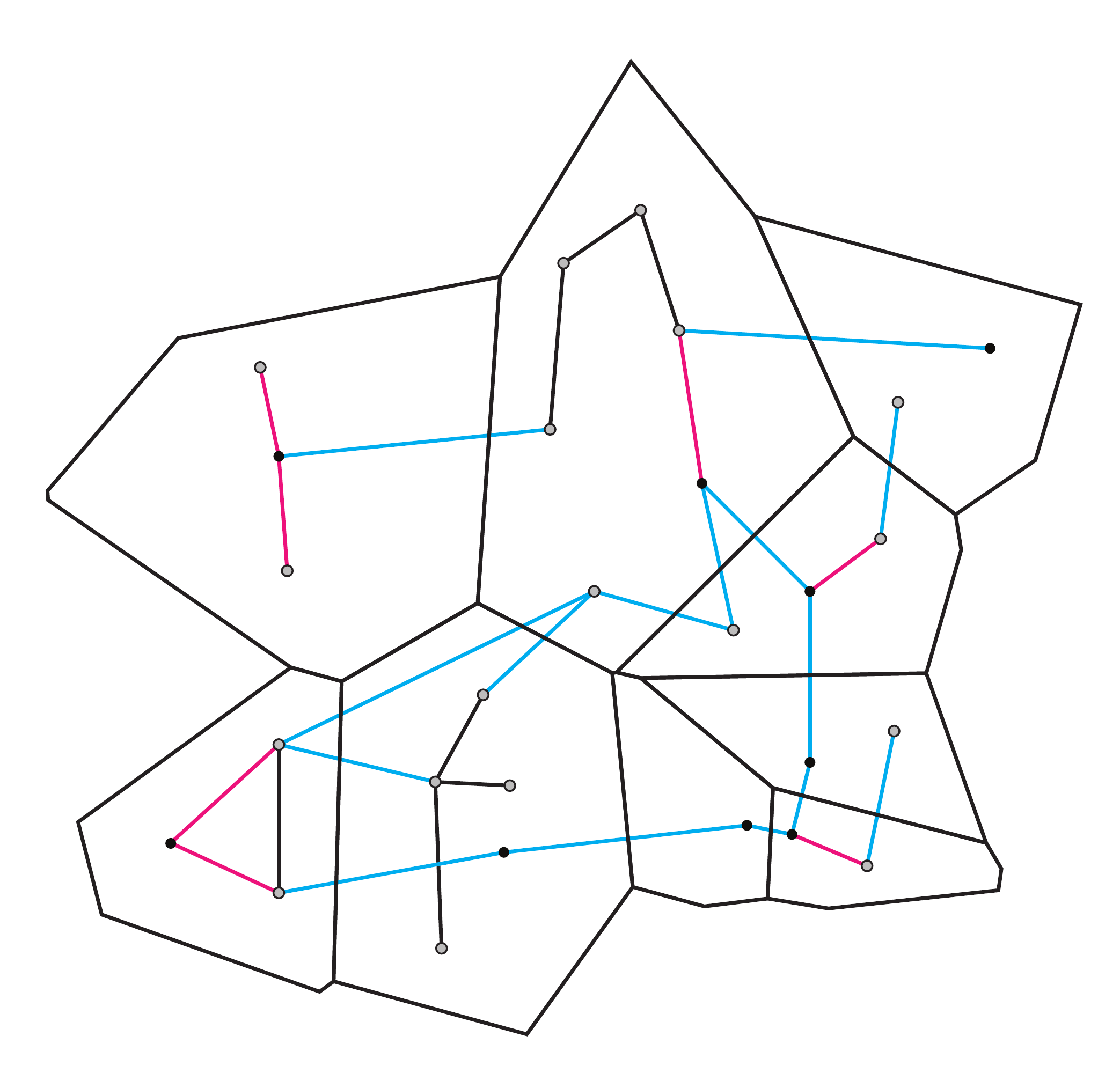}
\centering
\caption{A chunk of a point process, with starlike edges coloured magenta, intracell edges black, and crossing edges cyan.}
\end{figure}

We consider the space\footnote{By constructing an appropriate subset of a configuration space, one can encode these graphs as a standard Borel space.} $\mathcal{G}$ of marked factor graphs of $\Pi$ with the following properties:
\begin{itemize}
    \item They are simply $\mathscr{G}$ as an unmarked graph,
    \item Points $h$ of $\Pi'$ receive \emph{either} the blank mark $\bullet$,
    \item \emph{or} they are marked by a non backtracking path in $\mathscr{G}$ from $h$ to $g$, where $h \in F_\Pi(g)$, and one crossing or intracell edge of this path is coloured red, and
    \item each red edge appears in at most one of the paths of points of $\Pi'$.
\end{itemize}

These factor graphs are basically rewiring rules for $\mathscr{G}$. If $\mathscr{H} \in \mathcal{G}$, then each point of $\Pi'$ that receives a path label in $\mathcal{H}$ replaces the red crossing edge with its starlike edge (see Figure \ref{rewired}). This is an equivariant, measurable, and deterministic rule, so defines a factor graph of $\Pi$.

\begin{figure}[h]\label{rewired}
\includegraphics[scale=0.5]{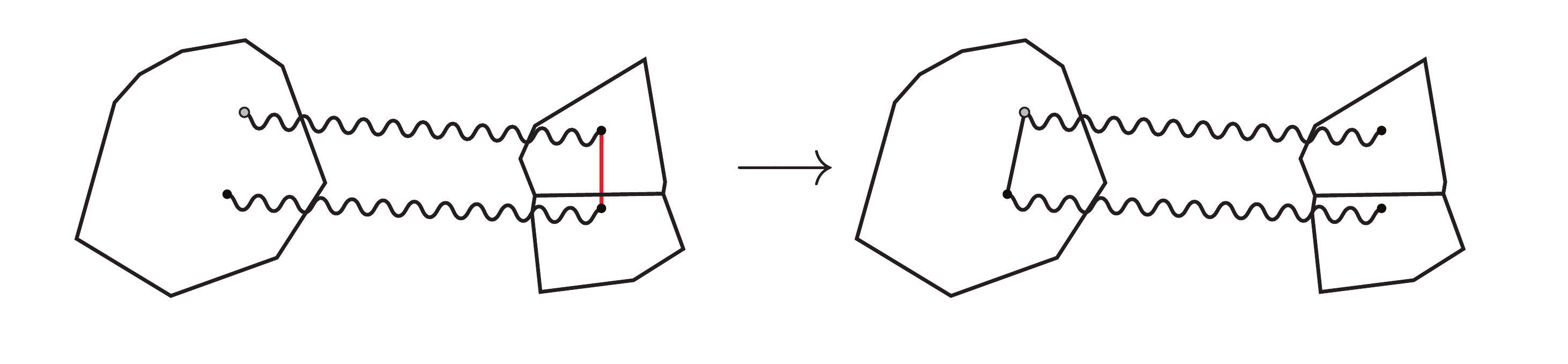}
\centering
\caption{A single rewiring move. We stress that this move must be taken be \emph{all} chosen points simultaneously, and therein lies the rub.}
\end{figure}

Note that this rewiring doesn't change the edge measure of the graph (we rob Peter to pay Paul), and it remains connected.

Let
\[
    \iota(\mathscr{H}) = \text{the intensity of } \Pi' \text{ points that receive path labels in } \mathscr{H}
\]
and
\[
    f(\mathscr{H}) = \sup\left\{\iota(\mathscr{H}') \mid \mathscr{H}' \in \mathcal{G} \text{ and } \mathscr{H} \preceq \mathscr{H}' \right\},
\]
where we declare $\mathscr{H} \preceq \mathscr{H}'$ if every path label in $\mathscr{H}$ is present in $\mathscr{H}'$. That is, $\mathscr{H}'$ has simply replaced $\bullet$ labelled points in $\mathscr{H}$ by path labels.

\begin{claim}
    There is a maximal element $\mathscr{G}_\infty$ (with respect to $\preceq$) of $\mathcal{G}$, and the rewiring of $\mathscr{G}$ associated to it is starlike.
\end{claim}

It's easy to find the maximal element. Choose $\mathscr{G}_1$ such that
\[
    f(\mathscr{G}) \leq \iota(\mathscr{G}_0) + 1,
\]
and then inductively choose $\mathscr{G}_{n+1}$ such that
\[
    f(\mathscr{G}_n) \leq \iota(\mathscr{G}_{n+1}) + \frac{1}{n}
\]
Let $\mathscr{G}_\infty$ denote the ``union'' of the $\mathscr{G}_n$, where we declare that path labels trump $\bullet$ labels.

If $\mathscr{H} \in \mathcal{G}$ is a factor graph with $\mathscr{G}_\infty \preceq \mathscr{H}$, then $\mathscr{G}_n \preceq \mathscr{H}$ for all $n$, so
\[
   \iota(\mathscr{H}) \leq f(\mathscr{G}_n) \leq \iota(\mathscr{G}_n) + \frac{1}{n} \to \iota(\mathscr{G}_\infty).
\]
We conclude that $\mathscr{G}_\infty = \mathscr{H}$ almost surely since one process is a subset of the other.

We will now prove that the rewiring associated to $\mathscr{G}_\infty$ is starlike by contradiction. Using the assumption we construct $\mathscr{H} \in \mathcal{G}$ with $\mathscr{G}_\infty \preceq \mathscr{H}$ and $\iota(\mathscr{G}_\infty) < \iota(\mathscr{H})$. This violates maximality of $\mathscr{G}_\infty$.

Let $\overline{\mathscr{G}}$ denote the result of rewiring $\mathscr{G}_\infty$. Set
\[
    \Pi_\times = \{ g \in \Pi \mid g \text{ is not starlike in } \overline{\mathscr{G}} \}.
\]
We are going to make these points more starlike. For each $g \in \Pi_\times$, choose a point $x_g \in \Theta(\Pi)$ in an equivariant and measurable way. More precisely, we consider the set
\[
    \{ y \in F_\Pi(g) \mid (g, y) \not \in \overline{\mathscr{G}} \}
\]
and choose $x_g$ to be the element minimising $I(g^{-1}y)$.

Fix a nonbacktracking path $P(g, x_g)$ from $g$ to $x_g$ in $\overline{\mathscr{G}}$. We do this for all $g \in \Pi_\times$ simultaneously, again in an equivariant and measurable way: look at all paths between $g$ and $x_g$ of minimal length (as in number of $\overline{\mathscr{G}}$ edges used), and choose one using the Borel isomorphism $I$ in a similar way to before.

Choose\footnote{If the process isn't ergodic then this $N$ should be a random variable, in any case one can manage.} $N$ so large that there is a positive intensity of points $g \in \Pi_\times$ with paths $P(g, x_g)$ of length at most $N$. 

We now construct our desired marked factor graph $\mathscr{H}$ as follows:
\begin{itemize}
    \item Every point in $\mathscr{G}$ that has a path label retains its path label.
    \item Every point of $\Pi' \setminus \Pi_\times$ is marked $\bullet$.
    \item Every point $g$ of $\Pi_\times$ whose path $P(g, x_g)$ has length greater than $N$ is marked $\bullet$.
\end{itemize}

This leaves the points of $\Pi_\times$ whose paths are bounded by $N$. Note that this is a locally finite family -- each edge appears on at most finitely many $P(g, x_g)$. 

Every path in the rewired graph $\overline{\mathscr{G}}$ can be associated to a path in $\mathscr{G}_\infty$ itself -- every time one of the starlike edges that was added in the rewiring process is added, just go the long way in $\mathscr{G}$. We refer to this as the \emph{detour version} of the path.

Now, to construct the remaining labels check if there are any paths $P(g, x_g)$ which contain an intracell edge $e = (h, h')$ with $h, h' \in F_\Pi(g)$. If so, then we label $g$ by the detour version of this path in $\mathscr{G}_\infty$ with $e$ coloured red.

Observe that the remaining paths must contain \emph{at least two} crossing edges. Each $g$ will \emph{apply} to the first edge crossing edge it sees on the path from $g$ to $x_g$.

Each edge $(h, h')$ receives finitely many applicants $\{g_1, g_2, \ldots, g_k\}$. It chooses the element of this set which minimises $\min \{I(g_i^{-1}h), I(g_i^{-1}h')$.

At last, we finish the construction of $\mathscr{H}$ by marking the points who were rejected by $\bullet$, and the remaining ones by the detour version of this path in $\mathscr{G}$ with their chosen edge coloured red.

Then $\mathscr{G}_\infty \preceq \mathscr{H}$ by construction, but $\iota(\mathscr{H}) > \iota (\mathscr{G}_\infty)$. 

We are therefore able to replace $\mathscr{G}$ by $\overline{\mathscr{G}_\infty}$ and assume our factor graph is starlike, as desired.
\end{proof}

\begin{remark}

The groupoid cost can really increase under a factor map: take the example of Remark \ref{thinninglost} with $\ZZ^n < \RR^n$ for $n > 1$. That is, consider the union of the $\ZZ^n$ lattice shift with the Poisson point process. As a free process, this has cost one. But it factors onto the lattice shift, which has cost greater than one.

\end{remark}

One can also prove cost monotonicity by invoking Gaboriau's theorem on the cost of complete sections:

\begin{thm}[Proposition II.6 of \cite{gaboriau2000cout}, see also Theorem 21.1 of \cite{kechris2004topics}]

If $(X, \Rel, \mu)$ is a pmp cber and $S \subseteq X$ is a complete section\footnote{That is, it meets almost every orbit of $X$.}, then 
\[
    \cost_\mu(\Rel) - 1 = \mu(S) \left( \cost_{\mu | S}(\restr{\Rel}{S}) - 1 \right),
\]
where $\restr{\Rel}{S} = \Rel \cap S \times S$ is the restriction and
\[
    \mu | S := \frac{\mu( \bullet \cap S)}{\mu(S)}
\]
is the conditional measure.
\end{thm}

Suppose $\Phi : (\MM, \mu) \to \MM$ is a point process factor map with $\Phi_* \mu$ of finite intensity. Then
\[
    Y := \Phi^{-1}(\MMo) = \{ \omega \in \MM \mid 0 \in \Phi(\omega) \}
\]
forms a \emph{discrete cross section}\footnote{See Section \ref{crosssectionappendix} for the definition and further context.} for the action $G \acts (\MM, \mu)$. One can define a ``Palm measure'' $\mu_Y$ on $Y$ by replacing all references to $\MMo$ with $Y$, and similarly there is a rerooting equivalence relation $\Rel_Y$ on $Y$. This again forms a pmp cber. Then we have a morphism $\Phi : (Y, \Rel_Y, \mu_Y) \to (\MMo, \Rel, \Phi_* \mu)$ of pmp cbers, so
\[
    \cost_{\mu_Y}(\Rel_Y) \leq \cost_{\Phi_* \mu} (\Rel).
\]

One can see that $Y \cup \MMo$ \emph{also} forms a discrete cross section, and both $Y$ and $\MMo$ are complete sections for it. Then two applications of Gaboriau's theorem shows
\[
    \cost_{\mu_Y}(\Rel_Y) = \cost_{\mu}(\Rel),
\]
thus proving cost montonicity.

\subsection{Unmarking}

We have defined cost of groups by looking at all free \emph{unmarked} point processes on the group. This is no loss of generality:

\begin{prop}\label{abstractlyisom}
Every free point process $\Pi$ on a nondiscrete group with marks from a standard Borel space $\Xi$ is equivariantly isomorphic to an \emph{unmarked} point process. More precisely, if $\Pi$ has marks from a standard Borel space $\Xi$ and $\mu$ is its law, then there is a measurable and equivariant almost everywhere isomorphism $\Phi: (\Xi^\MM, \mu) \to (\MM, \Phi_*\mu)$.
\end{prop}

Since cost is an isomorphism invariant (even if one process is marked and the other isn't), this shows that one can't find point processes with lower cost by using some tricky mark space.

We refer to Proposition \ref{abstractlyisom} as \emph{unmarking}. It should be easy to convince oneself that such a proposition will be true, although the details will necessarily be somewhat messy and ad hoc. We call the technique used \emph{local encoding}, which is illustrated in the following example:

\begin{figure}[h]\label{localencode}
\includegraphics[scale=0.5]{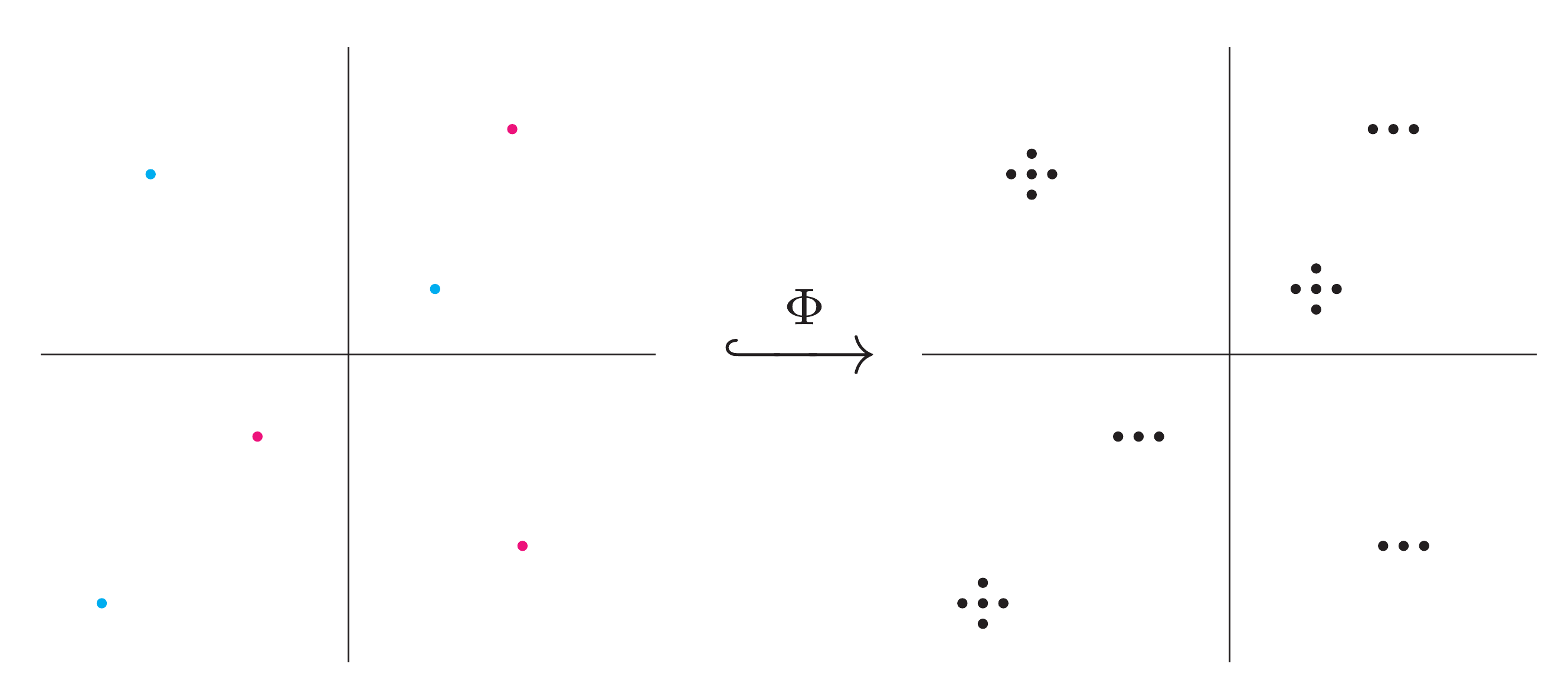}
\centering
\caption{Locally encoding labels of a point process.}
\end{figure}

This is a point process in $\RR^2$ labelled by the set $\{+, -\}$, which we have coloured as cyan and magenta respectively in the diagram. 

The map $\Phi : \{+, -\}^\MM \to \MM$ takes the input configuration, and adds a small decoration around each point. In this case we are literally encoding $+$ marks as a plus symbol centred at each point and similarly for $-$ marks. 

Barring some exceptional circumstances, you should be able to convince yourself that $\Phi$ is an injective map, and thus is an isomorphism onto its image for many input processes. The proof for general $G$ works along the same lines.

We will employ a general lemma that is no doubt well known to experts. For the convenience of the reader we translate a proof appearing in \cite{timar2004} and attributed to Yuval Peres into our language.

\begin{lem}\label{independentsetsexist}
Let $\mu$ be a \emph{free} point process on $G$, and $\mathscr{G}$ a locally finite measurable factor graph of $\mu$. Then one can equivariantly and measurably construct a non-trivial independent subset of $\mathscr{G}$.
\end{lem}

To spell this out, this means there exists a map $I : (\MM, \mu) \to \MM$ with the properties that
\begin{itemize}
    \item $I(\omega) \subset \omega$ almost surely, and
    \item if $g, h \in I(\omega)$, then $g$ and $h$ are not connected in $\mathscr{G}(\omega$).
\end{itemize}

\begin{proof}

The key idea in the proof can be illustrated by via the factor labelling $\odot : \MM \to \MMo^\MM$ given by
\[
    \odot(\omega) = \{ (g, g^{-1}\omega) \in G \times \MMo(G) \mid g \in \omega \}.
\]
Under $\odot$, each point $g$ of a configuration $\omega$ looks at how the configuration looks like from its perspective, and records it as a label. That is, it views itself as the centre of the universe (this is what the symbol $\odot$ is meant to represent, we will call the map \emph{egotistical} or \emph{self-centred}).

Observe that $\mu$ is an (essentially) free action if and only if $\odot(\omega)$ has distinct labels almost surely. For if $g, h \in \omega$ receive the same label under the egotistical map, then $g^{-1}\omega = h^{-1}\omega$, ie. $gh^{-1} \in \stab_G(\omega)$. Conversely, if $g \in \stab_G(\omega)$ is nontrivial, then for all $x \in \omega$ the label $x^{-1}\omega$ of $x$ is the same as that of $gx$, as $(gx)^{-1}\omega = x^{-1}\omega$.

Fix a countable dense subset $Q \subset \MMo$. Let us define a thinning $I_q : \MM \to \MM$ for each $q \in Q$ by
\[
    I_q(\omega) = \{ g \in \omega \mid d(g^{-1}\omega, q) < d(h^{-1}\omega, q) \text{ for all } h \in \omega \text{ adjacent to } g \text{ in } \mathscr{G}(\omega) \}.
\]

Note that each $I_q(\omega)$ is an independent subset of $\mathscr{G}(\omega)$, but it is possibly empty. However, the \emph{union} over all $q$ of the $I_q$ is $\omega$ by freeness, so at least one such $I_q$ must define a non-empty independent subset, as desired.
\end{proof}

In particular, by applying the lemma to the factor graph $\mathscr{D}_R$ of Example \ref{distanceR}, one has:

\begin{cor}
    
    Let $\Pi$ be a free point process. Then for all $R > 0$ one deterministically, measurably, and equivariantly select a subset $\Pi_R \subset \Pi$ that is $R$ uniformly separated, in the sense that if $x$ and $y$ are distinct points of $\Pi_R$, then $d(x,y) > R$.
    
\end{cor}

Our proof will also make use of a technique we refer to as \emph{label trickery}:

\begin{prop}[Label trickery]\label{labeltrickery}

Let $\Pi$ be any free point process (possibly marked), and $\theta(\Pi)$ any nonempty thinning. Then there exists a marked point process $\Upsilon$ such that the underlying point set of $\Upsilon$ is $\theta(\Pi)$, and $\Upsilon$ is isomorphic to $\Pi$ \emph{as a pmp action}. In particular, $\Upsilon$ is a free action.

The same can be achieved with $\Upsilon$ having marks from the \emph{compact} space $[0,1]$.

\end{prop}

\begin{proof}

Let $\Upsilon = \theta(\odot(\Pi))$, that is,
\[
    \Upsilon = \{ (g, g^{-1}\Pi) \in \MM \times \MMo \mid g \in \theta(\Pi) \}.
\]
Observe that this is an \emph{injective} map, as one can recover $\Pi$ uniquely from the knowledge of any point $\Upsilon$ and its label, and so $\Upsilon$ is an isomorphic process to $\Pi$.

For the second statement, simply fix a Borel isomorphism\footnote{It exists as $\MMo$ is a Polish space, and thus standard Borel, and all standard Borel spaces of the same cardinality are isomorphic.} $I : \MMo \to [0,1]$, and define
\[
    \Upsilon = \{ (g, I(g^{-1}\Pi) \in \MM \times [0,1] \mid g \in \theta(\Pi) \}.
\]\end{proof}

\begin{proof}[Proof of Proposition \ref{abstractlyisom}]

Suppose $\Pi$ is a free $\Xi$-marked point process with law $\mu$. We can (and do) assume that $\Pi$ is abstractly isomorphic to a uniformly separated process with a slightly different (but nevertheless standard Borel) mark space by using the previous two propositions.

Let $X$ denote the space:
\[
    X = \{ \omega \in \MMo(B(0, \delta/100)) \mid \omega \cap B(0, \delta/200) = \{0\}, \text{ and } \forall x \in \omega \setminus \{0\}, \abs{B(x, \delta/200)} > 1 \}.
\]
This is a Borel subset of a standard Borel space, and hence standard Borel in its own right. One can readily see that it is uncountable, and hence there is a Borel isomorphism $I: \Xi \to X$. Define the following factor map:
\begin{align*}
    &\Phi : \Xi^\MM \to \MM \\
    &\Phi(\omega) = \bigcup_{x \in \omega} x I(\xi_x),
\end{align*}
where $\xi_x$ denotes the label of $x$ (that is, $(x, \xi_x) \in \omega$.

This is an injective map: we can recover the underlying set of any input configuration to $\Phi$ by identifying the points which are $\delta/200$-isolated. We can then uniquely recover their labels by applying the inverse of $I$ locally.
\end{proof}

\subsection{Cost is finite for compactly generated groups}

\begin{prop}\label{finitecost}
    
    Suppose $G$ is \emph{compactly generated} by $S \subseteq G$. Then every free point process $\Pi$ on $G$ has finite cost.
    
\end{prop}
    
Implicitly we are assuming that $\Pi$ has finite intensity, so that its cost is defined. 
    
We recall some definitions and facts from metric geometry, see \cite{MR3561300} for further details in the specific context we are interested in.    
    
\begin{defn}\label{metricdefs}

Let $(X, d)$ be a metric space. 

\begin{itemize}
    \item $(X, d)$ is \emph{coarsely connected} if there exists $c > 0$ such that for all $x, x' \in X$ there are points $x_1, x_2, \ldots, x_n \in X$ with $x = x_1$, $x_n = x'$, and $d(x_i, x_{i+1}) \leq c$ for all $i$.
    \item A subset $\omega \subseteq X$ is \emph{uniformly discrete} if there exists $\e > 0$ such that $d(x, y) > \e$ for all distinct $x, y \in \omega$.
    \item A subset $\omega \subseteq X$ is \emph{coarsely dense} if there exists $r > 0$ such that for every $x \in X$, $d(x, \omega) < r$.
    \item A \emph{Delone set} is a subset $\omega \subseteq X$ which is both uniformly discrete and coarsely dense.
    \item An \emph{$\e$-net} is a subset $\omega \subseteq X$ which is $\frac{\e}{2}$ uniformly discrete and $\e$ coarsely dense.
\end{itemize}

\end{defn}

\begin{thm}[See Proposition 1.D.2 of \cite{MR3561300}]

Let $G$ be an lcsc group with a left-invariant proper metric $d$ which generates its topology. Then $G$ is compactly generated if and only if it is coarsely connected.

\end{thm}

Note that if $X$ is coarsely connected, then so too is any coarsely dense subset of $X$.

 \begin{defn}
 
 Let $S \subseteq G$ be a compact and symmetric generating set.
 
     The \emph{Cayley factor graph} associated to $S$ is the map $\Cay(\bullet, S) : \MM \to \graph(G)$ given by
    \[
        \Cay(\omega, S) = \{ (g, gs) \in \omega \times \omega \mid s \in S \}.
    \]
 
 \end{defn}

    Note that this graph is not necessarily connected, for instance for the Poisson point process. However, if $\Pi$ is a point process which is almost surely $c$-coarsely-connected for $c$ such that $B(0,c) \subseteq S$ then $\Cay(\Pi, S)$ is connected. This condition can always be satisfied by replacing $S$ with an appropriate power of the generating set $S^k$, since $S^k$ exhausts $G$ and in particular must contain $B(0,c)$ for $k$ sufficiently large

The following can be readily deduced from existing results in the literature (even removing the compact generation assumption), but we include a separate proof for completeness.

    \begin{prop}\label{factoronnet}
        
        Suppose $\Pi$ is a free and ergodic point process on a compactly generated group $G$. Then for every $R > 0$ there exists a \emph{finite intensity} thickening $\Theta$ of $\Pi$ such that $\Theta(\Pi)$ is almost surely $R$-coarsely-dense.
        
        Moreover, if $\Pi$ is $\delta$-separated (with $\delta < 2R$), then $\Theta$ will also be $\delta$-separated.
    \end{prop}

\begin{proof}
It suffices to prove the statement for ergodic processes. 

Fix $R > 0$. We will construct a factor map $\Phi$ of $\Pi$ such that $\Phi(\Pi)$ is $\frac{R}{2}$ uniformly separated and $\Theta(\Pi) := \Pi \sqcup \Phi(\Pi)$ is $R$-coarsely dense. The uniform separation then implies that this thickening has finite intensity.

The idea of the proof is the following: observe that every uniformly separated subset of a metric space is a subset of a Delone set. You can prove this using the well-ordering principle or Zorn's lemma (as to your taste). Now consider a sample $\Pi$ from the point process. We know there are \emph{some} ways to add points to it to get something coarsely dense, the only difficulty is that we are required to make these choices equivariantly. We will select points that see the ``frontier'' of the process, which will then add points to cover a piece of the frontier. At every stage the frontier gets smaller, and in the limit we cover the whole space.

For configurations $\omega \in \MM$, let $\omega^t$ denote the following closed set
\[
    \omega^t = \bigcup_{g \in \omega} B(g, t),
\]
that is, the union of all \emph{closed} balls about the points of $\omega$. If $\Pi$ is a point process, then $\Pi^t$ is a random closed subset of $G$.

We call a point $g \in \Pi$ \emph{on the frontier} if $B(g, c_1 R) \not\subseteq \Pi^R$, where $c_1 > 1$ is some parameter to be chosen later, and let $F(\Pi)$ denote the subset of frontier points of $\Pi$. This is a metrically defined condition, and hence equivariant. We will define a rule $\Phi_1(\Pi)$ that specifies a collection of points such that their $R$-balls cover all the $c_1 R$-balls of the frontier points of $\Pi$. We will then iterate this construction (so that $\Phi_2(\Pi)$'s $R$-balls cover the $c_2 R$-balls of $\Phi_1(\Pi) \cup \Pi$'s frontier points, for some $c_2 > c_1$, and so on). In this way we will find enough points to cover the whole space.

Choose $c_1$ large such that $\PP[ \Pi^{c_1 R} \setminus \Pi^R \neq \empt] = 1$. If this is not possible, then the process is already $R$-coarsely-dense by ergodicity.

One can decompose the frontier points of $\Pi$ as
\[
    F(\Pi) = \bigsqcup_n F_n(\Pi),
\]
where each $F_n(\Pi)$ is $10 c_1 R$ uniformly separated. This can be done by using the existence of a \emph{Borel kernel} of the factor graph $\mathscr{D}_{10 c_1 R}(\Pi_0)$ defined on the frontier points of $\Pi_0$, see Section 4 of \cite{KECHRIS19991} for further information on Borel kernels. Note that by using the Palm process, the resulting sets $F_n(\Pi)$ are equivariantly defined.

We now fix an auxiliary (deterministic) $R$-net $\mathcal{N} \subset G$. If $W \subseteq G$ is a Borel region and $g \in G$, then let
\[
    N(g, W) = \{ x \in g^{-1}\mathcal{N} \mid B(x, R) \cap W \neq \empt \}.
\]
Note that $N(g,W)^R \supseteq W$, as $\mathcal{N}$ is coarsely dense.
Define
\[
    \Phi_1(\Pi) = \bigcup_{g \in F_1(\Pi)} N(g, B(g, c_1 R) \setminus \Pi^R),
\]
and inductively
\[
    \Phi_{n+1}(\Pi) = \bigcup_{g \in F_n(\Pi)} N(g, B(g, c_1 R) \setminus \left(\Pi \cup \bigcup_{i \leq n} \Phi_i(\Pi)  \right)^R )
\]
Then
\[
    \Pi^{c_1 R} \subseteq \Pi^R \cup \bigcup_{n \geq 1} \Phi_n(\Pi)^R.
\]

We now repeat this procedure as many times as necessary (possibly countably infinitely many times) until we construct the desired thickening $\Theta$. The only care necessary is that one should choose the parameters $c_1 < c_2 < \cdots$ so that they tend to infinity, as
\[
    G = \bigcup_{n \geq 1} \Pi^{c_n R}
\]
for any such sequence. 
\end{proof}

\begin{remark}

In Section \ref{crosssectionappendix} we describe the connection between point processes and ``cross-sections'' of actions. The previous proposition can be deduced from the fact that every free action admits a ``cocompact cross-section''. A similar statement to the proposition directly phrased in terms of cross-sections can be found in Section 2 of \cite{slutsky2017lebesgue}, where it is shown that any cross-section can be extended to a cocompact cross-section. That proof works without the compact generation assumption.

\end{remark}

\begin{proof}[Proof of Proposition \ref{finitecost}]
   It suffices to consider the case where $\Pi$ is ergodic, see \cite{kechris2004topics} Corollary 18.6.
   
    If $\Pi$ is $\delta$-separated, then the thickening constructed in Proposition $\ref{factoronnet}$ has finite cost for the Cayley factor graph, which is connected (for a suitable choice of generating set). Cost is monotone for factors, so $\Pi$ itself has finite cost.
    
    Otherwise, choose $\delta$ sufficiently small so that the $\delta$-thinning of $\Pi$ is non-empty almost surely. By label trickery (Proposition \ref{labeltrickery}), $\Pi$ is isomorphic to a (marked) $\delta$-separated point process, reducing to the previous case.
\end{proof}

\section{The Poisson point process has maximal cost}

We begin with the observation that every IID process factors onto the Poisson:

\begin{prop}\label{factorontopoisson}

Let $\Pi$ be a point process on $G$. Then $[0,1]^\Pi$ factors onto the Poisson point process.

\end{prop}

\begin{proof}

Fix a map $F : [0,1] \to \MM(G)$ such that if $\xi \sim \texttt{Unif}[0,1]$, then $F(\xi)$ is a Poisson point process on $G$ of unit intensity.

We will use the Voronoi tessellation to simply glue independent copies of the Poisson point process in each cell, resulting in a Poisson point process.

Define a factor map $\Phi([0,1]^\Pi)$ by
\[
    \Phi([0,1]^\Pi) = \bigcup_{g \in [0,1]^\Pi} g \cdot F(\xi_g) \cap V_{\Pi}^T(g) , 
\]
where $\xi_g$ denotes the label of $g$ in $[0,1]^\Pi$. Then $\Phi([0,1]^\Pi)$ is the Poisson point process.
\end{proof}

In particular, the cost of every IID process is at most the cost of the Poisson. 

Our goal is to prove an \emph{asymptotic} version of this statement. We will show that every free point process ``weakly'' factors onto an IID process, and that cost is monotone for (certain) weak factors. This will prove that the Poisson point process has maximal cost amongst all free point processes.

\subsection{Weak factoring and Ab\'{e}rt-Weiss for point processes}

We have seen that cost is monotone under factor maps. We will now introduce a weaker version of factoring and investigate its relationship to cost:

\begin{defn}
    
    Let $\Pi$ and $\Upsilon$ be point processes. Then $\Pi$ \emph{weakly factors} onto $\Upsilon$ if there is a sequence $\Phi^n$ of factors of $\Pi$ such that $\Phi^n(\Pi)$ weakly converges\footnote{For more information on weak convergence in the context of point processes, see Appendix \ref{metricproperties}.} to $\Upsilon$.
    
\end{defn}

The restive reader is advised to take a look at the statements of Theorem \ref{abertweiss} and Theorem \ref{costmonotonicity}. These are the tools that will be used to prove the headline theorem of this section. The other results in this section are necessary but have a more routine flavour.

\begin{thm}\label{iidamenableweakfactor}

Let $\Pi$ and $\Upsilon$ be point processes on an amenable group $G$. If $\Pi$ is free, then $[0,1]^\Pi$ weakly factors onto $\Upsilon$.

\end{thm}

The proof of this uses a lemma, a proof of which can be found in a concurrently appearing paper by the second author:

\begin{lem}\label{hyperfinitelemma}
   If $\Pi$ is a free point process on an amenable group $G$, then there exists factor partitions $\mathcal{P}_n(\Pi) = \{P^n_g\}_{g \in \Pi}$ with the following properties:
   \begin{description}
       \item[Equivariance:] $P^n_{\gamma g} = \gamma P^n_g$,
       \item[Partitioning:] For each $n$, $G$ is the  union of $\{P^n_g\}_{g \in \Pi}$, and if $g, h \in \Pi$ then $P^n_g = P^n_h$ or $P^n_g \cap P^n_h = \empt$,
       \item[Increasing:] For each $n$, $P^n_g \subseteq P^{n+1}_g$,
       \item[Exhausting:] For all compact $C \subseteq G$, there exists $N$ and $g \in \Pi$ such that $C \subseteq P^N_g$,
       \item[Finite volume:] For all $n$, $0 < \lambda(P^n_g) < \infty$.
       \item[Finitariness:] For each $n$, $P^n_g \cap \Pi$ is finite (and contains $g$).
   \end{description}
\end{lem}

\begin{figure}[h]\label{clumpingfigure}
\includegraphics[scale=0.5]{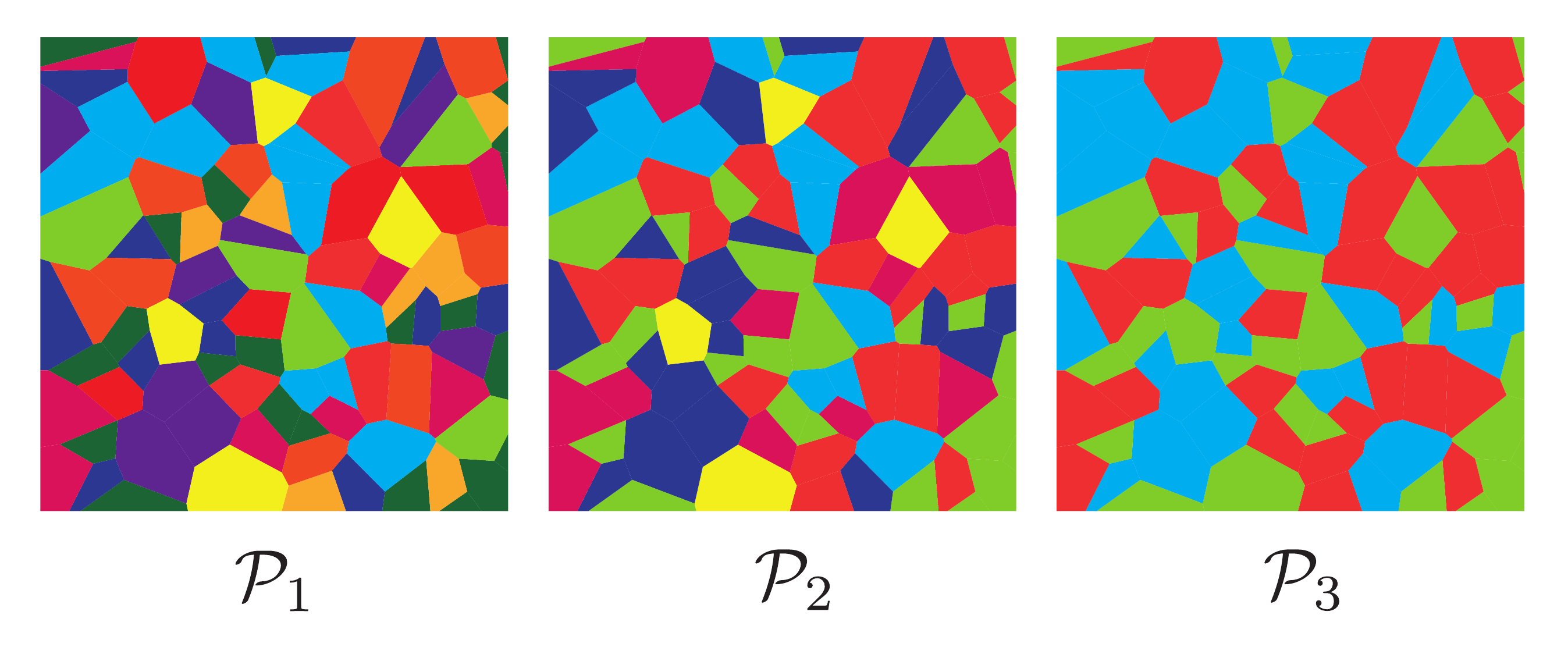}
\centering
\caption{The factor partitions from Lemma \ref{hyperfinitelemma} are ``clumpings'', and should be visualised like this. }
\end{figure}

\begin{proof}[Proof of Theorem \ref{iidamenableweakfactor}]
    
Let $f : [0,1] \to \MM$ be a measurable map with $f(\xi) \sim \Upsilon$ if $\xi \sim \texttt{Unif}[0,1]$.

Choose factor partitions $\mathcal{P}_n(\Pi) = \{P^n_g\}_{g \in \Pi}$ as in Lemma \ref{hyperfinitelemma}. Let $\Pi_n$ be an equivariantly defined subprocess of $\Pi$ which consists of one point chosen out of each cell $P^n_g$ -- we are able to do this by essential freeness. For instance, fix a Borel isomorphism of $\MMo$ with $[0,1]$ and note that the induced label in $[0,1]$ at each point of $\Pi$ is distinct for all points (essential freeness), and thus we may choose $\Pi_n$ to consist of the point with maximal label amongst its cell.

Define factors $\Phi_n$ as follows:
\[
    \Phi_n([0,1]^\Pi) = \bigcup_{g \in \Pi_n} g \cdot (f(\xi_g) \cap P^n_g).
\]
That is, in each cell we glue a copy of the process $\Upsilon$ sampled according to the label $\xi_g$ on $g$. 

It follows immediately that $\Phi_n([0,1]^\Pi)$ weakly converges to $\Upsilon$: if $C \subseteq G$ is any compact stochastic continuity set for $\Upsilon$, then for sufficiently large $N$  is entirely contained in some $P^N_g$, and thus the point counts of $\Phi^n([0,1]^\Pi)$ inside $C$ are exactly distributed the same as those of $\Upsilon$.
\end{proof}

The following statement is due to Ab\'{e}rt and Weiss\cite{abert2013bernoulli} for discrete groups, we extend it to point processes:

\begin{thm}\label{abertweiss}

Let $\Pi$ be an essentially free point process on a noncompact group $G$. Then $\Pi$ weakly factors onto $[0,1]^\Pi$, its own IID.

\end{thm}

\begin{proof}
    
    It suffices to show that $\Pi$ weakly factors onto $[d]^\Pi$, where $[d] = \{ 1, 2, \ldots, d\}$ is equipped with the uniform measure. Here $[d]^\Pi$ is thus the \emph{finitary} IID of $\Pi$. This suffices as $[d]^\Pi$ weakly converges to $[0,1]^\Pi$ as $d \to \infty$. We will do this by constructing factor \emph{$[d]$-labellings} $\mathscr{C}_n$ of $\Pi$ such that $\mathscr{C}_n(\Pi)$ weakly converges to $[d]^\Pi$.
    
    To do this, we'll use the second moment method, hewing close to the original Ab\'{e}rt-Weiss recipe.
    
    The strategy will be as follows. Consider the set of $[d]$-labellings of $\Pi$. We will study a probabilistic model that produces a \emph{random element} of this space. We will show that this random \emph{deterministic colouring} satisfies certain constraints with positive probability. In particular, there must exist a $[d]$-labelling satisfying those constraints. By adjusting the parameters of this model, one can produce the desired sequence $\mathscr{C}_n$.
    
    Fix a \emph{countable} weak convergence determining family $\{V_i\}$ as discussed at Lemma \ref{determiningclass}, so that the sets $V_i \subset G \times [d]$ are bounded stochastic continuity sets for $[d]^\Pi$. We will construct a sequence of factor colourings $\mathscr{C}_n$ of $\Pi$ such that for fixed $k$,
    \[
         N_{\boldsymbol{V}_k}(\mathscr{C}_n \Pi) \text{ converges weakly to } N_{\boldsymbol{V}_k}([d]^\Pi),
    \]
    where $\boldsymbol{V}_k = (V_1, V_2, \ldots, V_k)$.
    
    Set $W_k = \bigcup_{i \leq k} V_k$ to be the total window. Formally this is a subset of $G \times [d]$, but we view it as a subset of $G$. For $\e > 0$ arbitrary, we choose $\delta > 0$ so small that the following properties are true, where $\mu$ denotes the law of $\Pi$:
    \[
        \mu(\{ \omega \in \MM \mid \text{for all } g, h \in \omega \cap W_k, g \neq h \text{ implies } d(g^{-1}\omega, h^{-1}\omega) > \delta \}) > 1 - \e
    \]
    and
    \[
        (\mu \otimes \mu)(\{ (\omega, \omega') \in \MM \times \MM \mid \text{for all } (g, h) \in (\omega \cap W_k) \times (\omega' \cap W_k),\; d(g^{-1}\omega, h^{-1}\omega') > \delta \}) > 1 - \e.
    \]
    This is possible by essential freeness (for $A_\delta$) and essential freeness and noncompactness of $G$ (for $B_\delta$). To see this, let us formulate essential freeness in the following way:
    \[
        \mu(\{ \omega \in \MM \mid \text{ for all } g, h \in \omega, g \neq h \text{ implies } g^{-1}\omega \neq h^{-1}\omega \}) = 1.
    \]
    This remains an almost sure event if we restrict $g$ and $h$ to lie in the window $W_k$. Now observe that $A_\delta$ increases as $\delta$ tends to zero to this almost sure event.
    
    The argument for $B_\delta$ is similar, but depends on the fact that the set
    \[
        B_0 = \{(\omega, \omega') \in \MM \times \MM \mid \text{for all } g \in \omega, h \in \omega', g^{-1}\omega \neq h^{-1}\omega' \} 
    \]
    is $\mu \otimes \mu$ almost sure, which is less immediate and we now show.
    
    Observe that $\mu \otimes \mu$ defines a point process of $G \times G$ of intensity $\intensity(\mu)^2$, and Palm measure $\mu_0 \otimes \mu_0$. By the correspondence of measures between $\mu \otimes \mu$ and $\mu_0 \otimes \mu_0$, we equivalently ask that
    \[
        (\mu_0 \otimes \mu_0)(\{ (\omega, \omega') \in \MMo \times \MMo \mid \omega \neq \omega'\}) = 1,
    \]
    or equivalently that $\mu_0$ has no atoms. This is contradicted by essential freeness: if $\omega \in \mu_0$ is an atom, then $\omega$ is shift invariant, that is, $g^{-1}\omega = \omega$ for all $g \in \omega$. This implies that $\omega$ is in fact a \emph{subgroup} of $G$, and it's discrete by definition. Now the ergodic component of $\mu$ corresponding to $\omega$ is supported on $G\omega$, and thus defines a $G$-invariant probability measure on $G/\omega$, that is, it is a lattice shift. But $\mu$ was assumed to be essentially free.
    
    We now construct a \emph{random} colouring $\mathscr{C}$ of $\Pi$ in the following way: let
    \[
        \MMo = \bigsqcup_i D_i, \text{ where } \diam(D_i) < \delta.
    \]
    be a partition of $\MMo$ into small measurable sets. By the correspondences we've described, any $[d]$-colouring of the sets $D_i$ corresponds to a factor colouring $\mathscr{C} : \MM \to [d]^\MM$ in the following way:
    \[
        \mathscr{C}(\omega) = \{(g, c) \in \omega \times [d] \mid g^{-1}\omega \in \MMo \text{ is coloured by } c \}.
    \]
    We look at such $\mathscr{C}$ when the $D_i$ sets are coloured \emph{uniformly at random} by elements of $[d]$. To emphasise: we are considering a \emph{distribution} on \emph{deterministic} colourings.

   For an integral vector $\boldsymbol{\alpha} = (\alpha_1, \alpha_2, \ldots, \alpha_k) \in \NN_0^k$, we set
    \[
        T_\alpha = \{ \omega \in [d]^\MM \mid (N_{V_1}(\omega), \ldots, N_{V_k}(\omega)) = \alpha \}
    \]
    to be the set of configurations whose point/colour statistics in $W_k$ are prescribed by $\alpha$. 
    
    Note that $\mathscr{C}_*\mu(T_\alpha)$ is a random variable (whose source of randomness is $\mathscr{C}$).

    Given $k$ and $M$, we use the second moment method to prove the existence of $\mathscr{C}$ such that for all $\alpha \in \NN_0^k$ with $\norm{\alpha}_\infty \leq M$,
    \[
        \abs{\mathscr{C}_*(\mu) (T_\alpha) - [d]^\mu(T_\alpha)} < \e.
    \]
    Then any sequence of such colourings with $k, M$ tending to infinity will witness that $\Pi$ weakly factors onto $[d]^\Pi$.

    Exchanging order of integration allows us to express the mean of $\mathscr{C}_*(\mu) (T_\alpha)$ as
    \begin{align*}
        \EE \left[ \mathscr{C}_*(\mu) (T_\alpha) \right] &= \EE[ \mu(\mathscr{C}^{-1}(T_\alpha))] \\
        &= \EE\left[ \int_\MM \1[\mathscr{C}(\omega) \in T_\alpha] d\mu(\omega)\right] \\
       &= \int_\MM \EE\left[ \1[\mathscr{C}(\omega) \in T_\alpha\right]] d\mu(\omega).
    \end{align*}
    
    Note that for $\omega \in A_\delta$, all pairs of distinct points $g,h \in \omega$ from the window $W_k$ have the property that $g^{-1}\omega$ and $h^{-1}\omega$ fall into different $D_i$ sets, and are therefore assigned \emph{independent} colours. Thus
    
    \[
        \text{ for } \omega \in A_\delta, \EE \left[\1[\mathscr{C}(\omega) \in T_\alpha]\right] = [d]^\mu(T_\alpha).
    \]
    As $\mu(A_\delta) > 1 - \e$, it follows that
    \[
        \abs{  \EE \left[ \mathscr{C}_*\mu(T_\alpha) \right] - [d]^\mu(T_\alpha) } < 2\e.
    \]
    We now work on the variance. Again, exchanging order of integration in a similar way to before allows us to express the mean of $(\mathscr{C}_*(\mu) (T_\alpha))^2$ as
    \[
        \EE \left[ (\mathscr{C}_*(\mu) (T_\alpha))^2 \right]  = \iint_{\MM \times \MM} \EE \left[ \1[\mathscr{C}(\omega) \in T_\alpha] \1[\mathscr{C}(\omega') \in T_\alpha] \right] \,d\mu(\omega) \,d\mu(\omega').
    \]
    
    By similar reasoning to before, for $(\omega, \omega') \in (A_\delta \times A_\delta) \cap B_\delta$, the colours one will see at points in $W_k$ will be independent. Thus for such $(\omega, \omega')$ we have
    \[
        \EE \left[ \1[\mathscr{C}(\omega) \in T_\alpha] \1[\mathscr{C}(\omega') \in T_\alpha] \right] = \left( [d]^\mu(T_\alpha) \right)^2
    \]
    Note that $(A_\delta \times A_\delta) \cap B_\delta = (A_\delta \times \MMo) \cap (\MMo \times A_\delta) \cap B_\delta$, so by the union bound $(\mu \otimes \mu)((A_\delta \times A_\delta) \cap B_\delta) > 1 - 3\e$. 
    
    Putting this together,
    \[
    \Var(\mathscr{C}_*(\mu) (T_\alpha)) = \EE \left[ (\mathscr{C}_*(\mu) (T_\alpha))^2 \right] - \left(\EE \left[ \mathscr{C}_*\mu(T_\alpha) \right]\right)^2 < 12\e.
    \]
    We now apply Chebyshev's inequality which states that for any $c > 0$,
    \[
        \PP\left[ \abs{\mathscr{C}_*(\mu) (T_\alpha) - \EE[\mathscr{C}_*(\mu) (T_\alpha)]} > c \right] < \frac{\Var(\mathscr{C}_*(\mu) (T_\alpha))}{c^2}.
    \]
    Our bounds on the mean and the variance of $\mathscr{C}_*(\mu) (T_\alpha)$ and the choice $c = \e^{\frac{1}{3}}$ yield
    \[
        \PP\left[ \abs{\mathscr{C}_*(\mu) (T_\alpha) - [d]^\mu(T_\alpha)} > \e^{\frac{1}{3}}+2\e \right] < 12\e^{\frac{1}{3}}.
    \]
    Let $E_\alpha$ denote the event $\{\abs{\mathscr{C}_*(\mu) (T_\alpha) - [d]^\mu(T_\alpha)} < \e^{\frac{1}{3}} + 2\e \}$. Then by the union bound
    \[
        \PP[ \bigcap_{\substack{\alpha \in \NN_0^k \\ \norm{\alpha}_\infty \leq M}} E_\alpha] \geq 1 - 12M^k \e^{\frac{1}{3}}.
    \]
    In particular, by choosing $\e$ sufficiently small, such a colouring exists.
\end{proof}

\begin{prop}\label{weakfactoringtransitive}

Suppose $\Pi$ and $\Upsilon$ are point processes, with $\Pi$ weakly factoring onto $\Upsilon$ and $\Psi(\Upsilon)$ being a factor of $\Upsilon$. Then $\Pi$ weakly factors onto $\Psi(\Upsilon)$. 

It follows that weak factoring is a \emph{transitive} notion.
\end{prop}

\begin{proof}

We have seen that every factor map decomposes as the composition of a (coloured) thickening and a thinning. We are therefore able to reduce the problem to the case where $\Psi$ is a thinning, a colouring, and a thickening.

We will repeatedly use the following fact: if $\Pi$ weakly factors onto $\Psi_m(\Upsilon)$ for a sequence of factors $\Psi_m$, and $\Psi_m$ converges to $\Psi$ pointwise, then $\Pi$ weakly factors onto $\Psi(\Upsilon)$.

We begin with the case of thinnings.

Let $\Phi^n(\Pi)$ weakly converge to $\Upsilon$. We write $\mu$ and $\nu$ for the laws of $\Pi$ and $\Upsilon$ respectively. Let $\theta^A$ be a thinning of $\Upsilon$, determined by a subset $A \subseteq (\MMo, \nu_0)$ as in Theorem \ref{correspondencetheorem}.

The idea of the proof is this: if $A$ were a $\nu_0$ continuity set, then the corresponding thinning $\theta^A : \MM \to \MM$ is continuous $\nu$ almost everywhere, and so $\theta^A(\Phi^n(\Pi))$ converges to $\theta^A(\Upsilon)$ by Lemma \ref{continuity}. We handle the general case by approximating $A$ by $\nu_0$ continuity sets.

\begin{claim}
If $A$ is a $\nu_0$ continuity set, then $\theta^A : \MM \to \MM$ is continuous $\nu$ almost everywhere.
\end{claim}
To see this, recall the \emph{saturation} notion we used in the proof of Theorem \ref{correspondencetheorem}. We've assumed $\nu_0(\partial A) = 0$, and hence $\nu(G \partial A) = 0$ too. Then $\theta^A$ is continuous on the complement of this set. Note that if $\omega \not\in G \partial A$, then $g^{-1}\omega \not\in \partial A$ for all $g \in \omega$. One can now see that if $\omega_n$ converges to $\omega$, then $\theta^A(\omega_n)$ restricted to any fixed radius ball is eventually equal to $\theta^A(\omega)$, as desired.

For the general case, let $A_m \subseteq \MMo$ be $\nu_0$-continuity sets such that
\[
\nu_0(A \triangle A_m) < \frac{1}{ 2^m}.
\]
Then for every $m$, we have $\theta^{A_m}(\Phi^n(\Pi)) \to \theta^{A_m}(\Upsilon)$ by our earlier argument. 

By Borel-Cantelli, $A \triangle A_m$ does not occur infinitely often. This is true for the saturation, so we see that $\theta^{A_m} \to \theta^{A}$ pointwise almost surely and hence also in distribution. By choosing an appropriate subsequence of $m$s and $n$s we find our desired sequence of factor maps.

The above proof for thinnings can be immediately adapted to prove that $\Pi$ weakly factors onto $\Psi(\Upsilon)$ if $\Psi$ is any $[d]$-colouring of $\Upsilon$. Since any colouring is a pointwise limit of finitary colourings, we see that $\Pi$ weakly factors onto \emph{any} colouring of $\Upsilon$.

Finally, suppose $\Psi$ is a thickening of $\Upsilon$. By using the Voronoi tessellation we may express $\Psi$ in the following form:
\[
    \Psi(\omega) = \bigcup_{g \in \omega} g F(g^{-1}\omega),
\]
where $F : \MMo \to \MMo$ is a measurable function. 

We say that $\Psi$ is a \emph{bounded range} thickening if there exists $C > 0$ such that $F(\Upsilon_0) \subseteq B(0, C)$ almost surely.

It is easy to show that $\Psi$ is the pointwise limit of such thickenings, so we are reduced to this case.

Define $I : \MMo(B(0, C))^\MM \to \MM$ by
\[
    I(\omega) = \bigcup_{g \in \omega} g \xi_g,
\]
where $\xi_g$ is the label of $g$ in $\omega$, that is, $(g, \xi_g) \in \omega$.

This is the \emph{implementation} map: it takes a schema for a thickening and implements it.

\begin{claim}
The implementation map $I$ is continuous.
\end{claim}

The task is to show that given $R, \e > 0$ there exists $S, \delta > 0$ such that if $\omega$ and $\omega'$ are $(S, \delta)$-wobbles of each other, then $I(\omega)$ and $I(\omega')$ are $(R, \e)$-wobbles of each other.

The idea is simply that the behaviour of $I(\omega)$ in a ball of radius $R$ is determined by $\omega$ restricted to the ball of radius $R + C$, as the thickening is of bounded range $C$. By choosing $\delta$ sufficiently small (depending on the labels of the points in $\omega \cap B(0, R + C)$, we find our desired $S$ and $\delta$.

With the claim in hand, our desired result follows from the claim and Lemma \ref{continuity}.
\end{proof}

\begin{cor}\label{amenableequivalent}
    Let $\Pi$ and $\Upsilon$ be point processes on an amenable group, with $\Pi$ free. Then $\Pi$ weakly factors onto $\Upsilon$.
\end{cor}

\begin{proof}
    By Theorem \ref{abertweiss} $\Pi$ weakly factors onto $[0,1]^\Pi$, and by Theorem \ref{iidamenableweakfactor} $[0,1]^\Pi$ weakly factors onto $\Upsilon$. Hence the claim follows from Proposition \ref{weakfactoringtransitive}.
\end{proof}

\begin{lem}\label{weakfactorimpliesiiweakfactor}
    
    Suppose $\Pi_n$ weakly converges to $\Pi$. Then $[0,1]^{\Pi_n}$ weakly converges to $[0,1]^{\Pi}$.
    
\end{lem}

\begin{proof}

This can be seen, for instance, by verifying that the finite dimensional distributions of $[0,1]^{\Pi_n}$ weakly converge to those of $[0,1]^\Pi$.

Recall that a $[0,1]$-marked point process on $G$ is just a particular kind of point process on $G \times [0,1]$. It therefore suffices to check weak convergence of the finite dimensional distributions against stochastic continuity sets of $[0,1]^\Pi$ in product form.

To that end, let $\boldsymbol{V} = (V_1, V_2, \ldots, V_k)$ denote a collection of stochastic continuity sets for $\Pi$, and $[0, \boldsymbol{p}) = ([0, p_1), [0, p_2), \ldots, [0, p_k))$ a family of intervals in $[0,1]$. We denote by $\boldsymbol{V} \times [0, \boldsymbol{p}) = (V_1 \times [0, p_1), \ldots, V_k \times [0, p_k))$ the stochastic continuity set of $[0,1]^\Pi$. Fix an integral vector $\boldsymbol{\alpha} \in \NN_0^k$. We must show that $\PP[N_{\boldsymbol{V} \times [0, \boldsymbol{p})} [0,1]^{\Pi_n} = \boldsymbol{\alpha}]$ converges to $\PP[N_{\boldsymbol{V} \times [0, \boldsymbol{p})} [0,1]^{\Pi} = \boldsymbol{\alpha}]$. 

We find the following explicit expression simply by conditioning on $\boldsymbol{\beta}$, the total number of points appearing in $\boldsymbol{V}$:

\begin{align*}
    \PP[N_{\boldsymbol{V} \times [0, \boldsymbol{p})} [0,1]^{\Pi_n} = \boldsymbol{\alpha}] &= \sum_{\boldsymbol{\beta} \geq \boldsymbol{\alpha}} \PP[N_{\boldsymbol{V} \times [0, \boldsymbol{p})} [0,1]^{\Pi_n} \mid N_{\boldsymbol{V}} \Pi_n = \boldsymbol{\beta}] \PP[N_{\boldsymbol{V}} \Pi_n = \boldsymbol{\beta}]\\
    &= \sum_{\boldsymbol{\beta} \geq \boldsymbol{\alpha}}  \prod_{i=1}^k p_i^{\alpha_i} (1-p_i)^{\beta_i - \alpha_i} \PP[N_{\boldsymbol{V}} \Pi_n = \boldsymbol{\beta}],
\end{align*}
where by $\boldsymbol{\beta} \geq \boldsymbol{\alpha}$ we mean that $\beta_i \geq \alpha_i$ for each entry.

There is an identical expression for $\Pi$ (simply replace all instances of $\Pi_n$ by $\Pi$). The conclusion follows, as $\PP[N_{\boldsymbol{V}} \Pi_n = \boldsymbol{\beta}]$ converges to $\PP[N_{\boldsymbol{V}} \Pi = \boldsymbol{\beta}]$ for all $\boldsymbol{\beta}$.
\end{proof}

We have seen that all free point processes are able to \emph{weakly} factor onto their own IID. It is natural to ask if all this hassle was worth it -- can a point process always factor directly onto its own IID? 

\begin{thm}[Holroyd, Lyons, Soo\cite{MR2884878}]

The Poisson point process cannot be split into two \emph{independent} Poisson point processes of lower intensity without additional randomness.

More precisely, there does not exist a \emph{deterministic} two colouring $\mathscr{C} : (\MM, \PPP) \to \{0, 1\}^\MM$ such that $\mathscr{C}_* \PPP$ is the IID $\texttt{Ber}(p)$ labelled Poisson point process for $0 < p < 1$.

\end{thm}

\begin{example}

Some point processes \emph{can} factor onto their own IID, however. Note that taking the IID of a point process is idempotent, in the sense that
\[
    [0,1]^{[0,1]^\Pi} \cong ([0,1]^2)^\Pi \cong [0,1]^\Pi.
\]
For an unlabelled example, one can simply \emph{spatially implement} $[0,1]^\Pi$. That is, using the method sketched at Proposition \ref{abstractlyisom} one can find an unlabelled point process $\Upsilon$ (abstractly) isomorphic to $[0,1]^\Pi$, and thus $[0,1]^\Upsilon \cong \Upsilon$.
\end{example}

\subsection{Cost monotonicity for (certain) weak factors}\label{certainfactors}

In this section we will always assume $G$ is compactly generated by $S \subset G$.

\begin{question}

Suppose $\Pi$ weakly factors onto $\Upsilon$. Is it true that $\cost(\Pi) \leq \cost(\Upsilon)$? That is, is cost monotone for weak factors?

\end{question}

This is the \emph{real} theorem that we would like to prove. We are able only to prove the following theorem, which implies that cost is monotone for certain weak factors:

\begin{thm}\label{costmonotonicity}

Suppose $\Pi^n$ is a sequence of point processes that weakly converge to $\Pi$. Then
\[
    \limsup_{n \to \infty} \cost(\Pi^n) \leq \cost(\Pi)
\]
holds in the following cases:
\begin{enumerate}
    \item If there exists $\delta, R > 0$ such that $\Pi_n$ and $\Pi$ are all $\delta$ uniformly separated and $R$ coarsely connected.
    \item If all the $\Pi_n$ are free and $\Pi$ is $\delta$ uniformly separated.
\end{enumerate}

Moreover, the same statements are true if the point processes have labels from a \emph{compact} mark space $\Xi$.

\end{thm}

We will need an auxiliary lemma, which we will use again later:

\begin{lem}\label{continuousthinning}
    Let $\Pi$ be a point process with law $\mu$. Then for all but countably many $\delta > 0$, the $\delta$-metric-thinning map $\theta^\delta : \MM \to \MM$ is continuous $\mu$ almost everywhere.
    
    In particular, if $\Pi_n$ weakly converges to $\Pi$, then $\theta^\delta(\Pi_n)$ weakly converges to $\theta^\delta(\Pi)$.
\end{lem}

To prove the lemma, simply note that any $\delta$ such that $B_G(0, \delta)$ is a stochastic continuity set for $\Pi_0$ works.

\begin{proof}[Proof of Theorem \ref{costmonotonicity}.]
    
    We prove (1), and then show how to reduce (2) to (1).
    
    By increasing $S$ if necessary, we may assume that the Cayley factor graph of $\Pi_n$ and $\Pi$ with respect to $S$ is connected almost surely.
    
    Denote the distributions of $\Pi_n$ and $\Pi$ by $\mu_n$ and $\mu$ respectively.
    
    We call a factor graph $\mathscr{G}$ a \emph{$\mu$-continuity factor graph} if it has the property that 
    \[
        \lim_{n \to \infty} \muarrow^n(\mathscr{G}) = \muarrow(\mathscr{G}).
    \]
    The same technique used to prove Proposition \ref{palmconvergence} shows that factor graphs of the form\footnote{Let us unpack the definition:  there is an edge between $g, h \in \mathscr{G}_{A,V}(\omega)$ if $g^{-1}\omega \in A$ and $g^{-1}h \in V$. That is, each point $g$ decides if it will have edges (checks if $g^{-1}\omega \in A$), and then simply connects to all points in $gV$.} $\mathscr{G}_{A, V} = (A \times V) \cap \Marrow$, where $A \subseteq \MMo$ is a $\mu_0$ continuity set and $V \subseteq G$ is a bounded stochastic $\mu$ continuity set, are $\mu$-continuity factor graphs. 

    The idea of the proof is that we will take a cheap graphing $\mathscr{G}$ for the limit process $\mu$, and use it to produce a cheap $\mu_0$-continuity graphing $\mathscr{H}$. The continuity property then gives us information about the costs of $\mu_n$, \emph{but only if we can ensure $\mathscr{H}$ is connected on $\Pi_n$}. This is why we assume coarse density.

    Note that by outer regularity of the measure $\muarrow$, for every factor graph $\mathscr{G}$ and $\e > 0$ there exists an \emph{open} factor graph $\mathscr{G}' \supseteq \mathscr{G}$ such that $\muarrow(\mathscr{G}') \leq \muarrow(\mathscr{G}) + \e$. Therefore in the definition of cost one can replace ``measurable graphing'' by ``open graphing''. 

    \begin{claim}
        Every \emph{open} graphing $\mathscr{G}$ of $\mu$ contains a $\mu$-continuity factor graph $\mathscr{H}_N$ such that its $N$th power satisfies $\mathscr{H}_N^N \supseteq \Marrow \cap (H_\delta \times S)$.
    \end{claim}

    Here $H_\delta$ denotes the space of $\delta$ separated configurations as in Lemma \ref{hardcorecompact}.
    
    Note that this condition and the assumption on $R$ and $S$ implies that $\mathscr{H}$ is \emph{connected} on any $\delta$-uniformly separated and $R$ coarsely dense input. In particular, $\mathscr{H}(\Pi_n)$ is connected for every $n$. 
    
    Let us assume that all of $\Pi_n$ and $\Pi$ have unit intensity and finish the proof: for any $\e > 0$, choose a graphing $\mathscr{G}$ of $\Pi$ such that $ \muarrow(\mathscr{G}) \leq \cost(\Pi) + \e$. Take $\mathscr{H}$ as in the above procedure. Then:
    \begin{align*}
        \limsup_{n \to\ \infty} &\cost(\Pi_n) \leq \limsup_{n \to\ \infty} \muarrow^n(\mathscr{H}) && \text{as } \mathscr{H} \text{ is a graphing of } \Pi_n \\
        &= \muarrow(\mathscr{H}) && \text{since } \mathscr{H} \text{ is a } \muarrow\text{-continuity graphing} \\
        &\leq \muarrow(\mathscr{G}) && \text{as } \mathscr{H} \subseteq \mathscr{G} \\
        &\leq \cost(\Pi) + \e && \text{by assumption on } \mathscr{G}.
    \end{align*}

    Since $\e > 0$ was arbitrary, this proves the result for unit intensity processes. In general, the steps are the same except with the more complicated formula for cost of processes with non unit intensity. One just needs to know that if $\Pi_n$ weakly converges to a finite intensity process $\Pi$, then $\intensity(\Pi_n)$ converges to $\intensity(\Pi)$. This follows by weak convergence and choosing a stochastic continuity set $U$ for $\Pi$, noting that the point count function $N_U$ is thus continuous and \emph{bounded} as we are taking uniformly separated processes.
    
    It remains to prove the claim about $\mu$-continuity factor graphs.
    
    Recall from Lemma \ref{continuitysets} that $\MMo$ and $G$ admit \emph{topological bases} $\{A_i\}$ and $\{V_j\}$ consisting of $\mu_0$-continuity sets and $\mu$ stochastic continuity sets (respectively). So by definition of the subspace topology, $\Marrow$ admits a topological basis $\{\mathscr{G}_{A_i,V_j}\}$ consisting of $\mu$-continuity factor graphs.
    
    For each $k \in \NN$ define
    \[
        \mathscr{H}_k = \bigcup_{\substack{i, j \leq k \\ \mathscr{G}_{A_i,V_j} \subseteq \mathscr{G}}} \mathscr{G}_{A_i,V_j}. 
    \]
    Since $\mathscr{H}_k$ consists of \emph{finitely many} continuity factor graphs, it is itself a continuity factor graph. Each $\mathscr{H}_k$ is also open, and increases to $\mathscr{G}$ as $k$ tends to infinity.
    
    As $\mathscr{G}$ is generating, $\{\mathscr{H}_k^k\}_{k \in \NN}$ forms an open cover of the \emph{compact}\footnote{See Lemma \ref{hardcorecompact}.} space $\Marrow \cap (H_\delta \times S)$. In particular, there exists $N$ such that $\mathscr{H}_N^N \supseteq \Marrow \cap H_\delta \times S$, proving the claim.
    
    One sees that the essential feature in the above proof strategy was \emph{compactness}, and therefore it remains true for $\Xi$-labelled point processes if $\Xi$ is compact, as mentioned.
    
    With this observation in hand, we can now deduce the (2) statement from the (1). We will produce a weakly convergent sequence of separated and coarsely connected point processes, where each term has the same cost as $\Pi_n$ and the weak limit factors onto $\Pi$, thus has cost at most the cost of $\Pi$. This proves the statement.
    
    Choose $\delta' < \delta$ as in Lemma \ref{continuousthinning} so that the $\delta'$ metric thinning satisfies
    \[
        \theta^{\delta'}(\Pi_n) \text{ weakly converges to } \theta^{\delta'}(\Pi) = \Pi.
    \]  
    Now observe by label trickery (see Proposition \ref{labeltrickery}) we can find $[0,1]$-labelled point processes $\Upsilon_n$ each isomorphic to the respective $\Pi_n$ and such that their underlying point set is $\theta^{\delta'}(\Pi_n)$.
    
    Note that $\Upsilon_n$ \emph{might not weakly converge}, but it will have subsequential weak limits. All such subsequential weak limits will be some kind of (possibly random) labelling of $\Pi$.
    
    To see this, let $\pi : [0,1]^\MM \to \MM$ be the map that forgets labels. Thus $\pi(\Upsilon_n) = \theta^{\delta'}(\Pi_n)$. Since $\pi$ is continuous, it preserves weak limits. Let $\Upsilon$ be any subsequential weak limit of $\Upsilon_n$, along a subsequence $n_k$. Then by continuity
    \[
        \pi(\Upsilon) = \lim_{k \to \infty} \pi(\Upsilon_{n_k}) = \lim_{k \to \infty} \theta^{\delta'}(\Pi_{n_k}) = \Pi.
    \]
    Now let $\Theta_n(\Upsilon_n)$ be \emph{the input/output versions} of the $(\delta', R)$-Delone thickenings that exist from Proposition \ref{factoronnet}. Here we use that $\Upsilon_n$ are free actions. By input/output we mean you keep track of which points of the thickening are input and output, as in Definition \ref{inputoutputdefn}. In particular,
    \[
        \cost(\Theta_n(\Upsilon_n)) = \cost(\Upsilon_n) = \cost(\Pi_n),
    \]
    where the first equality holds because we took the input/output version of the thickening.
    
    Let $\Upsilon'$ denote any subsequential weak limit of $\Theta_n(\Upsilon_n)$. Then $\Upsilon'$ factors onto $\Pi$, by a similar argument to the earlier one about forgetting certain labels. Putting this all together:
    \[
        \limsup_{n \to \infty} \cost(\Pi_n) = \cost(\Theta_n(\Upsilon_n)) \leq \cost(\Upsilon') \leq \cost(\Pi), 
\]
    where the final inequality holds because cost can only increase under factors.
\end{proof}

\begin{remark}

In the second part of the proof, one might want to replace label trickery by something like ``each $\Pi_n$ is isomorphic to a random Delone set $\Upsilon_n$, which has subsequential weak limits, so choose one such limit $\Upsilon$...'', but then it's not clear what the cost of $\Upsilon$ has to do with the cost of $\Pi$. One would require the Delone-ification process to preserve weak limits in some sense in order to relate $\cost(\Upsilon)$ and $\cost(\Pi)$.

\end{remark}

\begin{remark}

There is label trickery in \cite{abert2013bernoulli} too: it is always assumed there that the action is continuous on a compact space. 

\end{remark}

\begin{thm}\label{poissonmax}

If $\Pi$ is a free point process, then its cost is at most the cost of the IID Poisson point process on $G$.

\end{thm}

\begin{proof}

We know by Theorem \ref{abertweiss} that $\Pi$ weakly factors onto $[0,1]^\Pi$, and $[0,1]^\Pi$ factors onto the IID Poisson. We would like to say ``so $\Pi$ weakly factors onto the IID Poisson, and hence has less cost by the cost monotonicity statement'', but our cost monotonicty statement is too weak for this, so we use a different argument. 

Note that $\Pi$ is \emph{abstractly isomorphic} to a $\delta$ uniformly separated process $\Pi'$ by Proposition \ref{abstractlyisom}. Then $\Pi'$ is also free and has the same cost as $\Pi$. Now $\Pi'$ weakly factors onto its own IID, more explicitly, there is a sequence of factors \emph{labellings} $\Phi_n(\Pi')$ weakly converging to $[0,1]^{\Pi'}$. Because $\Phi_n(\Pi')$ is a labelling of $\Pi'$ it is itself free and uniformly separated. Putting it all together:
\begin{align*}
    \cost(\Pi) &= \cost(\Pi') && \text{as they are isomorphic actions} \\
    &\leq \limsup_{n \to \infty} \cost(\Phi_n(\Pi')) && \text{cost can only increase for factors} \\
    &\leq \cost([0,1]^{\Pi'}) && \text{by Theorem \ref{costmonotonicity}} \\
    &\leq \cost([0,1]^\PPP) && \text{as } [0,1]^{\Pi'} \text{ factors onto } [0,1]^\PPP,
\end{align*}
as desired.
\end{proof}

\section{Some fixed price one groups}

Theorem \ref{poissonmax} is a strategy for proving that groups have fixed price, and we will use it to that end for $G \times \ZZ$. However, our argument is somewhat indirect and requires taking a further weak limit. It would be instructive to have a more direct argument:

\begin{question}

Can one \emph{explicitly} construct for every $\e > 0$ connected factor graphs of the IID Poisson on $G \times \ZZ$ of edge measure less than $1 + \e$?

\end{question}

To see what we mean by \emph{explicit}, one should consider the discrete case: if $\Gamma$ is a finitely generated group, then it is straightforward to construct factor of IID connected graphs of small edge measure on Bernoulli (site) percolation on $\Gamma \times \ZZ$. We would like a construction in that vein. 

So instead we will use the weak factoring strategy to reduce the above problem to a much simpler one, where we \emph{can} construct such factor graphs.

\subsection{Groups of the form \texorpdfstring{$G \times \ZZ$}{G x Z}}

\begin{defn}

Let $\Pi$ be a point process on $G$. Its \emph{vertical coupling} on $G \times \ZZ$ is $\Delta(\Pi) = \Pi \times \ZZ$.   

\end{defn}

Here $\Delta : \MM(G) \to \MM(G \times \ZZ)$ is induced by the diagonal embedding of $G$ into $G^\ZZ$. For this reason one might prefer to call $\Delta(\Pi)$ the \emph{diagonal coupling}, but this terminology will not be suitable when we go to $G \times \RR$.

\begin{figure}[h]
\includegraphics[scale=0.4]{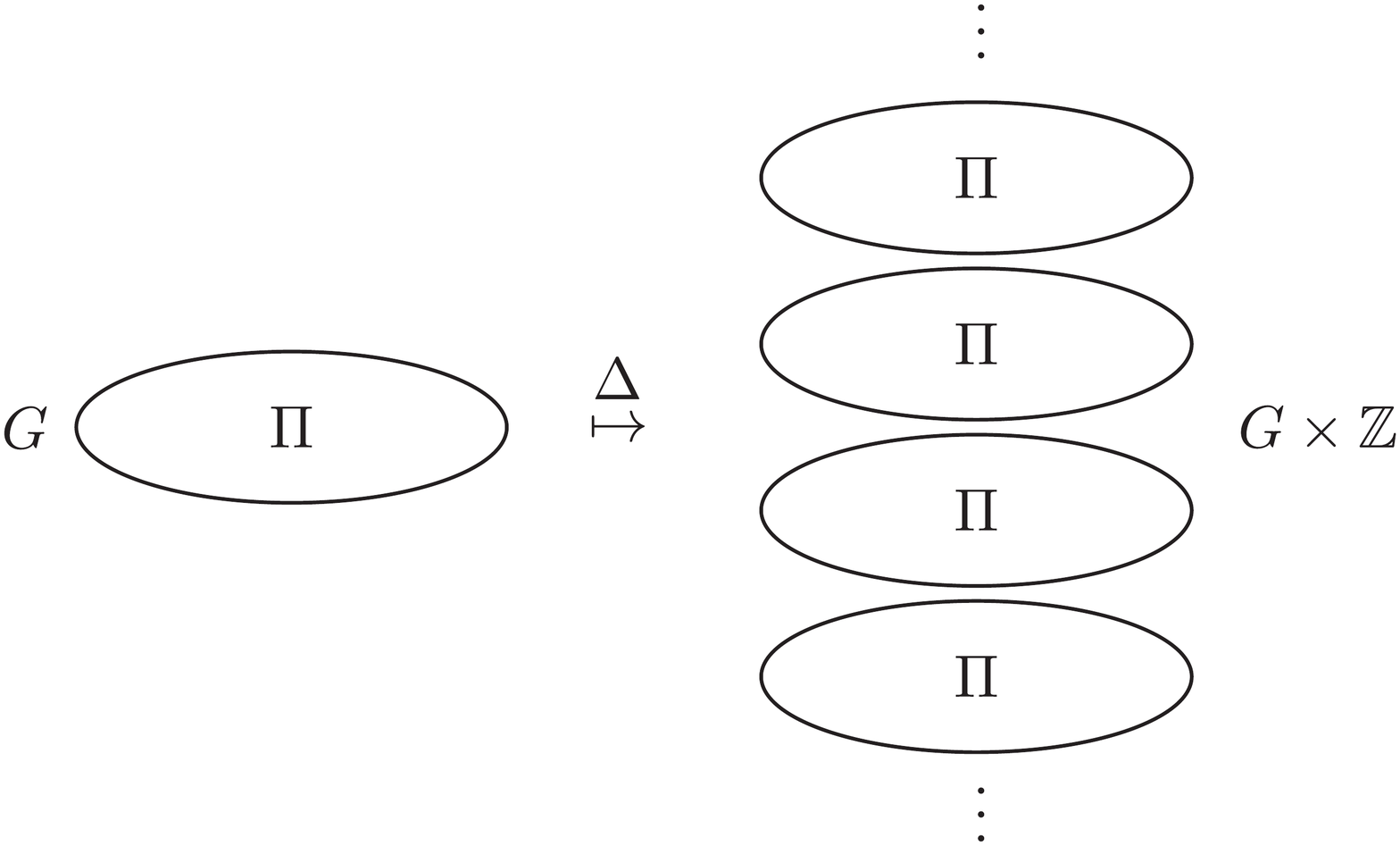}
\centering
\caption{How one should think of $G \times \ZZ$. Note that if $\Pi$ is a point process on $G$, then its vertical coupling is simply infinitely many copies of \emph{the same} points stacked on top of each other.}
\end{figure}
\begin{lem}\label{verticalcost}

The IID version $[0,1]^{\Delta(\Pi)}$ of a vertically coupled process has cost one.

\end{lem}

The proof uses the fact that Bernoulli percolation of a factor graph can be implemented as a factor of IID. This sort of trick will be familiar to many, but we will nevertheless spell it out:

\begin{defn}
    
    Let $\mathscr{G}$ be a factor graph of a point process $\Pi$. Its \emph{$\e$ edge percolation} is the factor graph $\mathscr{G}_\e$ defined on $[0,1]^\Pi$ in the following way: for points $g, h \in [0,1]^\Pi$ let
    \[
        g \sim_{\mathscr{G}_\e} h \text{ whenever } g \sim_{\mathscr{G}} h \text{ and } \xi_g \oplus \xi_h < \e.
    \]
    Here $\oplus$ denotes addition of the labels modulo one.
\end{defn}

Observe that if $(g, h_1)$ and $(g, h_2)$ are edges of $\mathscr{G}(\Pi)$, then the random variables $\xi_g \oplus \xi_{h_1}$ and $\xi_g \oplus \xi_{h_2}$ are independent uniform once again. 

\begin{remark}

If $\mathscr{G}$ is already a factor graph defined on $[0,1]^\Pi$, then we can implement $\mathscr{G}_\e$ on $[0,1]^\Pi$, that is, without adding further randomness (via the replication trick, see the discussion below).

\end{remark}

\begin{proof}[Proof of Lemma \ref{verticalcost}]
Let $\mathscr{G}$ be any graphing of $\Pi$ with finite edge density. We \emph{lift} it to a factor graph $\mathscr{G}^\Delta$ of $\Delta(\Pi)$ in the following way:
    \[
        (g, n) \sim_{\mathscr{G}^\Delta(\Pi)} (h, n) \text{ if and only if } g \sim_{\mathscr{G}(\Pi)} h,
    \]
that is, as $\Delta(\Pi)$ is just copies of $\Pi$ stacked on every level of $G \times \ZZ$, then we simply copy $\mathscr{G}$ onto every level of $G \times \ZZ$ as well.

Let $\mathscr{V}$ denote the factor graph of $\Delta(\Pi)$ consisting of \emph{vertical} edges, that is for every $(g, n) \in \Delta(\Pi)$ we have an edge to $(g, n+1)$.

One can see that $\mathscr{V} \cup \mathscr{G}^\Delta$ is a connected factor graph. But this is also true when we percolate the edges level wise, that is, when we consider $\mathscr{V} \cup \mathscr{G}^\Delta_\e$. This is because if $(g, n) \sim_{\mathscr{G}^\Delta} (h, n)$ is an edge destroyed in the percolation, then we can slide up along vertical edges and consider the edge $(g, n+1) \sim_{\mathscr{G}^\Delta} (h, n+1)$ instead. Its chance of survival in the percolation is independent from the previous edge, and hence we get another go to cross over. By sliding up far enough we are guaranteed to be able to cross. Finally, observe that the edge density of $\mathscr{V} \cup \mathscr{G}^\Delta_\e$ is $1$ plus $\e$ times the edge density of $\mathscr{G}$.
\end{proof}

\begin{lem}\label{factorimpliesiidfactor}

Suppose $\Pi$ and $\Upsilon$ are point processes, and $\Pi$ factors onto $\Upsilon$. Then $[0,1]^\Pi$ factors onto $[0,1]^\Upsilon$. In particular, if $[0,1]^\Pi$ weakly factors onto $\Upsilon$ then $[0,1]^\Pi$ weakly factors onto $[0,1]^\Pi$ too.

\end{lem}

The proof of this uses the following \emph{replication trick}: note that the randomness in one $\texttt{Unif}[0,1]$ random variable $\xi$ is equivalent to the randomness in an entire IID sequence $\xi_1, \xi_2, \ldots$ of $\texttt{Unif}[0,1]$ random variables.

More precisely, there is an isomorphism (as measure spaces)
\[
I : ([0,1], \text{Leb}) \to ([0,1]^\NN, \text{Leb}^{\otimes \NN}).
\]
So if $\xi \sim \texttt{Unif}[0,1]$, then we will write $I(\xi) = (\xi_1, \xi_2, \ldots)$ for the associated IID sequence of $\texttt{Unif}[0,1]$ random variables. 

\begin{proof}[Proof of Lemma \ref{factorimpliesiidfactor}]
Suppose $\Upsilon = \Phi(\Pi)$. If $g \in [0,1]^\Pi$, then we write $\xi^g$ for its label, and by the replication trick $\xi^g_1, \xi^g_2, \ldots$ for the associated IID seqence of $\texttt{Unif}[0,1]$ random variables.

We define a factor map $\widetilde{\Phi}$ of $[0,1]^\Pi$ as follows:
\[
    \widetilde{\Phi}([0,1]^\Pi) = \bigcup_{g \in [0,1]^\Pi} \{(h_i, \xi^g_i) \in G \times [0,1] \mid V_g(\Pi) \cap \Phi(\Pi) = (h_1, h_2, \ldots, ) \},
\]
where we mean that $(h_1, h_2, \ldots)$ is any \emph{enumeration} of $V_g(\Pi) \cap \Phi(\Pi)$ performed in an equivariant way. Note that the intersection $V_g(\Pi) \cap \Phi(\Pi)$ could be empty.

For instance, look at the elements of $h \in V_g(\Pi) \cap \Phi(\Pi)$ which are \emph{closest} to $g$. Then let $h_1$ be the element that minimises $T(g^{-1}h)$, where $T : G \to [0,1]$ is the tie-breaking function of Section \ref{voronoidefn}. Then let $h_2$ be the next smallest element, and so on, until you exhaust the closest elements. Then look at the batch of next closest elements and so on. One can check that this is an equivariant construction (any construction where you do the same thing at every point will be).

Then $\widetilde{\Phi}([0,1]^\Pi) = [0,1]^\Upsilon$, as desired.

For the second part, simply note that taking the IID is idempotent in the sense that $[0,1]^{([0,1]^\Pi)} \cong [0,1]^\Pi$, and apply Lemma \ref{weakfactorimpliesiiweakfactor}.
\end{proof}

\begin{prop}\label{GtimesZweakfactor}
    
    The IID Poisson on $G \times \ZZ$ weakly factors onto the vertically coupled Poisson of $G$.
    
\end{prop}

\begin{proof}
    We will construct factor maps $\Phi^n : [0,1]^{\MM(G \times \ZZ)} \to \MM(G \times \ZZ)$ that ``straighten'' the input in the following way: for a given input $\omega \in [0,1]^\MM$, we select a ``sparse'' subset of its points. At each one of these we \emph{propagate} them upwards by placing copies of them on the levels above. This will converge to a vertically coupled process for suitable inputs.
    
    More precisely, let $\Pi$ denote the (unit intensity) IID Poisson on $G \times \ZZ$. We will denote points of $\Pi$ by $(g, l) \in G \times \ZZ$, and write $\Pi_{g, l}$ for its label. 
    
    We now define the factor map $\Phi^n$ in two stages as a thinning and then a thickening to simplify the analysis. Let
    \[
        \Pi^{1/n} = \{(g, l) \in \Pi \mid \Pi_{g, l} \leq \frac{1}{n} \},
    \]
    and $F_n = \{0\} \times \{0, 1, \ldots, n-1\}$. Set
    \[
        \Phi^n \Pi = \Theta^{F_n}(\Pi^{1/n}),
    \]
    where we write $\Phi^n \Pi$ for $\Phi^n(\Pi)$ to conserve parentheses. 
    
    Let us explain what this means:
    \begin{itemize}
    \item At the first step $\Pi \mapsto \Pi^{1/n}$, we independently thin $\Pi$ to get a subprocess of intensity $\frac{1}{n}$. By the discussion in Example \ref{independentthinning}, the resulting process $\Pi^{1/n}$ is simply a Poisson point process on $G \times \ZZ$ of intensity $\frac{1}{n}$. We refer to the points of $\Pi^{1/n}$ as \emph{progenitors}.
    \item Each progenitor $(g, l)$ spawns additional points with the same $G$-coordinate on the next $n-1$ levels above it. This is the map $\Pi^{1/n} \mapsto \Theta^{F_n}(\Pi^{1/n}) = \Phi^n \Pi $.
    \item By the discussion at Example \ref{constantthickening}, $\Phi^n \Pi$ is a process of unit intensity.
    \end{itemize}
    
    We will employ the following strategy to show that $\Phi^n \Pi$ weakly converges to the vertical Poisson:
    \begin{enumerate}
        \item The sequence $(\Phi^n \Pi)$ admits weak subsequential limits, which a priori might be random counting measures, 
        \item These subsequential limits are actually simple point processes,
        \item All of these subsequential limits are vertical processes, and
        \item That process is the vertical Poisson.
    \end{enumerate}
    
    Recall that if $(x_n)$ is a relatively compact sequence and every subsequential limit of $(x_n)$ is $x$, then $x_n$ converges to $x$.
    
    By this basic fact and the above items, we can conclude that $(\Phi^n \Pi)$ weakly converges to the vertical Poisson. 
    
    We now verify that $\{ \Phi^n \Pi \}$ is \emph{uniformly tight}, proving (1). It suffices to verify that the distributions of point counts $N_C(\Phi^n \Pi)$, where $C = B_G(0, r) \times [L]$ denotes a cylinder whose base is a ball of radius $r$ and its height (in levels) is $L$, are uniformly tight.
    
    Let $X_i$ denote the number of points in $B_G(0, r) \times \{i\}$ with label $\Pi_{g, i} \leq \frac{1}{n}$, that is, the number of progenitors on the $i$th level. Thus the $X_i$ are IID Poisson random variables with parameter $\lambda(B_G(0,r)) / n$.
   
   One can explicitly describe the random variable $N_C(\Phi^n \Pi)$ in terms of the $X_i$s, but for our purposes it is enough to observe that:
   
    \[
        N_C(\Phi^n \Pi) \leq L \sum_{i=1}^n X_i.
    \]
    
    The sum of independent Poisson random variables is again Poisson distributed (with parameter the sum of the parameters of the individual Poissons), so we see that $N_C(\Phi^n \Pi)$ is bounded in terms of a random variable \emph{that does not depend on $n$}. Therefore $\{\Phi^n \Pi\}$ is uniformly tight.
    
    To prove item (2), note that the above shows that the point counts in $B_G(0, r) \times \{0\}$ for $\Phi^n \Pi$ are \emph{exactly} Poisson distributed with parameter $\lambda(B_G(0, r))$. Thus if $\Upsilon$ is any subsequential weak limit of $\Phi^n \Pi$ and $r$ is such that $B_G(0,r) \times \{0\}$ is a stochastic continuity set for $\Upsilon$, then $N_{B_G(0,r) \times \{0\}} \Upsilon$ will also be Poisson distributed. In particular, $\Upsilon$ must be a \emph{simple} point process.
    
    For item (3), let $\Upsilon$ be any subsequential weak limit of $\Phi^n \Pi$. Observe that $\Upsilon$ is \emph{vertical} if and only if $(g, l) \in \Upsilon$ implies $(g, l+1) \in \Upsilon$. The idea is that this property is satisfied for most points of $\Phi^n \Pi$, and therefore must be preserved in the weak limit. Note that a process is vertical if and only if its Palm measure is vertical almost surely.

    We can now explicitly describe the Palm measure of $\Phi^n(\Pi)$. Recall from Theorem \ref{palmofpoisson} that the Palm version $\Pi_0^{1/n}$ of $\Pi^{1/n}$ is simply $\Pi^{1/n} \cup \{(0,0)\}$. 
    
    To express the Palm version of the $F_n$-thickening of $\Pi^{1/n}$ (\`{a} la Example \ref{palmofthickening}), it will be useful to introduce the following notation. For each $k \in \NN$, let
    \[
        \Pi_k^{1/n} = \Pi_0^{1/n} \cdot (0, k) = \{ (g, l+k) \in G \times \ZZ \mid (g, l) \in \Pi_0^{1/n} \}.
    \]
    That is, you simply shift $\Pi_0^{1/n}$ up by $k$ levels.
    Then $\Phi^n(\Pi_0) = \Pi_0^{1/n} \cup \Pi_1^{1/n} \cdots \cup \Pi_{n-1}^{1/n}$.
    
    Denote by $K$ a random integer chosen uniformly from $\{0, 1, \ldots, n-1\}$. Then the Palm version of $\Phi^n(\Pi)$ is
    \[
        (\Phi^n \Pi)_0 = \Pi_{-K}^{1/n} \cup \Pi_{-K+1}^{1/n} \cup \cdots \cup \Pi_{-K+n-1}^{1/n},
    \]
    where we use parentheses to stress that it is the Palm version of $\Phi^n \Pi$, not $\Phi^n$ applied to $\Pi_0$.
    
    Let us say that a rooted configuration $\omega \in \MMo(G \times \ZZ)$ \emph{has an $\e$-successor} if there is a point approximately above the root $(0,0)$ in $\omega$. More precisely, we define an \emph{event}
    \[
        \{ \omega \text{ has an } \e\text{-successor} \} := \{ \omega \in \MMo \mid N_{B_G(0, \e) \times \{1\}} \omega > 1 \}.
    \]

    From this, we see
    \[
    \PP[(\Phi^{n} \Pi)_0 \text{ has an } \e\text{-successor} ] \geq \frac{n-1}{n},
    \]
    as $(\Phi^{n} \Pi)_0$ certainly has an $\e$-successor whenever $K < n-1$. 
    
    Recall that $\Upsilon$ was any subsequential weak limit of $\Phi^n \Pi$. Fix a subsequence $n_i$ such that $\Phi^{n_i} \Pi$ weakly converges to $\Upsilon$.
    
    Choose a sequence $\e_k$ tending to zero such that $B_G(0, \e_k) \times \{1\}$ is a stochastic continuity set for $\Upsilon$. This is possible by Lemma \ref{continuitysets}. Then for each $k$
    \[
        \frac{n_i-1}{n_i} \leq \PP[\Phi^{n_i} (\Pi)_0 \text{ has an } \e_k\text{-successor} ] \to \PP[\Upsilon_0 \text{ has an } \e_k\text{-successor} ],
    \]

    So $\Upsilon_0$ has $\e_k$-successors almost surely for every $k$, and hence has $0$-successors. That is, $\Upsilon$ is a vertical process, at last proving item (3).
    
    Finally, for item (4) we observe that any vertical process is completely determined by its intersection with $G \times \{0\}$. We observed in the proof of item (2) that $\Upsilon$ is a Poisson point process on the $0$th level, so it must be the vertical Poisson, as desired.
\end{proof}

\begin{cor}\label{GtimesZcost}
    
    Groups of the form $G \times \ZZ$ have fixed price one.
    
\end{cor}

\begin{proof}
By the previous proposition and Lemma \ref{factorimpliesiidfactor}, we know that the IID Poisson weakly factors onto the IID of the vertically coupled Poisson. Explicitly, there exists factor maps $\Phi^n : [0,1]^\MM \to [0,1]^\MM$ such that
\[
    \Phi^n \Pi \text{ weakly converges to } [0,1]^{\Delta(\PPP)},
\]
where $\Pi$ is the IID Poisson on $G$ and\footnote{This is a slight abuse of notation: we were using $\PPP$ to denote the \emph{law} of the Poisson point process, but in the above expression we treat it as if it were a random variable. We do this to prevent the profusion of asterisks representing pushforwards of measures.} $\PPP$ is the Poisson on $G$.

Choose $\delta < 1$ as in Lemma \ref{continuousthinning} such that metric $\delta$-thinning preserves the weak limit. Note that because $\delta < 1$, the thinning commutes with the vertical coupling: that is, $\theta^\delta(\Delta (\PPP))= \Delta(\theta^\delta \PPP)$. Therefore
\[
    \theta^\delta(\Phi^n \Pi) \text{ weakly converges to } [0,1]^{\Delta(\theta^\delta (\PPP))}.
\]
Putting this all together,
\begin{align*}
    \cost(\Pi) &\leq \limsup_{n \to \infty} \cost(\theta^\delta(\Phi^n \Pi)) && \text{As cost can only increase under factors} \\
    &\leq \cost([0,1]^{\Delta(\theta^\delta (\PPP))}) && \text{By Theorem \ref{costmonotonicity}} \\
    &= 1 && \text{By Lemma \ref{verticalcost}}.
\end{align*}
Since the IID Poisson has maximal cost, this proves that $G \times \ZZ$ has fixed price one.
\end{proof}

\begin{remark}

With further percolation-theoretic assumptions on $G$, one can \emph{directly} show that $\cost(\Phi^n(\Pi)) \leq 1 + \e_n$, where $\e_n$ tends to zero. This is by constructing factor graphs on $\Phi^n(\Pi)$. 

By using the Poisson net, one can prove an analogue of the Babson and Benjamini theorem \cite{10.2307/119068} and show that the distance $\mathscr{D}_R$ factor graph on the Poisson point process on a \emph{compactly presented} and one-ended group has a \emph{unique} infinite connected component if $R$ is sufficiently large. 

Now on $\Phi^n(\Pi)$, we construct a factor graph as follows: add in all vertical edges, and the $\mathscr{D}_R$ edges horizontally. Now percolate the horizontal edges. One can show that by adding a small amount of edges to this, the result is a graph with a unique infinite connected component.

\end{remark}

\subsection{Groups of the form \texorpdfstring{$G \times \RR$}{G x R}}

We now outline the modifications required to extend the $G \times \ZZ$ case to the following theorem:

\begin{thm}\label{GtimesRfp}
    
    Groups of the form $G \times \RR$ have fixed price one.
    
\end{thm}

\begin{proof}

The strategy will be exactly the same as in Proposition \ref{GtimesZweakfactor}.

We define factor maps $\Phi^n$ of the IID Poisson $\Pi$ using the same formula as in the $G \times \ZZ$ case. We claim these weakly converge to a point process $\Upsilon$ which is \emph{vertical} in the sense that $(g, t) \in \Upsilon$ implies $(g, t + n) \in \Upsilon$ for all $n \in \ZZ$.

First we show $\{\Phi^n (\Pi)\}$ is uniformly tight. This works exactly as in the $G \times \ZZ$ case, except instead of counting progenitors $X_i$ on $G \times \{i\}$, we count them on $G \times [i, i+1)$ for $i \in \ZZ$.

Next we show that any subsequential weak limit $\Upsilon$ of $\{ \Phi^n (\Pi) \}$ is not just a random counting measure, but an actual point process. This follows as in the $G \times \ZZ$ case, as $\Phi^n (\Pi)$ has the same distribution in $G \times [0, 1)$ as a Poisson point process on $G \times \RR$.

The proof that $\Upsilon$ is a vertical point process works the same as in the $G \times \ZZ$ case.

At this point one can observe that a vertical process is determine by its intersection with $G \times [0,1)$, and therefore $\Phi^n (\Pi)$ weakly converges to a unique point process $\Upsilon$. 

We now adapt Lemma \ref{verticalcost} to this context, and show that if $\Upsilon$ is \emph{any} vertical point process such that the projection $\pi(\Upsilon)$ has finite cost, then the IID process $[0,1]^\Upsilon$ has cost one.

Let $\pi : G \times \RR \to G$ denote the projection map. Observe that \emph{if $\Upsilon$ is vertical}, then $\pi(\Upsilon)$ is discrete, and hence defines a point process on $G$. For contrast, observe that the projection $\pi(\Pi)$ of the Poisson point process $\Pi$ is almost surely dense, and hence does not define a point process on $G$. In the case of the $\Upsilon$ we construct as a weak limit, its projection $\pi(\Upsilon)$ is just the Poisson point process on $G$.

Let $\mathscr{G}$ be a finite cost graphing of $\pi(\Upsilon)$. We lift this to a factor graph of $\Upsilon$ in the following way:
\[
    (g_1, t_1) \sim_{\mathscr{H}(\Upsilon)} (g_2, t_2) \text{ when } g_1 \sim_{\mathscr{G}(\pi(\Upsilon))} g_2 \text{ and } \abs{t_1 - t_2}  < 1.
\]
Let $\mathscr{V}(\Upsilon)$ denote the set of \emph{vertical edges}, that is
\[
    \mathscr{V}(\Upsilon) = \{ ((g, t), (g, t + 1))  \mid (g, t) \in \Upsilon \}.
\] 
Then as in Lemma \ref{verticalcost}, the vertical edges $\mathscr{V}(\Upsilon)$ together with an $\e$-percolation of $\mathscr{H}(\Upsilon)$ defines a cheap connected factor graph of $\Upsilon$.

\begin{figure}[h]
\includegraphics[scale=0.5]{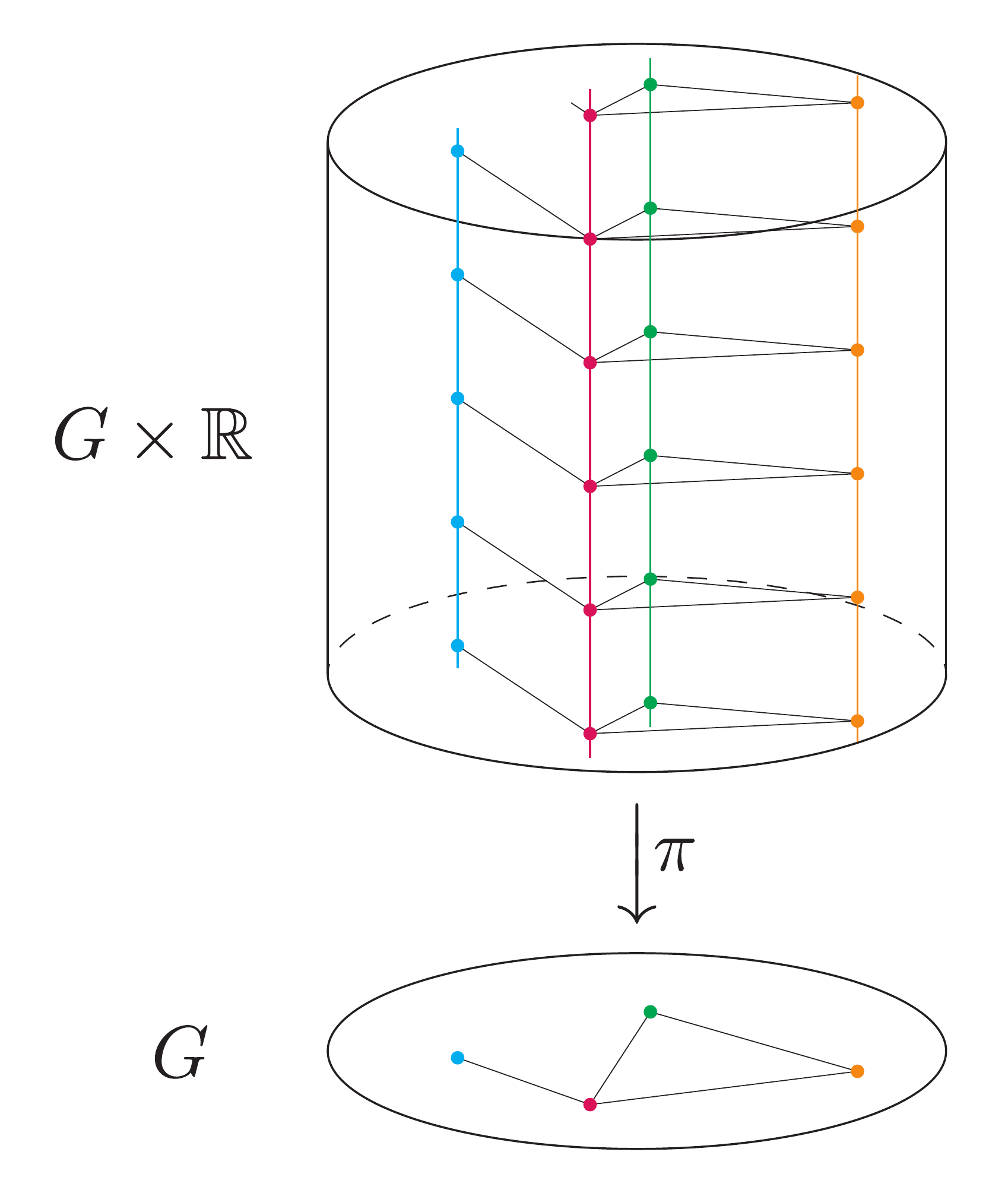}
\centering
\caption{A portion of a graphing on the projection of a vertical process, and how it might look when lifted. Note that it gets wobbled a bit in the process.}
\end{figure}

We conclude from this that $G \times \RR$ has fixed price one by the same kind of reasoning as in Corollary \ref{GtimesZcost}.
\end{proof} 

\begin{remark}

The limiting process here can be described as follows: sample from a Poisson on $G \times [0, 1)$, and then simply extend it periodically in the second coordinate. 

\end{remark}

\section{Rank gradient of Farber sequences vs. cost}

In this section we will discuss how to connect rank gradient to the cost of the Poisson point process in certain situations. We first work with general lcsc groups, and then specialise to semisimple Lie groups and prove a stronger theorem.

\begin{defn}
Let $(\Gamma_n)$ denote a sequence of lattices in a fixed group $G$. 

The sequence is \emph{Farber} if for every compact neighbourhood of the identity $V \subseteq G$ we have
\[
    \PP[a\Gamma_n a^{-1} \cap V = \{e\}] \to 1 \text{ as } n \to \infty,
\]
where $a\Gamma_n$ denotes a coset of $\Gamma_n$ chosen randomly according to the (normalised) finite $G$-invariant measure on $G/\Gamma_n$.
\end{defn}

Note that $a \Gamma_n a^{-1}$ is exactly the stabiliser of $a\Gamma_n$ for the action $G \acts G/\Gamma$. Thus the Farber condition says that the action on most points of the quotient is locally injective.

Equivalently, the condition states that $a\Gamma_n \cap Va = \{a\}$ with high probability. It is this second form that we will actually use in the proof below. We think of $a$ as being a point sampled randomly from a fundamental domain for $\Gamma_n$ in $G$, and thus it states that the $V$-neighbourhood around this point $a$ meets the lattice shift $a\Gamma_n$ only trivially.

\begin{defn}

Let $(\Gamma_n)$ denote a sequence of lattices in a fixed group $G$. Its \emph{rank gradient} is
\[
   \RG(G, (\Gamma_n)) = \lim_n \frac{d(\Gamma_n) - 1}{\covol \Gamma_n},
\]
whenever this limit exists.

\end{defn}

\begin{remark}

If $G$ is discrete, then the $\Gamma_n$ are all finite index subgroups. The Nielsen-Schreier formula
\[
    \frac{d(\Gamma_n) - 1}{[G : \Gamma_n]} \leq d(G) - 1
\]
shows that the terms in the rank gradient are at least bounded.

Gelander proved\cite{MR2863908} an analogue of this formula for lattices in connected semisimple Lie groups without compact factors.

In the Seven Samurai paper\cite{MR3664810}, it is shown that if $G$ is a centre-free semisimple Lie group of higher rank with property (T), then \emph{any} sequence of irreducible lattices $(\Gamma_n)$ in $G$ is automatically Farber, as long as $\covol(\Gamma_n)$ tends to infinity.
\end{remark}

In the particular case of a decreasing \emph{chain} $\Gamma = \Gamma_1 > \Gamma_2 > \ldots$ of finite index subgroups, Ab\'{e}rt and Nikolov showed \cite{MR2966663} that the rank gradient $RG(\Gamma, (\Gamma_n))$ can be described as the \emph{groupoid} cost of an associated pmp action $\Gamma \acts \partial T(\Gamma, (\Gamma_n))$ on the boundary of a rooted tree.

\subsection{Cocompact lattices in general groups}
\begin{defn}

We say that a lattice $\Gamma < G$ is \emph{$\delta$-uniformly discrete} if all of its \emph{right} cosets $\Gamma a \in \Gamma \backslash G$ are $\delta$ uniformly separated as subsets of $G$. That is, for all distinct pairs $\gamma_1, \gamma_2 \in \Gamma$, we have $d(\gamma_1 a, \gamma_2 a) \geq \delta$. Equivalently by left-invariance of the metric, $d(e, a^{-1} \gamma a ) \geq \delta$ for all $\gamma \in \Gamma$ not equal to the identity $e \in G$.

If $(\Gamma_n)$ is a sequence of lattices, then we say it is $\delta$ uniformly discrete if each $\Gamma_n$ is $\delta$ uniformly discrete in the above sense. 

\end{defn}

\begin{thm}\label{farbertheorem}

Let $(\Gamma_n)$ be a Farber sequence of \emph{cocompact} lattices. Suppose further that the sequence is \emph{uniformly discrete}. If its rank gradient exists, then
\[
    \RG(G, (\Gamma_n)) \leq c_P(G) - 1,
\]
where $c_P(G)$ denotes the cost of the Poisson point process on $G$. In particular, if $G$ has fixed price one then the rank gradient vanishes.
\end{thm}

The above theorem is spiritually the same as one proved independently by Carderi in \cite{carderi2018asymptotic}, but in a drastically different language (namely, that of ultraproducts of actions). The theorem is therefore his, but we include our own proof as it has a different flavour. In the subsequent section we will discuss a similar theorem which applies for nonuniform lattices, at least with additional assumptions on the group.

\begin{proof}[Proof of Theorem \ref{farbertheorem}]
 
 Recall that the cost of a lattice shift is
 \[
    \cost(G \acts G/\Gamma_n) = 1 + \frac{d(\Gamma_n) - 1}{\covol \Gamma_n},
 \]
 which is essentially the term appearing in the rank gradient definition. We would therefore like to take a weak limit of these actions to get some free point process, and then appeal to the cost monotonicity result. Of course, this is completely senseless: the intensity of the lattice shift tends to zero, so it weak limits on the empty process.
 
 Therefore we \emph{thicken} the lattice shifts to get processes $\Pi_n$ with a nontrivial weak limit. This thickening procedure must be done correctly, so that we can apply our (weak) cost monotonicity result.
 
 We will produce a sequence of $[0,1]$-marked point processes $\Pi_n$ such that
 \begin{itemize}
    \item each $\Pi_n$ is a $2\delta$-net,
     \item each $\Pi_n$ is a factor of the lattice shift $a\Gamma_n$, and so has cost \emph{at least} $1+\frac{d(\Gamma_n) - 1}{\covol \Gamma_n}$, and
     \item they have a subsequential weak limit $\Upsilon$ with IID $[0,1]$ labels.
 \end{itemize}
 
Then
\[
    \RG(G, (\Gamma_n)) + 1 \leq \limsup_{n \to \infty} \cost(\Pi_n) \leq \cost(\Upsilon) \leq \mathrm{c}_\mathrm{P}(G)
\]
 by the cost monotonicity result, as desired.
 
View the space of \emph{right} cosets $\Gamma_n \backslash G$ as a compact metric space, where the distance between two cosets $\Gamma_n b_1, \Gamma_n b_2 \in \Gamma_n \backslash G$ is just their distance as closed subsets of $G$.
 
 Let $B_n = \{\Gamma_n b_1^n, \Gamma_n b_2^n, \ldots, \Gamma_n b^n_{k_n} \}$ be a collection of $2\delta$-nets in $\Gamma_n \backslash G$, where $\delta$ is the uniform discreteness parameter. We also choose $b_1^n = e$ for all $n$.
 
 This specifies a sequence of \emph{thickenings} $\Theta_n$ of the corresponding lattice shifts: that is, $a\Gamma_n \mapsto a B_n$.
 
Note that $\Theta_n(a\Gamma_n)$ is a $2\delta$-net: it's true that $d(a\Gamma_n b_i^n, a\Gamma_n b_j^n) = d(\Gamma_n b_i^n, \Gamma_n b_j^n) \geq \delta$ for $i \neq j$ by our choice of $B_n$, and points of $\Gamma_n b_i^n$ are uniformly separated too exactly by our uniform discreteness assumption. It is also $2\delta$-coarsely dense, by the same property for $B_n$.
 
Since $\{\Theta_n(a\Gamma_n)\}$ is a collection of random $2\delta$-nets, it is automatically uniformly tight, and all subsequential weak limits are $2\delta$-nets (and in particular, simple point processes).
 
At this point we would like to apply cost monotonicity to $\Theta_n$ by passing to a subsequential weak limit. Our issue here is that one would have to demonstrate that this weak limit is a \emph{free} action in order to compare its cost to the cost of the Poisson. To do this, one would have to use the Farber condition in an essential way.
 
 We bypass this by a labelling trick: note that the IID of \emph{any} point process is automatically a free action (as any two points of it will receive distinct values almost surely, killing any possible symmetries). So we will limit on an IID labelled process instead.
 
  Consider the action $G \acts G/\Gamma_n \times [0,1]$, where the action on the second coordinate is trivial. We refer to this as the \emph{periodic IID lattice shift}. It is simply the lattice shift, but with every point of it receiving \emph{the same} IID label. This is of course distinct from the IID of the lattice shift, but note
  \[
    \cost(G \acts G/\Gamma_n) = G \acts G/\Gamma_n \times [0,1]
  \]
  as cost is the integral of the cost over the ergodic decomposition (see Corollary 18.6 of \cite{kechris2004topics}). We write $(a\Gamma_n, \xi)$ for a sample from this space.
  
  Now we thicken as before, but this time distribute labels: let
 \[
    \Theta_n(a \Gamma_n, \xi) = \bigcup_{i = 1}^{k_n} a \Gamma_n b_i^n \times \{\xi_i\},
 \]
 where $\xi_1, \xi_2, \ldots$ is an IID sequence of $\texttt{Unif}[0,1]$ random variables produced as a factor of $\xi$.
 
 In other words, each point of the lattice shift adds points to the space and gives them an IID $[0,1]$ label (but note that each point starts with \emph{the same} label $\xi$, so this is not the IID of the thickening).
 
Let $\Upsilon$ denote any subsequential weak limit of $\Theta_n(a \Gamma_n, \xi)$, and pass to that subsequence. Then $\pi(\Theta_n(a \Gamma_n, \xi)$ weakly converges to $\pi(\Upsilon)$, as the projection $\pi: G \times [0,1] \to G$ that forgets labels is continuous.

Our task is to show that $\Upsilon = [0,1]^{\pi(\Upsilon)}$. 

The idea of the proof is the following: fix $C \subseteq G$ to be a bounded stochastic continuity set for $\Upsilon$. We want to prove that the labels of the points of $\Theta_n(a \Gamma_n, \xi)$ in $C$ are independent and uniform on $[0,1]$. They are already $\texttt{Unif}[0,1]$ by definition, so we must now consider their dependencies. Again, by definition, points of $C$ arising from \emph{distinct} $a\Gamma_n b_i^n$ are automatically independent. The only dependency issue that can arise is when $a\Gamma_n b_i^n \cap C$ has \emph{multiple} points. We will show that this is a vanishingly rare event.

This will be achieved via the following lemma:

\begin{lem}\label{farbermult}

Let $C \subseteq G$ be compact. If $(\Gamma_n)$ is a Farber sequence and $B_n \subseteq G$ is any sequence of finite subsets, then $\PP[ \exists b \in B_n \text{ such that } \abs{a\Gamma_n b \cap C} > 1] \to 0$.
\end{lem}

\begin{proof}

Apply the Farber condition with any set $V$ containing $CC^{-1}$. If $b \in B_n$ is such that there are $a\gamma_1, a\gamma_2$ \emph{distinct} elements of $a\Gamma_n b \cap C$, then
\[
(a\gamma_1 b)(a \gamma_2 b)^{-1} = a \gamma_1 \gamma_2^{-1} a^{-1} \text{ is in } CC^{-1},
\]
so $a \gamma_1 \gamma_2^{-1} a^{-1} ab = a \gamma_1 \gamma_2^{-1} b \in Vab$, and this element is also in $a\Gamma_n b$. By the Farber condition,
\[
    a\Gamma_n b \cap Vab = \{ab\}
\]
with high probability, and so
\[
    \PP[ \exists b \in B_n \text{ such that } \abs{a\Gamma_n b \cap C} > 1] \leq \PP[a\Gamma_n a^{-1} \cap V = \{e\}] \to 0,
\]
finishing the proof of the lemma.
\end{proof}

 Let $\boldsymbol{V} = (V_1, V_2, \ldots, V_k)$ denote a collection of bounded stochastic continuity sets for $\Upsilon$, and $[0, \boldsymbol{p}) = ([0, p_1), [0, p_2), \ldots, [0, p_k))$ a family of intervals in $[0,1]$. We denote by $\boldsymbol{V} \times [0, \boldsymbol{p}) = (V_1 \times [0, p_1), \ldots, V_k \times [0, p_k))$ the stochastic continuity set of $[0,1]^\Upsilon$.
 
 Let $C$ be a compact set large enough to contain $\bigcup_i V_i$. 
 
On the event from the lemma,
\[
    N_{\boldsymbol{V}}(\Theta_n(a\Gamma_n, \xi)) = N_{\boldsymbol{V}}([0,1]^{\Theta_n(a\Gamma_n)},
\]
where by $\Theta_n(a\Gamma_n)$ we simply mean $(\Theta_n(a\Gamma_n, \xi))$ with the labels erased. 

Therefore $\Theta_n(a\Gamma_n, \xi)$ converges weakly to $[0,1]^\Upsilon$, finishing the proof.
\end{proof}

\subsection{The rerooting groupoid for homogeneous spaces}\label{homogeneousspaces}

One must also investigate point processes on the homogeneous spaces of groups. The setup which we will consider is the action of $G$ on $X = G/K$, where $K \leq G$ is a \emph{compact} subgroup. This covers our principal case of interest, namely Riemannian symmetric spaces (such as $\Isom(\HH^d) \acts \HH^d$).

All of the point process theoretic definitions (such as thinnings and factor graphs) can be readily adapted to this context. There are some minor subtleties that must be addressed, however. Our aim is to define a cost notion for $G$-invariant point processes on $X$, and relate it to cost for $G$-invariant point processes on $G$ itself. We will show:

\begin{thm}\label{symmetricspacecostmax}
Assume that the Poisson point process on $X = G/K$ is free\footnote{Some assumption is required. For instance, if $K$ contains an element of the centre of $G$ then the action will not be free.} as a $G$ action. Then
\[
    \sup_\Pi \cost_X(\Pi) = \sup_\Pi \cost_G(\Pi),
\]
where the supremum is taken over the set of free point processes on $X$ and on $G$ respectively.
In particular, if $X$ has fixed price one then $G$ itself has fixed price one.
\end{thm}

\begin{remark}
The appeal of the above theorem is that it allows one to prove fixed price for a group by working on the symmetric space instead, where the geometry is more apparent. As will be evident in the proof, it suffices to prove fixed price for the Poisson point process on $X$, which is a concrete probabilistic object.
\end{remark}

Our starting point is to note that such spaces $X$ enjoy the properties we've been using\footnote{The limiting factor here is really the metric: a $G$-invariant Borel measure exists on $G/H$ in fairly great generality (it is an imposition on the modular functions of $G$ and $H$, but an invariant metric will only exist if $G/H$ is homeomorphic (in an appropriate sense) to a quotient $G'/H'$ with $H'$ compact in $G$, see \cite{ANANTHARAMANDELAROCHE2013546} for further details.}: namely, an invariant proper metric that makes $X$ an lcsc space and a $G$-invariant Borel measure $\lambda_X$ on $X$.

For the metric on $X$, one takes the Hausdorff metric:
\[
    d_X(aK, bK) = \max\{ \sup_{k_1 \in K} \inf_{k_2 \in K} d(ak_1, bk_2), \sup_{k_2 \in K} \inf_{k_1 \in K} d(ak_1, bk_2) \},
\]
and for the measure $\lambda_X$ one takes the pushforward $\pi_* \lambda$, where $\pi : G \to X$ is the quotient map $a \mapsto aK$.

We recall \emph{the mapping theorem} (See Section 2.3 of \cite{MR1207584} for further details):
\begin{thm}[Mapping theorem]\label{mappingtheorem}
Suppose $X$ is a standard Borel space with $\sigma$-finite measure $\lambda$, $\Pi$ is the Poisson point process with intensity measure $\lambda$, and $f : X \to Y$ is a measurable function, then $f(\Pi)$ is the Poisson point process on $Y$ with intensity measure $f_* \lambda$, assuming this measure has no atoms. 
\end{thm}

In other words, a map between the base spaces $X$ and $Y$ induces a map from the Poisson point process on $X$ to the Poisson point process on $Y$ (with some intensity measure). It is intuitive that this should lead to an inequality on costs, and we will detail how this works.

Our study splits into two cases, according to if $K$ is Haar null or not (for $G$'s Haar measure, of course). If $\lambda(K) > 0$, then by Steinhaus Theorem\footnote{It states that if $A \subset G$ is a subset of a locally compact group with positive Haar measure, then $AA^{-1}$ contains an identity neighbourhood.} we have that $K$ is open, and hence a compact clopen subgroup of $G$. It will then also have countable index. This situation occurs for instance in the study of Cayley-Abels graphs of totally disconnected locally compact groups. In this case one is essentially looking at Bernoulli percolation on the quotient space. We will instead focus on the case that $\lambda(K) = 0$.

One can consider $G$-invariant point processes $\Pi$ on $X$, which should be understood as random elements of $\MM(X)$ whose distribution is $G$-invariant. We may define the intensity of $\Pi$ as before:
\[
    \intensity(\Pi) = \frac{1}{\lambda_X(U)} \EE \abs{\Pi \cap U},
\]
where $U \subseteq X$ is any set of unit volume.

One can further consider notions of thinnings, partitionings, cost, and so on. We follow our previous strategy of studying these by looking at the associated groupoid. Let us write $0$ for the element $K \in X$, which we view as the root of the space. Accordingly we will define the space of rooted configurations in $X$ as
\[
    \MM_0(X) = \{ \omega \in \MM(X) \mid 0 \in \omega \}.
\]
Note that the orbit equivalence relation of $G \acts \MM(X)$ \emph{no longer} restricts to $\MM_0(X)$ to an equivalence relation with countable classes. The solution to this problem is to take a quotient:

\begin{prop}

Let $K$ act on $\MM_0(X)$ by left multiplication. Then the action is \emph{smooth}, that is, the space of orbits $K \backslash \MM_0(X)$ is a standard Borel space.

\end{prop}
It is a general fact that compact groups \emph{always} act smoothly on standard Borel spaces (see Proposition 2.1.12 of \cite{MR776417} and its corollary), but it is possible to give a direct proof in this case. The proof is very reminiscent of the section ``Extension to more general point processes'' of \cite{holroyd2003}.

\begin{proof}
We will construct a measurable function $F : \MMo(X) \to [0,1]$ with the property that $F(\omega) = F(\omega')$ if and only if $\omega' \in K\omega$. This is a characterisation of smoothness.

Fix a family ${U_n}$ of open subsets of $G$ with the property that it \emph{determines} elements of $\MMo(X)$ in the sense that
\[
    \omega = \omega' \Leftrightarrow \{n \in \NN \mid \omega \cap U_n \neq \empt \} = \{n \in \NN \mid \omega' \cap U_n \neq \empt \}
\]

Let $f : \{0,1\}^\NN \to [0,1]$ be any continuous order-preserving injection, and consider the map $F : \MMo(X) \to [0,1]$ given by
\[
F(\omega) = \inf_{k \in K} f((\1[U_n \cap k \cdot\omega \neq \empt])_n)
\]

Note that the component functions $\omega \mapsto \1[U_n \cap \omega \neq \empt]$ are lower semicontinuous, so the infimum is attained. 

The function is constant on $K$-orbits by definition, but by the separating nature of the family $\{U_n\}$ it also takes distinct values for points in distinct orbits.
\end{proof}

A Borel thinning $\theta : \MM(X) \to \MM(X)$ corresponds to a Borel subset $A \subseteq {K \backslash} \MM_0(X)$. Note that the latter can be identified with a subset of $\MM_0(X)$ which is $K$-invariant in the sense that $KA = A$, we will make such identifications freely.

To see why this $K$-invariance is required, consider the formula
\[
    \theta^A(\omega) := \{gK \in \omega \mid g^{-1}\omega \in A \}.
\]
For this to be well-defined, we need that the condition does not depend on our choice of coset representative $gK$. This is exactly asking for $K$-invariance of $A$.

In the same way one can see that Borel factor $[d]$-colourings correspond to Borel partitions of $K \backslash \MM_0(X)$ indexed by $[d]$, and so on.

If $\Pi$ is a $G$-invariant point process on $X$ with law $\mu$, then we may use the above to define its \emph{Palm measure} $\mu_0$ on $K \backslash \MM_0(X)$ as before:
\begin{align*}
    \mu_0(A) &:= \frac{\intensity \theta^A(\Pi)}{\intensity \Pi} \\
    &= \frac{1}{\intensity \Pi} \EE \left[ \# \{ gK \in U \mid g^{-1} \omega \in A \}\right], \text{ where } U \subseteq X \text{ has unit volume}.
\end{align*}

We equip $K \backslash \MM_0(X)$ with the following rerooting equivalence relation:
\[
    K\omega \sim_\Rel K\omega' \text{ if and only if } \exists aK \in \omega \text{ such that } K\omega' = K a^{-1} \omega.
\]
We can also define a groupoid structure. If one defines
\[
    \Marrow(X) = \{ (\omega, aK) \in \MM_0(X) \times X \mid aK \in \omega \},
\]
then there is a natural diagonal action of $K$ on $\Marrow(X)$. The quotient is again standard Borel, we denote it $K \backslash \Marrow(X)$.

Then $K \backslash \MM_0(X)$ and $K \backslash \Marrow(X)$ form the unit space and arrow space (respectively) of a groupoid, which we call \emph{the rerooting groupoid}. The source and target maps are defined as before
\begin{align*}
    &s, t : K \backslash \Marrow(X) \to K \backslash \MM_0(X),\\
    &s(K\omega, KaK) = K\omega,\\
    &t(K\omega, KaK) = Ka^{-1}\omega.
\end{align*}
A pair of arrows $(K\omega, KaK), (K\omega', KbK) \in K \backslash \Marrow(X)$ are deemed composable if $K\omega' = Ka^{-1}\omega$, in which case
\[
    (K\omega, KaK) \cdot (K\omega', KbK) := (K\omega, KabK).
\]

We can equip this groupoid with the Palm measure, resulting in a $r$-discrete pmp groupoid as before. 

\begin{defn}
    Let $\Pi$ be a $G$-invariant point process on $X$. Its \emph{groupoid cost} is
    \[
    \cost_X(\Pi) - 1 = \inf_{\mathscr{G}}\left\{  \frac{1}{2}\EE\left[ \sum_{x \in U \cap \Pi} \deg_x{\mathscr{G}(\Pi)} \right] \right\} - \intensity(\Pi),
\]
where $U$ is a set of unit volume in $X$ and the infimum is taken over all \emph{connected} equivariantly defined factor graphs of $\Pi$.
    
    The \emph{cost of $X$} is
    \[
        \cost(X) := \inf \{ \cost_X(\Pi) \mid \Pi \text{ is an invariant free point process on } X \}.
    \]
    
\end{defn}

Aside from replacing $\MMo(G)$ with $K \backslash \MMo(X)$, our earlier arguments apply verbatim and prove the following:
\begin{itemize}
    \item If $\Phi : \MM(X) \to \MM(X)$ is a $G$-equivariant factor map, then $\cost_X(\Pi) \leq \cost_X(\Phi(\Pi))$. In particular, $\cost_X(\Pi)$ only depends on the isomorphism class of $\Pi$,
    \item Every \emph{free} point process weakly factors onto its own IID,
    \item The Poisson point process \emph{on $X$} has maximal $\cost_X$ amongst free $X$ processes, assuming the Poisson point process is free.
\end{itemize}

\begin{remark}
In this level of generality, the Poisson point process on $X = G/K$ might not be free and thus must be assumed. For instance, let $G = \RR \times \RR/\ZZ$ and $K = \RR/\ZZ$. Then $K$ acts trivially on the quotient $X$, and thus \emph{no} $G$-invariant point process on $X$ is free (even their IID will not be free). These examples are rather contrived however.
\end{remark}

\begin{thm}\label{equalcost}
If $\Pi$ is a \emph{free} point process on $X$, then its $\cost_X$ is equal to its cost as a $G$-action.
\end{thm}

Recall from the introduction that the cost of a free pmp action of $G$ is defined by picking an isomorphic representation of the action as a point process, and taking the cost of that.

This theorem will employ a ``whittling'' construction. Note that we can view point processes on $X$ as random closed subsets of $G$ (which happen to be unions of cosets of a fixed compact subgroup). We are able to exploit freeness to \emph{deterministically} whittle this random closed subset to a point process:

\begin{prop}\label{deterministiclift}
If $\Pi$ is a free point process on $X$, then it admits a \emph{deterministic lift} to $G$: that is, there exists an invariant point process $\Upsilon = \Upsilon(\Pi)$ on $G$ such that
\begin{itemize}
    \item $\Upsilon \subset \Pi$ almost surely,
    \item $\Upsilon$ intersects each coset $aK$ \emph{at most} once, and
    \item $\pi(\Upsilon) = \Pi$.
\end{itemize}

In other words, we are able to select one element out of every coset $aK \in \Pi$ in an equivariant and deterministic way.

\end{prop}

\begin{proof}[Proof of Theorem \ref{equalcost}]
    Observe that the process $\Upsilon$ from Proposition \ref{deterministiclift} is isomorphic to $\Pi$, so $\cost(\Pi) = \cost(\Upsilon)$. We verify that $\cost(\Upsilon) = \cost_X(\Pi)$. Note that factor graphs of $\Pi$ and $\Upsilon$ can be bijectively identified, and so will have the same edge measures. Finally, they have the same intensity: choose $U \subseteq X$ with volume one, and observe that $\Pi \cap U$ and $\Upsilon \cap UK$ are in bijection, with $UK$ also having volume one.
\end{proof}

We will require a simple lemma, which essentially already appears in Lemma \ref{independentsetsexist} but we isolate for clarity. It works for point processes on $G$ and on $X$.

\begin{lem}\label{freeness}
   A point process $\Pi$ is free if and only if it admits a deterministic labelling by $[0,1]$ such that all of the labels are distinct (almost surely).
\end{lem}

\begin{proof}
    Clearly if such a labeling exists, then the process must be free. 
    
    For the converse, let $I: \MMo \to [0,1]$ be a Borel isomorphism. Define a labelling by
    \begin{align*}
        &L : \MM \to [0,1]^\MM \\
        &L(\omega) = \{ (x, I(x^{-1}\omega) \mid x \in \omega \}.
    \end{align*}
    Observe that two distinct points $x, y \in \omega$ receive the same label in $L(\omega)$ exactly when $I(x^{-1}\omega) = I(y^{-1}\omega)$, and so $xy^{-1}\omega = \omega$. If the process is free, then this never occurs, as desired.
    
    To run the proof on $X$, simply replace $\MMo$ by $L \backslash \MMo(X)$.
\end{proof}

\begin{proof}[Proof of Proposition \ref{deterministiclift}]

By virtue of being free, we may use Theorem \ref{minden} to fix an isomorphism of $\Pi$ with a point process $\Pi'$ on $G$. The desired process $\Upsilon$ will be the result of pushing the points of $\Pi'$ into $\Pi$.

Of course $\Pi'$ is itself a free process, so we may fix a deterministic labelling $L(\Pi')$ of its points \`{a} la Lemma \ref{freeness}.

Assign each coset $aK$ of $\Pi$ to a point of $x$ of $\Pi'$ in your preferred equivariant way. For instance, note that every such coset intersects some (finite but non-zero) number of Voronoi cells of $\Pi'$. For each $aK \in \Pi$, assign it to whichever of these cells has the germ with the highest label in $L(\Pi')$. We denote by $A_x$ the set of cosets in $\Pi$ that we assign to $x \in \Pi'$ in this way. 

Fix a Borel transversal $T \subset G$. Note that $xT$ is another Borel transversal for any $x \in G$, so $xT \cap aK$ selects the unique point representative of $aK$ with respect to this transversal.

Finally, set
\[
    \Upsilon = \bigcup_{\substack{x \in \Pi' \\ aK \in A_x}} xT \cap aK.
\]
This selects one representative from every coset in $\Pi$, and at every stage it was performed in an equivariant way, so is our desired invariant point process.
\end{proof}

\begin{proof}[Proof of Theorem \ref{symmetricspacecostmax}]

Let $\Pi$ denote the Poisson point process on $G$. Then by the mapping theorem (Theorem \ref{mappingtheorem}), the image $\Upsilon$ of $\Pi$ under the quotient map $G \to X$ is the Poisson point process on $X$. Since $\cost_G$ can only increase under factor maps, we have
\[
    \cost_G(\Pi) \leq \cost_G(\Upsilon).
\]
But the Poisson point process has maximal cost amongst free $G$-actions, so there is equality. By Theorem \ref{equalcost},
\[
    \cost_X(\Upsilon) = \cost_G(\Upsilon),
\]
and as discussed, the Poisson point process has maximal cost amongst all free point processes on $X$, finishing the proof.
\end{proof}

\begin{question}
It is natural to ask if $G$ and $X$ have the same \emph{infimal} cost as well. 
\end{question}

\subsection{Farber sequences in semisimple Lie groups}

The goal of this section is to to prove Theorem \ref{carderiextension} from the introduction, which we restate here:

\begin{thm*}
Let $G$ be a semisimple real Lie group and let $\Gamma_{n}$ be a Farber sequence of lattices in $G$. Then
\[
\limsup_{n\to\infty}\frac{d(\Gamma_n)-1}{\mathrm{vol}(G/\Gamma_n)}
\leq\mathrm{c}_{\mathrm{P}}(G)-1.
\]
\end{thm*}

There would be a more straightforward (and general) proof of the above if a more general form of cost monotonicity were true, however we are unable to prove (or disprove) the following statement: suppose $\Pi_n$ is a sequence of finite cost point processes that weakly converge to a random net $\Pi$. Is it true that
\[
    \limsup_{n \to \infty} \cost(\Pi_n) \leq \cost(\Pi) ?
\]

To prove Theorem \ref{carderiextension} we will use the geometric interpretation of being a Farber sequence -- specifically, see Corollary 3.3 of \cite{MR3664810}. In brief, it means that for all $r > 0$ the injectivity radius of a randomly chosen point of the quotient manifold $M_n = \Gamma_n \backslash X$ is larger than $r$ with high probability, where $X = G / K$ denotes the symmetric space of $G$. We will also heavily call upon the paper \cite{MR2863908}. Additionally, it will be assumed that the reader understands the proof of Theorem \ref{farbertheorem}.

\begin{proof}[Proof of Theorem \ref{carderiextension}]

First, let us handle the special case of $G = \PSL_2(\RR)$, for which the theorem is true but for fundamentally different reasons. In this case the $\Gamma_n$ being discussed are fundamental groups of finite volume hyperbolic surfaces, and we only require that $\covol(\Gamma_n)$ tends to infinity. This allows us to eliminate additive constants in the following:
\[
\lim_{n\to\infty}\frac{d(\Gamma_n)-1}{\mathrm{vol}(G/\Gamma_n)} = \lim_{n\to\infty}\frac{b_1(\Gamma_n)}{\mathrm{vol}(G/\Gamma_n)} = \lim_{n\to\infty}\frac{-\chi(G/\Gamma_n)}{\mathrm{vol}(G/\Gamma_n)} = \frac{1}{2\pi} = \beta_1(G) \leq \mathrm{c}_{\mathrm{P}}(G)-1
\]
where $\beta_1(G)$ is the first $L^2$-Betti number of $G$. We are using the Gauss-Bonnet formula above and Gaboriau's result that the first $L^2$-Betti number of a relation is dominated by its cost-1. 

Note that in \cite{conley2021one}, they prove (in particular) that (cross sections of) actions of $\PSL_2(\RR)$ are treeable and thus have fixed price equal to their $L^2$-betti number plus one. Thus we actually have equality all the way above. 

We now handle the other cases. Let us denote by $a\Gamma_n$ the lattice shift corresponding to $\Gamma_n < G$. Let us summarise the proof:
\begin{enumerate}
    \item We produce a sequence of uniformly separated factors $\Phi^n(a\Gamma_n)$ of the lattice shifts $G \acts G/\Gamma_n$. Note that by equivariance they must be of the form $\Phi^n(a\Gamma_n) = a\Gamma_n F_n$, and
    \[
    \cost(G \acts G/\Gamma_n) - 1 = \frac{d(\Gamma_n) - 1}{\text{vol}(G/\Gamma_n)} \leq \cost(\Phi^n(a\Gamma_n)) - 1
    \]
    as cost is monotone under factors.
    \item We show that $\Phi^n(a\Gamma_n)$ admits subsequential weak limits, and any such subsequential weak limit is a random net.
    \item As in the proof of Theorem \ref{farbertheorem}, we now use the periodic IID lattice shift and distribute randomness, replacing $\Phi^n(a\Gamma_n)$ by marked processes which converge to an IID-labelled random net.
    \item Using results of \cite{MR2863908}, our desired inequality follows from cost monotonicity.
\end{enumerate}

We will show that the distance-$R$ factor graph $\mathscr{D}_R$ is connected on $\Phi^n(a\Gamma_n) = a\Gamma_n F_n$. Observe that, by left-invariance of the metric, this is true if and only if it is connected on $\Gamma_n F_n$. Observe that this is finitely many \emph{right} cosets, that is, $\Gamma_n F_n \subset {\Gamma_n}\backslash G$. We will show that $\mathscr{D}_R$ is connected by appealing to properties of the further quotient ${\Gamma_n}\backslash G/K = {\Gamma_n}\backslash X$.

The essential result we use from \cite{MR2863908} is the following. As long as $G$ is not $\PSL(2,\RR)$, there exist constants $\e, \e' > 0$ and a sequence of \emph{connected} subsets $U_n \subset X$ such that
\begin{itemize}
    \item The projection $\Gamma_n U_n \subset \Gamma_n \backslash X$ contains the $\e$-thick part of $\Gamma_n \backslash X$, and
    \item The projection $\Gamma_n U_n \subset \Gamma_n \backslash X$ is contained in the $\e'$-thick part of $\Gamma_n \backslash X$. 
\end{itemize}
 Here $\e$ is defined in Lemma 2.3 of \cite{MR2863908} and $\e' = \e/(2\mu)$, where $\mu$ is defined after the proof of Lemma 2.4. Crucially, these constants only depend on $G$. In the paper, our $U_n$ is denoted by $\widetilde{\psi}_{\leq 0}$ and it is a level set with respect to a function inversely measuring the injectivity radius. 

\begin{claim}
There exists a sequence $\Phi^n(a\Gamma_n)$ of factors that are uniformly separated and any subsequential weak limit of them is a random net.
\end{claim}

\begin{proof}[Proof of claim]
    We choose $F_n \subset G$ following Corollary 2.13 of \cite{MR2863908}. We choose a maximal $\e'$-separated subset $E_n$ of $\Gamma_n U_n K \subset \Gamma_n \backslash X$. Then the union of $\e'$-balls around $E_n$ will cover $\Gamma_n U_n$ by maximality, hence, the set of points not covered by these balls lies in the $\e'$-thin part of $\Gamma_n \backslash X$. By the Farber condition, the density of these points tends to $0$ in $n$. This means that for any subsequential weak limit of the point processes $a\Gamma_n F_n$, the probability that the identity is distance more than $\e'$ from the root of $X$ is $0$. That is, the union of $\e'$-balls for any subsequential weak limit will equal $X$ a.s., that is, the weak limit will be a net. The slight difference is that we now need the same on $G$, not on $X$. In order to do that, we pick a coset representative (randomly or deterministically) with respect to $K$. This can only increase the $X$-distance but only by a bounded amount, so even on $G$, the same argument holds with worse constants.  
\end{proof}

Similar to before, we distribute randomness from the $\Gamma_n$-periodic IID processes to $\Phi^n(a\Gamma_n)$. Call the resulting process $\Pi^n$. By passing to a subsequential weak limit, we can assume that $\Pi^n$ weakly converges to some process $\Upsilon$. As before in Theorem \ref{farbertheorem}, the Farber condition ensures that $\Upsilon$ is in fact an IID labelled process (and in particular, its cost is at most the cost of the Poisson point process).

Our final task is to relate the cost of the $\Pi^n$ processes to the cost of $\Upsilon$. We write $\mu_0^n$ for the Palm measure of $\Pi^n$ and $\mu_0$ for the Palm measure of $\Upsilon$.

By the proof of Theorem \ref{costmonotonicity}, any factor graph which $\delta$-computes the cost of $\Upsilon$ contains a $\mu_0$-continuity factor graph $\mathscr{G}$ which is connected for $\Upsilon$. Thus
\[
    \limsup_{n \to \infty} \overrightarrow{\mu_0^n}(\mathscr{G}) \leq \overrightarrow{\mu_0}(\mathscr{G}).
\]
A priori, there is no reason to expect that $\mathscr{G}$ is connected for \emph{any} $\Pi_n$, however. But note that by construction the graphing $\mathscr{G}$ has the property that for large enough $R > 0$ there exists a constant $N$ such that $\mathscr{G}^N(\omega) \supseteq \mathscr{D}_R(\omega)$ for all $\e'$-separated configurations $\omega$, where $\mathscr{D}_R$ denotes the distance-$R$ factor graph as usual. In particular, $\mathscr{G}$ is connected for the lattice shift factors $a\Gamma_n F_n$, as they are coarsely connected: by left-invariance of the metric, we may simply consider $\Gamma_n F_n$, and recall that its image in $X$ lies in the connected subset $U_n$. Now for any pair of points $x$ and $y$ in $\Gamma_n F_n$, take a path between their images $xK$ and $yK$ lying within $U_n$, and note that it induces a coarse path (with bounded jumps) between $x$ and $y$ themselves. Thus
\[
\limsup_{n \to \infty} \frac{d(\Gamma_n) - 1}{\text{vol}(G/\Gamma_n)} \leq \cost(\Pi^n) - 1
\]
as desired.
\end{proof}

\appendix

\section{Metric properties of configuration spaces and weak convergence}\label{metricproperties}

In this appendix, we will discuss the necessary technical background to understand the rest of this paper. The aim is that a reader unfamiliar with the theory of point processes will be able to read this section and have the core ideas of what's going on (if not the finer technical details). No originality is claimed for this material, and so explicit references are given. 

The following fact is the most basic requirement for a well-behaved probability theory:

\begin{thm}[See Theorem A2.6.III of \cite{vere2003introduction}]
	If $X$ is a complete and separable metric space, then $\MM(X)$ is a Borel subset of a complete and separable metric space $\mathcal{M}^{\#}(X)$, and is thus a standard Borel space.
\end{thm}

Note that configurations $\omega \in X$ can be viewed as measures on $X$, by defining $\omega(A) = \abs{\omega \cap A}$. So configurations form particular examples of \emph{locally finite measures} on $X$, and $\mathcal{M}^{\#}(X)$ will be the space of such measures. In this language, a point process is a particular example of a \emph{random measure}. Probabilists are interested in other examples of random measures\footnote{And, for that matter, non-invariant point processes.}, and have thus developed a framework suitable to handle all their cases of interest. We adopt their framework with small notational modifications.

We assume the reader is at least passingly familiar with weak converge of measures on metric spaces. Recall:

\begin{defn}

Let $\mathcal{M}(X)$ denote the space of \emph{totally finite} measures $\eta$ on $X$, that is, those with $\eta(X) < \infty$. 
    
    The \emph{Prokhorov metric} $\dprok$ on $\mathcal{M}(X)$ is
    \[
        \dprok(\eta, \eta') = \inf\{ \e \geq 0 \mid \text{for all Borel } A \subseteq X, \eta(A) \leq \eta'(A^\e) + \e \text{ and } \eta'(A) \leq \eta(A^\e) + \e \},
    \]
    where $A^\e$ is the \emph{$\e$-halo} of $A$, that is,
    \[
        A^\e = \{ x \in X \mid d(x, A) < \e \}.
    \]
    
    If $\eta$ is a totally finite measure on $X$, then a \emph{$\eta$-continuity set} is a subset $A \subseteq X$ with the property that $\eta(\partial A) = 0$, where $\partial A$ denotes the \emph{topological boundary} of $A$.
    
    A sequence of totally finite measures $\eta_n$ \emph{weakly converges} to $\eta$ if either of the following conditions hold:
    
    \begin{description}
        \item[WC1] for all continuous and bounded functions $f : X \to \RR$, 
        \[
            \lim_{n \to \infty} \int_X f(x) d\eta_n(x) = \int_X f(x) d\eta(x).
        \]
        \item[WC2] for all $\eta$-continuity sets $A \subseteq X$,
        \[
            \lim_{n \to \infty} \eta_n(A) = \eta(A).
        \]
    \end{description}
    
\end{defn}

\begin{remark}

    The Prokhorov metric metrises this convergence notion (that is, $\eta_n$ converges weakly to $\eta$ if and only if $\dprok(\eta_n, \eta)$ converges to zero).
    
    The equivalence of \textbf{WC1} and \textbf{WC2} is usually referred to as \emph{the Portmanteau theorem}. 
    
    The definition involving continuity sets will have preeminence for us. To explain the name: note that the indicator function $\1_A : X \to \{0, 1\}$ is continuous $\eta$ almost everywhere if and only if $A$ is an $\eta$-continuity set.
\end{remark}

We often make use of the following well-known fact:

\begin{lem}\label{continuity}

If $\Phi : X \to Y$ is a continuous map of metric spaces, then $\Phi$ preserves weak limits: if $\mu_n$ is a sequence of Borel probability measures on $X$ weakly converging to $\mu$, then $\Phi_* \mu_n$ weakly converges to $\Phi_* \mu$. 

Moreover, the same is true if $\Phi$ is merely continuous $\mu$ almost everywhere.

\end{lem}

\begin{proof}

The first statement is immediate: if $f : Y \to \RR$ is a continuous and bounded function, then $f \circ \Phi : X \to \RR$ is continuous and bounded as well, so
\[
    \lim_{n \to \infty} \int_Y f d\Phi_* \mu_n = \lim_{n \to \infty} \int_X f \circ \Phi d\mu_n = \int_X f \circ \Phi d\mu = \int_Y f d\Phi_* \mu.
\]
For the second statement we work with the definition involving continuity sets. Let $A \subseteq Y$ be a $\Phi_* \mu$ continuity set, that is, assume $\mu(\Phi^{-1}(\partial A)) = 0$. Let $D_\Phi = \{ x \in X \mid \Phi \text{ is discontinuous at } x\}$ denote the discontinuity set of $\Phi$. One can show that for any $A \subseteq Y$ that $\partial \Phi^{-1}(A) \subseteq \Phi^{-1}(\partial A) \cup D_\Phi$, and so
\[
    \mu(\partial \Phi^{-1}( A)) \leq \mu(\Phi^{-1}(\partial A)) + \mu(D_\Phi) = 0,
\]
that is, $\Phi^{-1}(A)$ is a $\mu$-continuity set. Therefore
\[
    \lim_{n \to \infty} \Phi_* \mu_n(A) = \lim_{n \to \infty} \mu_n(\Phi^{-1}(A)) = \mu(\Phi^{-1}(A)) = \Phi_* \mu(A),
\]
as desired.
\end{proof}

\begin{defn}

Let $\mathcal{M}^{\#}(X)$ denote the space of \emph{boundedly finite} measures, that is, those Borel measures $\eta$ on $X$ that are finite on metrically bounded subsets of $X$.

Fix a basepoint $x_0 \in X$. Let
\[
    \eta^{(r)}(A) := \eta(A \cap B(x_0; r))
\]
denote the restriction of a boundedly finite measure $\eta$ to the $r$-ball about $x_0$. Note that $\eta^{(r)}$ is therefore an element of $\mathcal{M}(X)$.

We now define a metric $d^{\#}$ on $\mathcal{M}^{\#}(X)$:

\[
    d^{\#}(\eta, \eta') = \int_0^\infty e^{-r} \frac{\dprok(\eta^{(r)}, \eta'^{(r)})}{1 + \dprok(\eta^{(r)}, \eta'^{(r)})} dr.
\]

A sequence of boundedly finite measures $\eta_n$ \emph{weak-$\#$ converges} to $\eta$ if any of the following conditions hold:

\begin{description}
    \item[WHC1] for all continuous and bounded functions $f: X \to \RR$ which vanish outside a bounded set,
    \[
    \lim_{n \to \infty} \int_X f(x) d\eta_n(x) = \int_X f(x) d\eta(x).
    \]
    \item[WHC2] for all bounded $\eta$-continuity sets $A \subseteq X$,
        \[
            \lim_{n \to \infty} \eta_n(A) = \eta(A).
        \]
    \item[WHC3] there exists a sequence $r_k$ of radii increasing to infinity such that for every $k \in \NN$
    \[
        \eta_n^{(r_k)} \text{ converges weakly to } \eta^{(r_k)}.
    \]

\end{description}

\end{defn}

\begin{remark}\label{metricremark}

The space we've defined is obviously \emph{extremely} metrically dependent (recall that every metric is topologically equivalent to a bounded metric). However, our case of interest is proper left-invariant metrics on locally compact groups, which are all coarsely equivalent and thus have a well-defined notion of metrically boundedness.

Defining the metric required us to fix an arbitrary base point $x_0$. If one chooses a different basepoint $x'_0$, then the resulting metrics will be bilipschitz equivalent.

In case $X$ is locally compact, then weak-$\#$ convergence is equivalent to vague convergence.

\end{remark}

\begin{thm}
    
    The space $\mathcal{M}^{\#}(X)$ equipped with the $d^{\#}$ metric is complete and separable. Its Borel structure is exactly such that the mass measuring\footnote{Earlier we called these ``point counting'' functions, because that's a more suitable name when the measure is atomic.} functions $N_A : X \to \NN_0 \cup \infty$ given by $\eta \mapsto \eta(A)$ are measurable, where $A$ is an arbitrary Borel subset of $X$.
    
\end{thm}

\begin{remark}

The Borel structure on $\mathcal{M}^{\#}(X)$ can be generated by an even smaller collection of mass measuring functions: one only needs to look at $N_A$ where $A$ ranges over a semiring of bounded Borel sets that generate the Borel structure on $X$.

\end{remark}

We will require the following more explicit explanation of what weak-$\#$ convergence is:

\begin{defn}

Let $\omega \in \MM(X)$ be a configuration. We call another configuration $\omega' \in \MM(X)$ a \emph{$(\e, R)$-wobble} of $\omega$ (where $\e, R > 0$ are some parameters) if $\omega^{(R)}$ is in bijection with ${\omega'}^{(R)}$, and moreover this bijection $\sigma : \omega^{(R)} \to {\omega'}^{(R)}$ can be chosen in such a way that $d(x, \sigma(x)) < \e$ for all $x \in \omega^{(R)}$.

\end{defn}

One direction of the following lemma is immediate, the converse is less elementary and can be found in \cite{vere2003introduction} as Proposition A2.6.II:

\begin{lem}

A sequence of configurations $\omega_n$ converges to $\omega$ with respect to $d^{\#}$ if and only if there are sequences $\e_n \to 0$ and $R_n \to \infty$ such that each $\omega_n$ is a $(\e_n, R_n)$-wobble of $\omega$.
\end{lem}

We can now discuss \emph{weak convergence} of point processes. View $\MM(X)$ as a subset of $\mathcal{M}^{\#}(X)$ equipped with the $d^{\#}$ metric, and recall that a point process is a probability measure on $\MM(X)$. This is what we mean by a sequence of point processes weakly converging.

Note that the weak limit of a sequence of point processes $\mu_n$ will (a priori) be a probability measure on $\mathcal{M}^{\#}(X)$, \emph{not} on $\MM(X)$. That is, a point process might converge to a random measure which is not a point process. It's easy to see that the only thing that can go wrong is mass accumulation: the limit measure will be a random atomic measure, but some atoms might have mass larger than one.
\begin{defn}\label{countingmeasuredefn}
A \emph{counting measure} on $X$ is a measure $\eta$ with $\eta(A) \in \NN_0$ for all bounded Borel subsets $A \subseteq X$. A \emph{simple} counting measure is a measure $\eta$ with $\eta(\{x\}) = 0 \text{ or } 1$ for all $x \in X$.

If $\eta$ is a counting measure, then its \emph{support} is $\support(\eta) = \{ x \in X \mid \eta(\{x\}) > 0 \}$. That is, $\support(\eta)$ is $\eta$ with the multiplicities removed.

\end{defn}

\begin{example}
    
Let $\{X_k\}_{k \in \ZZ}$ denote an IID sequence of uniform $[0,1]$ random variables. Consider the following sequence of point processes:

\[
    \Pi_n = \ZZ \cup \left\{ k + \frac{X_k}{n} \mid k \in \ZZ \right\}.
\]

In words: take two copies of $\ZZ$, where you wobble all the points of one copy by smaller and smaller amounts (this is not a point process proper in our sense, as it is not invariant, but one can take a uniform $[0,1]$ shift of $\Pi_n$ if one insists).

Then $\Pi_n$ weakly converges to the deterministic measure $\mu$ given by $\mu(A) = 2 \abs{A \cap \ZZ}$.
\end{example}

\begin{remark}
In this language, what we've been calling point processes are \emph{random simple counting measures}, and the comment above states that the weak limit of random simple counting measures, if it exists, will be a possibly non-simple random counting measure.

In the literature one sometimes sees random counting measures referred to as ``point processes'', and random simple counting measures as ``simple point processes''. 

\end{remark}

\begin{defn}

Let $(X, d)$ be a csms, and $(\mu_n)$ a sequence of Borel probability measures on $X$. The sequence is \emph{uniformly tight} if for every $\e > 0$ there exists a compact set $K \subseteq X$ such that $\mu_n(X \setminus K) < \e$ for all $n \in \NN$.

\end{defn}

Recall from Prokhorov's theorem that a sequence $(\mu_n)$ is uniformly tight if and only if its closure $\overline{(\mu_n)}$ is compact. To apply this to point processes, we need to know about compact sets in $\mathcal{M}^{\#}(X)$, and it is not at all evident what compact sets are given our definition of the metric. It is thus the following more explicit form of uniform tightness that we will use:

\begin{thm}[See Proposition 11.1.VI of \cite{daley2007introduction}]

Suppose $X$ is a locally compact\footnote{There is a slightly more complicated version of the theorem for general Polish spaces, but we will not use it, so do not state it.} csms. A sequence of point processes $(\Pi_n)$ on $X$ is uniformly tight if and only if for every closed ball $B \subseteq X$ and any $\e > 0$ there exists an $M > 0$ such that
\[
    \PP[ N_B \Pi_n > M] < \e \text{ for all } n \in \NN.
\]

\end{thm}

\begin{lem}\label{hardcorecompact}

Let $(X,d)$ be a locally compact csms, and
\[
    H_\delta = \{ \omega \in \MM(X) \mid d(x, y) \geq \delta \text{ for all distinct } x, y \in \omega \}
\]
denote the space of \emph{$\delta$-uniformly-separated} configurations. Then $H_\delta$ is compact in $\MM(X)$.

\end{lem}
Probabilists often refer to such configurations as \emph{hard-core}, hence our choice of letter.

The previous lemma is proved using the following basic fact:

\begin{lem}

Let $(X, d)$ denote a compact metric space. Then for all $\delta > 0$ there exists some $C > 0$ such that $\abs{A} \leq C$ for any $\delta$-separated subset $A \subseteq X$.

\end{lem}

The above discussion has been rather abstract. We now outline an equivalent interpretation of weak convergence that will be much more useful in certain applications.

\begin{defn}
    Let $\Pi$ be a point process with law $\mu$. A \emph{stochastic continuity set} of $\Pi$ is a Borel subset $V \subseteq G$ of the ambient space such that $\PP[\Pi \cap \partial V \neq \empt] = 0$. Equivalently, it is a subset such that its point counting function $N_V : \MM \to \NN_0 \cup \{\infty\}$ is continuous $\mu$ almost everywhere.
    
    Let $\boldsymbol{V} = (V_1, V_2, \ldots, V_k)$ denote a collection of stochastic continuity sets for $\Pi$.
    
    The \emph{finite dimensional distributions} of $\Pi$ are the random vectors
    \[
        N_{\boldsymbol{V}}(\Pi) = (N_{V_1} \Pi, N_{V_2} \Pi, \ldots, N_{V_k} \Pi ),
    \]
    where $\boldsymbol{V}$ runs over all possible collections of stochastic continuity sets.
\end{defn}

\begin{remark}

The sets
\[
    \{ \omega \in \MM \mid N_{\boldsymbol{V}}(\omega) = \boldsymbol{\alpha} \}, \text{ where } \boldsymbol{\alpha} = (\alpha_1, \alpha_2, \ldots, \alpha_k) \in \NN_0^k
\]
should be thought of as analogous to the \emph{cylinder sets} in the space $\{0, 1\}^\Gamma$, where $\Gamma$ is a discrete group.
\end{remark}

Note that if $\boldsymbol{V}$ is a family of stochastic continuity sets for $\Pi$, then $N_{\boldsymbol{V}}$ is continuous $\mu$ almost everywhere. Thus by the earlier fact on weak limits and continuous functions, we see that weak convergence of point processes implies weak convergence of the finite dimensional distributions. The surprising fact is that the converse is true:

\begin{thm}[See Theorem 11.1.VII of \cite{daley2007introduction}]

A sequence $\Pi_n$ of point processes weakly converges to $\Pi$ if and only if for all collections $\boldsymbol{V} = (V_1, V_2, \ldots, V_k)$ of stochastic continuity sets of $\Pi$, the finite dimensional distributions $N_{\boldsymbol{V}}(\Pi_n)$ weakly converge to $N_{\boldsymbol{V}}(\Pi)$.

\end{thm}

For this to make any sense at all, it must be the case that the finite dimensional distributions \emph{determine} a point process. That is, if two point processes $\Pi$ and $\Pi'$ have $N_{\boldsymbol{V}} (\Pi) \overset{d}{=} N_{\boldsymbol{V}}(\Pi')$ for all collections of stochastic continuity sets $\boldsymbol{V}$, then $\Pi \overset{d}{=} \Pi'$. This is proved using the following lemma, which states that there is an abundance of continuity sets:

\begin{lem}\label{continuitysets}
    Let $\Pi$ be a point process with law $\mu$. Then
    \begin{itemize}
        \item for all $g \in G$, there are at most countably many $r > 0$ such that the open ball $B_G(g, r)$ is \emph{not} a stochastic continuity set of $\Pi$, and
        \item for all $\omega \in \MM$, there are at most countably many $r > 0$ such that the open ball $B_\MM(\omega, r)$ is \emph{not} a $\mu$ continuity set.
    \end{itemize}
    In particular, both $G$ and $\MM$ admit \emph{topological bases} consisting of $\mu$ stochastic continuity sets / $\mu$ continuity sets (respectively).
\end{lem}

\begin{proof}

The method is the same in both cases, so we only write the proof for the first statement. The idea of the proof is that that there cannot be so many stochastic continuity sets, else local finiteness will be contradicted.
It is enough to prove that for every $r > 0$ and $\e > 0$ there exists only finitely many $r_1, r_2, \ldots$ in $(0,r)$ such that
\[
    \PP[\Pi \cap \partial B(0, r_i) \neq \empt] > \e \text{ for all } i.
\]

Suppose not. That is, suppose we have $r, \e > 0$ and infinitely many $\{r_n\} \subset (0, r)$ satisfying the above equation. Then
\[
\e \leq \limsup_{n \to \infty} \PP[\Pi \cap \partial B(0, r_n) \neq \empt] \leq \PP\left[ \limsup_{n \to \infty} \left(\Pi \cap \partial B(0, r_n) \neq \empt \right) \right].
\]
Recall that the $\limsup$ of a sequence of events is the event that they occur infinitely often. So we see
\[
    \{ \Pi \cap \partial B(0, r_n) \neq \empt \text{ for infinitely many } n\} \subseteq \{ \abs{\Pi \cap B(0, r)} = \infty \}.
\]
We've shown that with positive probability, $\Pi$ has infinitely many points in $B(0,r)$, a contradiction by local finiteness.
\end{proof}

Note that the continuity sets form an algebra, and the cylinder sets $\{ \omega \in \MM \mid N_{\boldsymbol{V}}(\omega) = \boldsymbol{\alpha} \}$ are continuity sets when $\boldsymbol{V}$ is a collection of stochastic continuity sets. As a measure is determined by its values on any algebra that generates the Borel sigma algebra, we therefore see that point processes are determined by their finite dimensional distributions. With a bit more work (see \cite{vere2003introduction} Proposition A2.3.IV and \cite{daley2007introduction} Corollary 11.1.III, Theorem 11.I.VII), one can prove:

\begin{lem}\label{determiningclass}
   Let $\Pi$ be a point process. Then there exists a \emph{countable} family $\{V_i\}_{i \in \NN}$ of \emph{metrically bounded} and disjoint Borel subsets $V_i \subseteq G$ such that $\Pi_n$ weakly converges to $\Pi$ if and only if $N_{\boldsymbol{V}} \Pi_n$ weakly converges to $N_{\boldsymbol{V}} \Pi$ were $\boldsymbol{V}$ ranges over all finite subcollections of $\{V_i\}$.
   
   In particular, weak convergence can be verified by a \emph{countable} collection of statements, each of which only requires one to observe the process in compact windows.
   
\end{lem}

The following lemma is a simpler case of Exercise 13.2.2 in \cite{daley2007introduction}, and is presumably known with a more elegant proof. The technique will be used for a later proof, so we include it.

\begin{prop}\label{palmconvergence}

Suppose $\mu^n$ is a sequence of finite intensity point processes that weakly converge to a finite intensity process $\mu$, and $\intensity \mu_n$ converges to $\intensity \mu$. Then the Palm measures $\mu^n_0$ weakly converge to $\mu_0$.

\end{prop}

\begin{proof}
    
    Let $A \subseteq \MMo$ be a $\mu_0$-continuity set, and $U \subseteq G$ a stochastic $\mu$ continuity set. Recall that
    \[
        \mu_0(A) = \frac{1}{\intensity \mu \lambda(U)} \EE_\mu \left[ \#\{g \in U \mid g^{-1}\omega \in A \} \right]
    \]
    \begin{claim}
    For every $k \in \NN$, the function $\omega \mapsto \min\{\#\{g \in U \mid g^{-1}\omega \in A \}, k\}$ is continuous $\mu$ almost everywhere.
    \end{claim}
    This function can only be discontinuous on the boundary of
    \[
        A^{(l)} = \{ \omega \in \MM \mid \#\{g \in U \mid g^{-1}\omega \in A \} = l\}.
    \]
    
    Observe that
    \begin{align*}
        &\mu_0(\partial A) = 0 && \text{By assumption, so} \\
        &\mu_0([\partial A]) = 0 && \text{As saturations of null sets are null in a countable groupoid, } \\
        &\mu(G \partial A) = 0 && \text{By Proposition \ref{transferprinciple}}.
    \end{align*}
    We show that $\partial A^{(l)} \cap \{N_{\partial U} \omega = 0\} \subseteq G\partial A$ for all $l \geq 1$, establishing the claim.

    Suppose $\omega \in \partial A^{(l)} \cap \{N_{\partial U} \omega = 0\}$. Then we can find two sequences $\omega_n, \omega'_n$ both converging to $\omega$ such that $\omega_n \in A^{(l)}$ and $\omega'_n \not\in A^{(l)}$ for every $n \in \NN$. We take these to be $(\e_n, R_n)$-wobbles of $\omega$.
    
    We see that (for large $n$) the configurations $\omega, \omega_n, \omega'_n$ are all approximately equal inside $U$. See Figure \ref{convergence}.
    
    \begin{figure}[h]\label{convergence}
\includegraphics[scale=0.4]{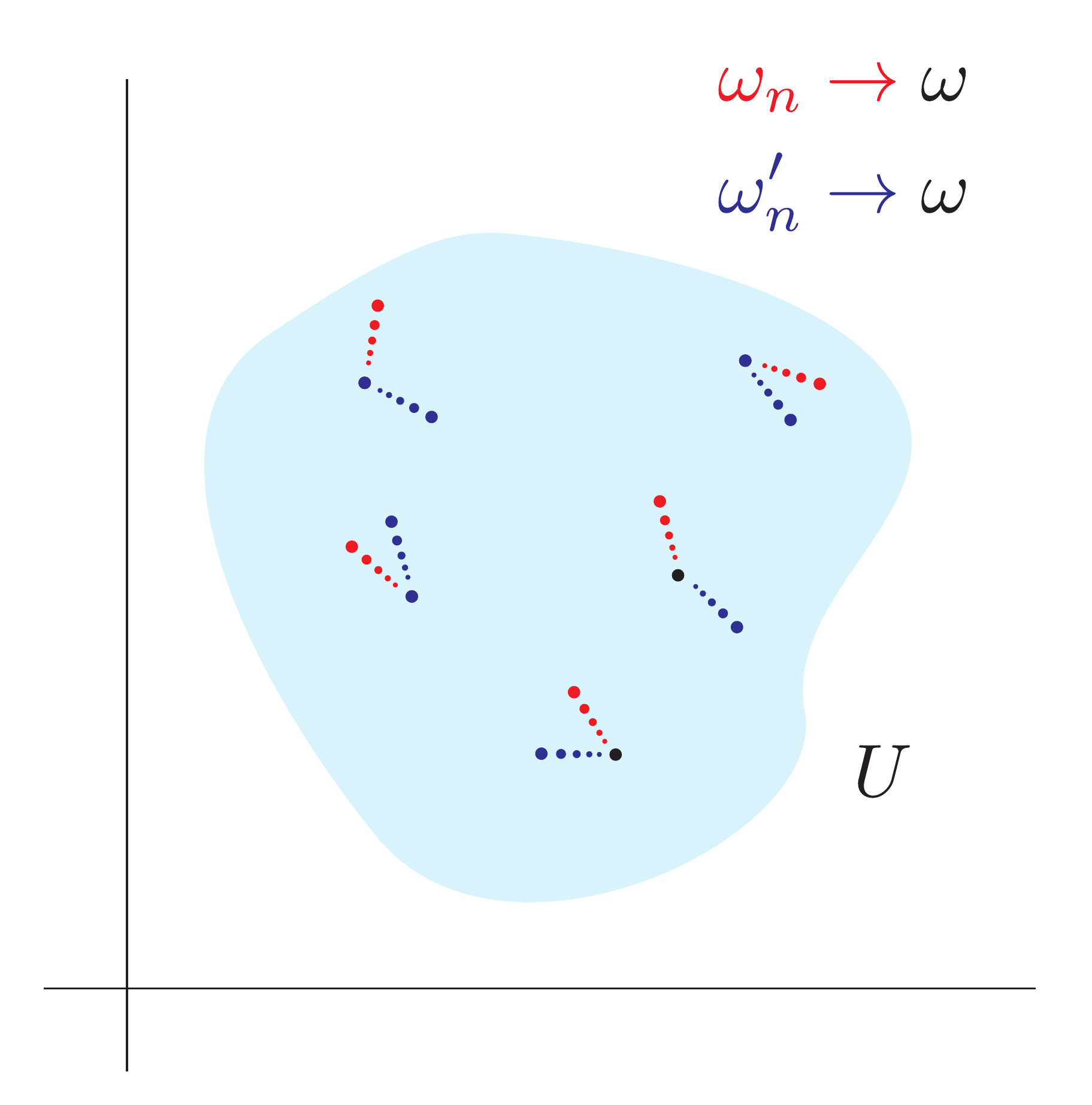}
\centering
\end{figure}

   We will refer to the points $g \in U \cap \omega$ such that $g^{-1}\omega \in A$ as \emph{$A$-points of $\omega$} and likewise for $\omega_n$ and $\omega'_n$.
   
   Now for every (large) $n$, the number of $A$ points of $\omega_n$ in $U$ is $l$, and the number of $A$ points of $\omega'_n$ in $U$ is some (bounded) number other than $l$. Since the configurations are a small wobble of $\omega$ then, we can find $g_n \in \omega \cap U$ such that the corresponding points $x_n$ of $\omega_n$ and $y_n$ of $\omega'_n$ are an $A$ point and \emph{not} an $A$ point, respectively.
   
   As $g_n$ ranges over a finite set $\omega \cap U$, we can choose $g \in \omega \cap U$ and a subsequence $(n_k)$ such that $g_{n_k} = g$ for every $k \in \NN$. Then
   \[
        x_{n_k}^{-1} \omega_{n_k} \to g^{-1} \omega, \text{ and } y_{n_k}^{-1} \omega'_{n_k} \to g^{-1} \omega,
   \]
   which witnesses that $g^{-1} \omega \in \partial A$, so $\omega \in G \partial A$, as desired.
\end{proof}

\bibliographystyle{alpha} 
\bibliography{refs} 

\end{document}